\documentclass[11pt,reqno]{amsart}
\usepackage{}
\usepackage{amssymb}
\usepackage{mathrsfs}
\usepackage{amsfonts}
\usepackage{amsfonts,amssymb,amsmath,amsthm}
\usepackage{url}
\usepackage{enumerate}
\usepackage[pdftex,bookmarksnumbered]{hyperref}
\usepackage{graphicx,xcolor}
\newcommand{\ds}{\displaystyle}
\newtheorem{theorem}{Theorem}[section]
\newtheorem{lemma}[theorem]{Lemma}
\newtheorem{proposition}[theorem]{Proposition}
\newtheorem{corollary}[theorem]{Corollary}
\theoremstyle{definition}
\newtheorem{definition}[theorem]{Definition}
\newtheorem{remark}[theorem]{Remark}
\numberwithin{equation}{section}

\newtheorem{example}{Example}[section]

\usepackage[left=2.5cm, right=2.5cm, top=2.3cm, bottom=2.3cm]{geometry}

\usepackage{graphicx}%
\usepackage{color}
\usepackage{cite}
\usepackage{hyperref}
\hypersetup{
	colorlinks = true,%
	citecolor = blue,
	filecolor=red,%
	linkcolor = [rgb]{0.65,0.0,0.0},%
	anchorcolor = red,
	pagecolor = red,
	urlcolor= [rgb]{0.65,0.0,0.0}
	linktocpage=true,
	pdfpagelabels=true,
	bookmarksnumbered=true,
}

\DeclareMathOperator{\id}{Id}
\DeclareMathOperator{\diam}{Diam}
\DeclareMathOperator{\vol}{vol}
\DeclareMathOperator{\Ca}{Cap}

 \DeclareMathOperator{\Cov}{Cov}

\DeclareMathOperator{\dd}{d\mathfrak{m}}

\DeclareMathOperator{\ddd}{d}

\DeclareMathOperator{\supp}{Spt}

\DeclareMathOperator{\suppor}{supp}

\DeclareMathOperator{\Do}{Dou}

\DeclareMathOperator{\CD}{CD}

\DeclareMathOperator{\CDD}{CDD}
\DeclareMathOperator{\pCDD}{pCDD}

\DeclareMathOperator{\dis}{dis}

\DeclareMathOperator{\Leb}{\mathscr{L}}

\author{Alexandru Krist\'aly}
\address{Department of Economics\\
	Babe\c s-Bolyai University\\
	400591 Cluj-Napoca, Romania \&  Institute of Applied Mathematics\\
 \'Obuda University\\
 1034 Budapest, Hungary}
  \email{alexandru.kristaly@ubbcluj.ro; kristaly.alexandru@nik.uni-obuda.hu}

\author{Wei Zhao}
\address{
Department of Mathematics\\
East China University of Science and Technology\\
200237 Shanghai, China}
\email{szhao\underline{ }wei@yahoo.com}

\begin{document}

\title[]{On the geometry of irreversible metric-measure spaces: Convergence, Stability and Analytic aspects}

\begin{abstract}
The paper is devoted to the study of Gromov-Hausdorff convergence and stability of irreversible metric-measure spaces, both in the compact and noncompact cases. While the compact setting is mostly similar to the reversible case developed by J. Lott, K.-T. Sturm  and C. Villani, the noncompact case provides various surprising phenomena. Since the reversibility of noncompact irreversible spaces might be infinite, it is motivated to introduce  a suitable nondecreasing function that bounds the reversibility of larger and larger balls. By this approach, we are able to prove satisfactory convergence/stability results in a suitable  -- reversibility depending --  Gromov-Hausdorff topology.   A wide class of irreversible spaces is provided by Finsler manifolds, which serve to construct various model examples by pointing out genuine differences between the reversible and irreversible settings. We conclude the paper by proving various geometric and functional inequalities (as Brunn-Minkowski, Bishop-Gromov, log-Sobolev and Lichnerowicz inequalities) on  irreversible structures.
\end{abstract}
\maketitle

\tableofcontents

\section{Introduction} \label{sect1}

Irreversible metrics often occur in nature; a prominent example is the Matsumoto  metric (see \cite{Matsumoto}) describing the law of walking on a mountain slope under the action of gravity. Another important class of irreversible metrics is the Randers metric appearing as the solution of the Zermelo navitation problem, see Bao, Robles and Shen \cite{BRS}. A particular representation of the latter metric is the perturbation of the usual Klein metric over the $n(\geq3)$-dimensional Euclidean unit ball $\mathbb{B}^n=\{x\in \mathbb R^n:\|x\|<1\}$, called  the \textit{Funk metric} (see e.g. Shen \cite{Sh1}), defined as $F:\mathbb{B}^n\times
\mathbb R^{n}\to \mathbb R$ by
\begin{equation}\label{Funckmeatirc}
	F(x,y)=\frac{\sqrt{\|y\|^2-(\|x\|^2\|y\|^2-\langle
			x,y\rangle^2)}}{1-\|x\|^2}+\frac{\langle x,y\rangle}{1-\|x\|^2},\ x\in
	\mathbb{B}^n,\ y\in T_x\mathbb{B}^n=\mathbb R^n,
\end{equation}
where $\|\cdot\|$ and
$\langle\cdot, \cdot\rangle$ denote the $n$-dimensional Euclidean
norm and inner product, respectively. The distance function associated
to $F$ is
$$
d_{F}(x_1,x_2)=\ln\frac{\sqrt{\|x_1-x_2\|^2-(\|x_1\|^2\|x_2\|^2-\langle x_1,x_2\rangle^2)}-\langle x_1,x_2-x_1\rangle}{\sqrt{\|x_1-x_2\|^2-(\|x_1\|^2\|x_2\|^2-\langle x_1,x_2\rangle^2)}-\langle x_2,x_2-x_1\rangle},\ x_1,x_2\in \mathbb{B}^n.
$$
It is immediate that usually $d_F(x_1,x_2)\neq d_F(x_2,x_1)$ and particularly,
\[
\lim_{\|x\|\rightarrow1^-}d_F(\mathbf{0},x)=\infty,\ \lim_{\|x\|\rightarrow1^-} d_F(x,\mathbf{0})=\log2.
\]
  We observe that $d_F$ is non-negative and verifies the triangle inequality, but not the symmetry; in fact,  $(\mathbb{B}^n,d_F)$  is an object which serves as a model structure where the symmetry fails.

In general, a pair $(X,d)$ is called an {\it irreversible metric space} if $X$ is a nonempty set and the metric $d: X\times X\to \mathbb R$ verifies  for any $x,y,z\in X$  the following properties:
\begin{itemize}
	\item[(i)] non-negativity: $d(x,y)\geq 0$ with equality if and only if $x=y$;
	
	\item[(ii)] triangle inequality: $d(x,z)\leq d(x,y)+d(y,z)$.
\end{itemize}

In addition, if for every $x,y\in X$ one has the
\begin{itemize}
	\item[(iii)] symmetry: $d(x,y)=d(y,x)$,
\end{itemize}
the pair $(X,d)$ is called a {\it reversible metric space}.

The geometry of reversible metric spaces has been widely studied.  Especially, Gromov \cite{Gromov} introduced the so-called {\it Gromov-Hausdorff  topology} to study the convergence of such spaces, which plays an important role in many fields of mathematics.
During the last two decades deep studies appeared by describing the geometry of reversible metric-measure spaces.
In particular, based on the theory of optimal transport, Lott and Villani \cite{LV}, and Sturm \cite{Sturm-1, Sturm-2}  introduced independently the synthetic notion of Ricci curvature on reversible metric-measure spaces by providing the corresponding  stability under the measured Gromov-Hausdorff convergence.

Motivated by the aforementioned works, our purpose is to present a comprehensive study of \textit{irreversible metric spaces}, by describing the appropriate convergence/stability of such   objects. Inspired  by Rademacher \cite{R, Rademacher}, a  central role in our study is played by the {\it reversibility} of a metric space $(X,d)$  defined by   $$\lambda_d(X):=\sup_{x\neq y} \frac{d(x,y)}{d(y,x)}.$$  Clearly, $\lambda_d(X)\geq1$, while $\lambda_d(X)=1$ holds if and only if $(X,d)$ is reversible.


The first part of the present paper focuses on the geometry of irreversible metric spaces.
The primordial question in this setting is
the modality of studying the convergence of such spaces.
In fact, this issue has been investigated first by  Shen and the second author \cite{SZ} for a collection of compact metric spaces whose  reversibilities are uniformly bounded from above by some constant $\theta\in [1,\infty)$. The main tool rests upon the  introduction of a generalized topology, say {\it $\theta$-Gromov-Hausdorff topology}, to study the convergence of such spaces. In particular, Gromov's precompactness theorem remains valid in this setting. Moreover, by analyzing the approach from \cite{SZ}, we are going to prove that the $\theta$-Gromov-Hausdorff topology
is optimal from several points of view:

\begin{itemize}

\item \textbf{Compatibility}. As expected, the $1$-Gromov-Hausdorff topology is exactly the original Gromov-Hausdorff topology in the reversible case. Moreover, if a sequence of compact metric spaces is convergent in the $\theta$-Gromov-Hausdorff topology, then it must converge to the same limit in the $\vartheta$-Gromov-Hausdorff topology for any $\vartheta\geq \theta$.

\item \textbf{Necessity of boundedness}. Although different $\theta$-Gromov-Hausdorff topologies are compatible, the uniform boundedness of their reversibilities  is a necessary condition. Indeed, whenever no such uniform boundedness  is imposed on the  reversibilities, the convergence might be not well-defined and particularly,  Gromov's precompactness
 would fail; for details, see Example \ref{flawgroha} and Remark \ref{precompactnessfail}.

\item \textbf{Fineness}. Every irreversible metric space $(X,d)$ can be symmetrized to be a reversible one $(X,\hat{d})$ by setting $\hat{d}(x,y):=\frac12[d(x,y)+d(y,x)]$. It is remarkable that if a sequence  of compact irreversible metric spaces is convergent in some $\theta$-Gromov-Hausdorff topology, then the sequence of the corresponding symmetrized spaces must converge in the original Gromov-Hausdorff topology. However, the converse need \textit{not} hold, see Example \ref{symmetricexample}. Accordingly, this generalized Gromov-Hausdorff topology is (strictly) finer than the original Gromov-Hausdorff topology.

\end{itemize}

The convergence of noncompact irreversible metric spaces is also considered in the paper; note that in this case the reversibilities might be infinity, which prevents the applicability of the approach from Shen and Zhao \cite{SZ}.
However, considering the model Funk metric    \eqref{Funckmeatirc}, a direct calculation yields
\[
\lambda_{d_F}(\mathbb{B}^n)=\infty,\quad \lambda_{d_F}( \overline{B^+_\mathbf{0}(r)})\leq 2e^r-1,
\]
 where $B^+_\mathbf{0}(r)$ is the forward open ball of radius $r$ centered at $\mathbf{0}$, i.e.,
$B^+_\mathbf{0}(r)=\{x\in \mathbb{B}^n|\, d_F(\mathbf{0},x)<r\}$.
The latter estimate suggests to consider the collection of  pointed irreversible metric spaces $(X,\star,d)$ whose reversibilities satisfy $\lambda_d(\overline{B^+_\star(r)})\leq \Theta(r),$ where $\Theta:(0,\infty)\rightarrow [1,\infty)$ is a given nondecreasing function. By means of $\Theta$, we define an appropriate topology, called as the {\it pointed forward $\Theta$-Gromov-Hausdorff topology}  on such spaces. Compared with the compact case, it is immediate to observe that the presence of the boundedness function $\Theta$ is necessary, while   different pointed forward $\Theta$-Gromov-Hausdorff topologies are compatible. Furthermore, -- as expected -- this new topology covers both the reversible and  compact cases, see Proposition \ref{welldefinednoncompactGHCONVER}.
Even more, under such a topology, Gromov's precompactness theorem still holds, see Theorem \ref{noncompactprecompact}, while every tangent space is a tangent cone of a Finsler manifold, cf. Proposition \ref{tanggespacecone}. Spectacularly, it turns out that the generalized Gromov-Hausdorff topologies both in the compact and noncompact cases can
be defined equivalently  by almost isometries, which provides the required tool to study the convergence of irreversible metric-measure spaces. In particular, the corresponding measured Gromov-Hausdorff topology is a generalized Gromov-Hausdorff-Prokhorov topology, see Section \ref{GHPTOPOLOGY}.

 The second part of the paper is devoted to the study of optimal transport on irreversible metric-measure spaces. In particular, in the compact case optimal transport is stable under the generalized Gromov-Hausdorff topology (Theorem \ref{stabilityoptimal}), which is an irreversible version of Villani \cite[Theorem 28.9]{Vi}. In addition, we define the Ricci curvature on irreversible metric-measure space, i.e., {\it the weak curvature-dimension condition} ${\CD}(K,N)$. It turns out that  the Bishop-Gromov type volume comparison, the Bonnet-Myers  compactness theorem and the Brunn-Minkowski inequality remains valid in this setting, see Section \ref{Riccappli}. Furthermore,
 the weak curvature-dimension condition is also stable under the generalized Gromov-Hausdorff topology, see Section \ref{staofRicc}. Finally, we provide various functional inequalities on weak ${\CD}(K,N)$ irreversible spaces, see Section \ref{section5-4}; {special cases  of these inequalities in the Finsler context have been proved by  Ohta \cite{O,O1, O2}, and Ohta and Sturm \cite{Ot}.}


Although some constructions throughout the paper are similar to the reversible case, peculiar differences appear due to the irreversible character of the metric spaces we are working on. {Such phenomena are strongly supported by various examples (mostly coming from irreversible Finsler structures) which provide the  motivation and real flavor of the present work.}

\section{Forward metric and length spaces}\label{Forelengthsectonstad}

\subsection{Forward metric spaces}
In order to study irreversible metric spaces, we recall the following definition which was introduced in \cite{SZ}.

\begin{definition}\label{generalsapcedef}
Let $X$ be a set and $d:X\times X\rightarrow [0,\infty)$ be a function on $X$. The pair $(X,d)$ is called  {an} {\it irreversible metric space} if for any   $x,y,z \in X$:

\smallskip

(i) $d(x,y)\geq 0,
\mbox{ with equality if and only if } x=y;$
\ \ \ (ii) $ d(x,z)\leq d(x,y)+d(y,z).$

\smallskip

\noindent In particular, if $(X,d)$ is {an} irreversible metric space, then $d$ is called a {\it metric} on $X$.
\end{definition}
\begin{remark}
The reason why we assume $d(x,y)<\infty$ for every $x,y\in X$ is to eliminate  some problematic cases. Moreover, the distance on every forward/backward geodesically complete Finsler manifold is always finite.
\end{remark}

Since the metric $d$ of
 {an} irreversible metric space $(X,d)$
  could be asymmetric, there are two kinds of balls, i.e., forward and backward balls, respectively.
  More precisely, given any  $r>0$ and a point $x\in X$,  the forward ball $B^+_x(r)$ (resp., backward ball $B_x^-(r)$) of radius $r$ centered at $x$ is defined as
\[
B^+_x(r):=\{y\in X| \, d(x,y)<r\},\quad B^-_x(r):=\{y\in X| \, d(y,x)<r\}.
\]

Let $\mathcal {T}_+$ (resp., $\mathcal {T}_-$) denote by the   topology induced by forward balls (resp., by backward balls).
In order to investigate the relation between $\mathcal {T}_+$ and $\mathcal {T}_-$,   the following definition was introduced in \cite{SZ}.
\begin{definition}\label{reversibilitydef}
Let $(X,d)$ be {an}  irreversible metric space. Given any nonempty subset $A\subset X$, define
\begin{equation*}
\lambda_d(A):=\inf\left\{\lambda\geq1|\,d(x,y)\leq \lambda\cdot d(y,x) \text{ for any }x,y\in A  \right\}.
\end{equation*}
Here, $\lambda_d(X)$ is call the {\it reversibility} of $(X,d)$.  In particular, $(X,d)$ is called reversible if  $\lambda_d(X)=1$.
\end{definition}

Irreversible metric spaces with finite reversibility are studied in \cite{SZ,Z3}, in which case $\mathcal {T}_+=\mathcal {T}_-$. In this paper, we mainly consider
a more general case.

\begin{definition}\label{thetametricspace}
Let $\Theta:[0,\infty)\rightarrow [1,\infty)$ be a (unnecessarily continuous) nondecreasing function. A triple $(X,\star,d)$ is called a {\it pointed forward  $\Theta$-metric space} if $(X,d)$ is {an}  irreversible metric space and $\star$ is a point in $X$ such that $\lambda_d\left(\overline{B^+_\star(r)}\right)\leq \Theta(r)$ for any $r>0$. Moreover, if $\Theta\equiv\theta$ is a constant, $(X,d)$ is called a {\it $\theta$-metric space}.
\end{definition}

\begin{remark}\label{forwardpointspaceandbackwardones}

 We present some remarks which are useful in the sequel:

\begin{itemize}

\item[(a)]  If $(X,\star,d)$ is a pointed forward $\Theta$-metric space, then for every other point $x\in X$, the triplet $(X,x,d)$ is also a pointed forward $\Theta'$-metric space, where $\Theta'(r):=\Theta(r+d(\star,x))$. Moreover, if $\diam(X,d):=\sup_{x,y\in X}d(x,y)<\infty$, then $(X,d)$ is a $\theta$-metric space where $\theta:=\Theta(\diam(X,d))$.

\smallskip

\item[(b)] One can similarly define a     triple $(X,\star,d)$ to be a {\it pointed backward  $\Theta$-metric space}  if  there holds $\lambda_d\left(\overline{B^-_\star(r)}\right)\leq \Theta(r)$ for any $r>0$. Note that
a pointed backward  $\Theta$-metric space could be not a pointed forward one, and vice versa. For instance, the space $(\mathbb{B}^n,\mathbf{0},d_F)$ arising in \eqref{Funckmeatirc} is a pointed forward $(2e^r-1)$-metric space, but not backward one, since $\lambda_{d_F}\left(\overline{B^-_\star(r)}\right)=\infty$ for any $r\geq \log 2$.

\item[(c)] Given {an}  irreversible metric space $(X, d)$, {\it the  reverse metric} is defined as  $\overleftarrow{d}(x,y):=d(y,x)$.  In particular,
$(X,\star,d)$ is a pointed forward (resp., backward) $\Theta$-metric space if and only if $(X,\star,\overleftarrow{d})$ is a pointed backward (resp., forward)  $\Theta$-metric space. For this reason, we focus only to  forward $\Theta$-metric spaces.

\end{itemize}
\end{remark}


For the general case, we have the following result. Since the proof is long but standard,
we leave it in Appendix \ref{propergenerlengappex}.

\begin{theorem}\label{topologychara}
Let $\mathcal {X}=(X,\star, d)$ be a pointed forward $\Theta$-metric space. Thus:
\begin{itemize}

\item[(i)] $\mathcal {T}_-\subset \mathcal {T}_+$; hence, $d$ is continuous under $\mathcal {T}_+\times  \mathcal {T}_+$ and $\mathcal {X}$ is a Hausdorff space;

\smallskip

\item[(ii)]  $\mathcal {T}_+$ is exactly the topology $\hat{\mathcal{T}}$ induced by the symmetrized metric
\[
\hat{d}(x,y)=\frac12[d(x,y)+d(y,x)].\tag{2.1}\label{symmmetricde}
\]


\end{itemize}

\end{theorem}


\noindent\textbf{Convention 1.} For simplicity of presentation, we introduce the following conventions in the sequel:

\begin{itemize}

\item[(1)]  $\theta$ always denotes a constant not less than $1$ while $\Theta$ denotes a (unnecessarily continuous) nondecreasing function from  $[0,\infty)$ to $[1,\infty)$. Particularly, $\Theta$ could be a constant $\theta$.

\item[(2)] Every pointed forward $\Theta$-metric space  $(X,\star,d)$ is endowed with the forward topology. In particular, if $\star$ and $\Theta$ are not relevant or can be naturally deduced  their form from the context, we just use $(X,d)$ to denote them and called them {\it forward  metric spaces}  for convenience.

\end{itemize}

Now we recall the following definition (cf. \cite{SZ}).
\begin{definition}
Let $(X,d)$ be {an}  irreversible metric space.

\begin{itemize}

\item
A sequence $(x_i)_i$ in $X$ is called
a {\it forward} (resp., {\it backward}) {\it Cauchy sequence} if, for each $\epsilon>0$, there exists
$N>0$ satisfying when $j\geq i>N$, then $d(x_i,x_j)<\epsilon$ (resp.,
$d(x_j,x_i)<\epsilon$).

\item
Given $\epsilon>0$, a subset $A\subset X$ is called a
{\it forward} (resp., {\it backward})  {\it $\epsilon$-net} of $X$ if, for each $x\in
X$, there exists $a_x\in A\text{ such that }d(a_x,x)<\epsilon $ (resp., $d(x,a_x)< \epsilon$).

\smallskip

\item  $(X, d)$ is called {\it forward} (resp., {\it backward}) {\it complete} if
every forward (resp. backward) Cauchy sequence in $X$ converges in
$X$ with respect to $\mathcal {T}_+$.

\item   $(X,d)$ is called {\it forward} (resp., {\it backward}) {\it totally bounded} if  it has a finite forward (resp. backward) $\epsilon$-net
for each $\epsilon > 0$.

\item  $(X,d)$ is called {\it forward} (resp., {\it backward}) {\it boundedly compact} if every bounded closed forward (resp., backward) ball  is compact.

\end{itemize}

In particular, if $(X,d)$ is called {\it complete/totally bounded/boundedly compact} if it is both  forward  and backward complete/totally bounded/boundedly compact, respectively.
\end{definition}

If $\Theta$ is a constant, then forward and backward properties are equivalent.
On the other hand, the situation dramatically changes whenever $\Theta$ is unbounded. For example, $(\mathbb{B}^n,\mathbf{0},d_F)$ coming from \eqref{Funckmeatirc} is forward but not backward complete. Nevertheless,  the following relations hold.


\begin{proposition}\label{forwardandbackwardrelation}
Let $(X,d)$ is a forward metric space. Thus,
\begin{itemize}
\item[(i)] $(X,d)$ is  complete   if and only if it is  backward complete;

\item[(ii)]  $(X,d)$ is  boundedly compact if and only if it is  backward  boundedly compact;

\item[(iii)]  $(X,d)$ is forward complete if it is forward boundedly compact.
\end{itemize}

\end{proposition}

\begin{proof}  (i) It is enough to show the $\Leftarrow$ part. Given any forward Cauchy sequence $(a_i)_i$, there exists $I>0$ such that $a_i\in B^+_{a_I}(1)$ for any $i>I$. Since $(X,a_I,d)$ is a forward $\Theta'$-metric space (see Remark \ref{forwardpointspaceandbackwardones}/(a)), we have $d(a_j,a_i)\leq \Theta'(1) \cdot d(a_i,a_j)$ for any $I<i<j$.
Thus, $(a_i)_{i\geq I+1}$ is a backward Cauchy sequence and hence, $(a_i)_i$ is convergent.

(ii) It suffices to prove the $\Leftarrow$ part.
 Given a closed forward ball $\overline{B^+_x(R)}$, by considering the pointed forward $\Theta'$-metric space $(X,x,d)$, we have $\overline{B^+_x(R)}\subset \overline{B^-_x(\Theta'(R)R)}$. Thus, $ \overline{B^+_x(R)}$ is compact
because $\overline{B^-_x(\Theta'(R)R)}$ is compact.

(iii) Given a forward Cauchy sequence $(x_i)_i$, there exists $N>0$ such that $d(x_i,x_j)<1$ if  $N\leq i<j$. Thus, $(x_j)_{j>N}\subset \overline{B^+_{x_N}(1)}$ and hence, $(x_j)_{j>N}$ is a Cauchy sequence with respect to $\hat{d}$ defined by (\ref{symmmetricde}). In view of Theorem \ref{topologychara}/(ii),
the compactness of $ \overline{B^+_{x_N}(1)}$ implies the convergence of $(x_i)_i$ under $\mathcal {T}_+$. Hence, $(X,d)$ is forward complete.
\end{proof}

Moreover, we have the following result.
\begin{theorem}\label{compactequvitheorem}
Let $(X,d)$ be a  forward metric space. The following are equivalent:

\begin{itemize}
\item[(i)] $(X,d)$ is compact;
\item[(ii)] $(X,d)$ is sequentially compact;
\item[(iii)] $(X,d)$ is forward complete and forward totally bounded;
\item[(iv)] $(X,d)$ is  complete and  totally bounded.
\end{itemize}
\end{theorem}
\begin{proof}[Proof of Sketch] If $(X,d)$ is compact, then it is a $\theta$-metric space, where $\theta:=\Theta(\diam(X,d))$, in which case (iii)$\Leftrightarrow$(iv). On the other hand,
a suitable modification of the proof of Bredon \cite[Theorem 9.4, p.25]{Bre} together with Theorem \ref{topologychara} and Proposition \ref{forwardandbackwardrelation} furnishes the equivalences (i)$\Leftrightarrow$(ii)$\Leftrightarrow$(iii).
\end{proof}

A forward metric space can be noncompact if it is forward complete and backward totally bounded; for instance, the space  $(\mathbb{B}^n,\mathbf{0},d_F)$ from \eqref{Funckmeatirc} satisfies these properties.
Now we turn to study the length of a curve in a forward metric space.
\begin{definition}\label{dfelength}
Let $(X,d)$ be a forward metric space and $\gamma$ be a path in $X$, i.e., a
continuous map $\gamma:[a,b]\rightarrow X$. Consider a partition $Y$
of $[a,b]$, that is, a finite collection of points $Y=\{t_0,\ldots
,t_N\}$ such that $a=t_0\leq t_1\leq \cdots \leq t_N=b$. The
supremum of the sums
\[
\Sigma(Y):=\overset{N}{\underset{i=1}{\sum}}d(\gamma(t_{i-1}),\gamma(t_i)),
\]
over all the partitions $Y$ is called the {\it length} of $\gamma$ (with
respect to the metric $d$) and denoted by $L_d(\gamma)$. A path is said
to be {\it rectifiable} if its length is finite.

The {\it length structure} induced by the metric $d$ is defined as follows: all continuous paths (parameterized by closed intervals) are admissible, and the length is given by the function $L_d$.
\end{definition}
For a continuous curve defined on an open interval, even if it has finite length,   the reverse may have infinite length. For example, consider the pointed forward metric space $(\mathbb{B}^n,\mathbf{0},d_F)$ from  \eqref{Funckmeatirc}; thus, $\gamma(t)=(1-t,0,\cdots,0)$, $t\in (0,1)$ satisfies  $L_{d_F}(\gamma)=\log 2$ while $L_{d_F}(\gamma^{-1})=\infty$, where $\gamma^{-1}(t):=\gamma(1-t)$.

\smallskip

\noindent \textbf{Convention 2.}
For convenience,
all the   admissible paths in this paper are defined on $[0,1]$. Given an admissible path $\gamma:[0,1]\rightarrow X$,  we use $L_d(\gamma,t,t')$ to denote the length $\gamma|_{[t,t']}$ for any $0\leq t\leq t'\leq 1$.


\smallskip

Using the same argument as in the proof of Burago, Burago and Ivanov\cite[Proposition 2.3.4]{DYS}, one can easily show the following result.

\begin{proposition}\label{basisessentially2} Let $(X,d)$  is a  forward metric space.
The length structure  satisfies the following properties: for any  admissible path $\gamma:[0,1]\rightarrow X$, we have

\begin{itemize}
\item[(i)] Generalized length inequality: $L_d(\gamma)\geq d(\gamma(0),\gamma(1))$.

\item[(ii)] Additivity: if $0<t<1$, then $L_d(\gamma,0,t)+L_d(\gamma,t,1)=L_d(\gamma)$. In particular, $L_d(\gamma,0,t)$ is a nondecreasing function of $t$.

\item[(iii)] If $\gamma(t)$, $0\leq t\leq 1$, is rectifiable, then function $L_d(\gamma,a,b)$ is uniformly continuous in $a$ and $b$.

\item[(iv)]   $L_d$ is lower semi-continuous function on $C([0,1];X)$
  with respect to pointwise convergence, where $C([0,1];X)$ denotes the set of  curves from $[0,1]$ to $X$.
\end{itemize}
\end{proposition}



\begin{definition}\label{unformconverge}
Given a forward metric space $(X,d)$,  equip $C([0,1];X)$ by the {\it (forward) uniform topology} $\mathfrak{T}_X$ induced by the metric
$\rho(\gamma_1,\gamma_2):=\max_{0\leq t\leq 1}d(\gamma_1(t),\gamma_2(t))$; i.e., a sequence $(\gamma_k)_k$ is said to be {\it convergent uniformly} to $\gamma$  {if} $\lim_{k\rightarrow \infty}\rho(\gamma,\gamma_k)=0$.
\end{definition}

\begin{remark}\label{unfromequaive}
Let  $(X,\star,d)$ be a pointed forward $\Theta$-metric space. Then   $\mathfrak{T}_X$ actually
coincides with the uniform topology induced by the symmetrized metric $\hat{\rho}(\gamma_1,\gamma_2):=\max_{0\leq t\leq 1}\hat{d}(\gamma_1(t),\gamma_2(t))$. In fact,
\[
\frac12 \rho(\gamma_1,\gamma_2)\leq \hat{\rho}(\gamma_1,\gamma_2)=\hat{\rho}(\gamma_2,\gamma_1)\leq \frac{1+\Theta\left(\rho(\star,\gamma_2)+\rho(\gamma_2,\gamma_1)\right)}2 \rho(\gamma_2,\gamma_1),
\]
where $\star$ denotes the constant curve $\gamma(t)\equiv \star$ for $t\in [0,1]$.
\end{remark}

Thanks to Remark \ref{forwardpointspaceandbackwardones}/(a), we have the following Arzel\`a-Ascoli theorem. The proof is almost the same as Burago, Burago and Ivanov  \cite[Theorem 2.5.14]{DYS} and hence, we omit it.

\begin{theorem}\label{Arzela-Ascoli Theorem}If a forward metric space is compact, then
 any sequence of curves $\gamma_i:[0,1]\rightarrow X$ with uniformly bounded length contains a uniformly converging subsequence.
\end{theorem}


\begin{definition}\label{shortpathdef}Let $(X,d)$ be a forward metric space and let $I$ denote $[0,1]$.

\begin{itemize}
\item A curve $\gamma:I\rightarrow X$ is called a {\it shortest path} if its length is minimal among the curves with the same endpoints; in other words $L_d(\zeta)\geq L_d(\gamma)$ for any curve $\zeta$ from $\gamma(0)$ to $\gamma(1)$.

\item A curve $\gamma:I\rightarrow X$ is called a {\it geodesic} if for every $t\in I$, there exists a closed interval $[a,b]$ containing $t$ in $I$ such that $L_d(\gamma,a,b)=d(\gamma(a),\gamma(b))$.

\item A curve $\gamma:I\rightarrow X$ is called a {\it minimal geodesic}, if  $L_d(\gamma,a,b)=d(\gamma(a),\gamma(b))$ for every  closed interval $[a,b]\subset I$.

\item A curve $\gamma:I\rightarrow X$ is said to  {\it have constant speed} if there exists a constant $C>0$ such that
\[
L_d(\gamma,a,b)=C(b-a),\ \forall\,[a,b]\subset I.
\]
In particular, if $C=1$, then $\gamma$ is said to be {\it naturally parameterized}.

\item A curve $\gamma:I\rightarrow X$ is said to  be {\it Lipschitz continuous} if there is a constant $C>0$ such that
\[
d(\gamma(a),\gamma(b))\leq C(b-a),\ \forall\,[a,b]\subset I,
\]
{in which case} $\gamma$ is also called a {\it $C$-Lipschitz curve}.
\end{itemize}
\end{definition}

Note that a minimal geodesic is always a shortest path, but not vice versa unless $(X,d)$ is a forward length space.

\subsection{Forward length spaces}
A forward metric space $(X,d)$ is called {\it accessible} if for every $x,y\in X$, there is a rectifiable path from $x$ to $y$. In the sequel, all  spaces are accessible.

\begin{proposition}\label{basiclengthspace}
Let $(X, \star,d)$ be a pointed forward  $\Theta$-metric space. Given two points $x,y\in X$, define the
{\rm associated metric} of $d$ as follows:
\[
d_{L}(x,y):=\inf\left\{L_d(\gamma)\,|\,\gamma:[0,1]\rightarrow X, \gamma\text{ is continuous},
\gamma(0)=x,\gamma(1)=y\right\}.\tag{2.2}\label{structureinstrcut}
\]
Then the following statements hold:
\begin{itemize}

\item[(i)] $(X,\star,d_L)$ is  a  pointed forward $\widehat{\Theta}$-metric space, where $\widehat{\Theta}(r):=\Theta\left(  (2+\Theta(r+1))r+1 \right);$

\smallskip

\item[(ii)]  $L_{d_L}(\gamma)=L_d(\gamma)$ for any rectifiable curve $\gamma$ in $(X,d);$

\smallskip

\item[(iii)] $d_{L^2}=d_L$, where   $d_{L^2}:=(d_L)_L$ is the associated metric of $d_L$.

\end{itemize}

\end{proposition}
\begin{proof}(i)
Since $(X,d)$ is accessible, it is not hard to check that $d_L: X\times X\rightarrow [0,\infty)$ satisfies the conditions in Definition \ref{generalsapcedef}. For any $r>0$, let $\mathfrak{B}^+_\star(r)$ (resp., $B^+_\star(r)$) denote the forward ball induced by $d_L$ (resp., $d$).
Obviously, $\mathfrak{B}^+_\star(r)\subset B^+_\star(r)$.
We claim
\[
d_L(x,\star)\leq \theta_1\, d_L(\star,x),\ \forall\,x\in \overline{\mathfrak{B}^+_\star(r)},\tag{2.3}\label{lengthspacereversib}
\]
where $\theta_1=\Theta(r+1)$.
In fact, for any point $x\in \overline{\mathfrak{B}^+_\star(r)}$ and any $\varepsilon\in (0,1)$, there exists a path $\gamma_1(t)$, $0\leq t\leq 1$ from $\star$ to $x$ with $L_d(\gamma_1)<d_L(\star,x)+\varepsilon<r+1$.  Hence, $\gamma_1([0,1])\subset \mathfrak{B}^+_\star(r+1)\subset B^+_\star(r+1)$. For any partition $Y=\{t_i\}$ of $[0,1]$, by Definition \ref{thetametricspace} we have
\[
\Sigma(Y^{-1}):=\sum_i d(\gamma_1(t_{i+1}),\gamma_1(t_i))\leq \theta_1\sum_i d(\gamma_1(t_{i}),\gamma_1(t_{i+1}))=\theta_1\Sigma(Y),
\]
which together with relation (\ref{structureinstrcut}) implies
\[
d_L(x,\star)\leq \sup_Y\Sigma(Y^{-1})\leq \theta_1 \sup_Y\Sigma(Y)=\theta_1 \,L_d(\gamma_1)\leq \theta_1(d_L(\star,x)+\varepsilon).
\]
Therefore,  the claim (\ref{lengthspacereversib}) is true.

Now we show
\[
d_L(x,y)\leq \widehat{\Theta}(r)\, d_L(y,x),\ \forall\,x,y\in \overline{\mathfrak{B}^+_\star(r)}.\tag{2.4}\label{medietherPro3.15}
\]
In fact,
for any $x,y\in \overline{\mathfrak{B}^+_\star(r)}$, the triangle inequality together with (\ref{lengthspacereversib}) yields
\[
d_L(x,y)\leq d_L(x,\star)+d_L(\star,y)\leq  \theta_1\,d_L(\star,x)+d_L(\star,y)\leq(1+\theta_1)r.
\]
For any $\varepsilon\in (0,1)$, there exists a path $\gamma_2(t)$, $0\leq t\leq 1$ from $x$ to $y$ with
\[
L_d(\gamma_2)<d_L(x,y)+\varepsilon<(1+\theta_1)r+1,
\]
which implies $\gamma_2([0,1])\subset \mathfrak{B}^+_\star\left((2+\theta_1)r+1  \right)\subset {B}^+_\star\left((2+\theta_1)r+1  \right)$.
Thus, for any partition $Y$ of $[0,1]$,
\[
\Sigma(Y^{-1}):=\sum_i d(\gamma_2(t_{i+1}),\gamma_2(t_i))\leq \widehat{\Theta}(r)\,\sum_i d(\gamma_2(t_{i}),\gamma_2(t_{i+1}))=\widehat{\Theta}(r)\cdot\Sigma(Y),
\]
which implies $d_L(y,x)\leq L_d(\gamma_2^{-1})\leq \widehat{\Theta}(r) L_d(\gamma_2)\leq \widehat{\Theta}(r) \,(d_L(x,y)+\varepsilon)$.
Therefore, (\ref{medietherPro3.15}) follows.

(ii) For simplicity, set $\tilde{d}:=d_L$. It follows from $ \tilde{d}\geq d$ that $L_{\tilde{d}}(\gamma)\geq L_d(\gamma)$. For the reverse inequality, let $Y:=\{t_i\}$ be an arbitrary partition of $[0,1]$. Thus (\ref{structureinstrcut}) yields $\tilde{d}(\gamma(t_i),\gamma(t_{i+1}))\leq L_d(\gamma,t_i,t_{i+1})$,
which together with Proposition \ref{basisessentially2}/(ii) furnishes
\[
\Sigma_{\tilde{d}}(Y)=\sum \tilde{d}(\gamma(t_i),\gamma(t_{i+1}))\leq \sum L_d(\gamma,t_i,t_{i+1})=L_d(\gamma).
\]
The arbitrariness of $Y$ indicates $L_{\tilde{d}}(\gamma)\leq L_d(\gamma)$. Hence, (ii) follows.
Moreover, (iii) is a direct consequence of (ii).
\end{proof}

\begin{definition}Let  $(X,\star,d_L)$ be a pointed forward $\Theta$-metric space. The metric $d_L$ is called an {\it intrinsic  metric} (or a {\it length metric}) if there is a forward metric space $(X,d)$ satisfying (\ref{structureinstrcut}).
 In this case $(X,\star,d_L)$ is called a {\it pointed forward ${\Theta}$-length space}.  Furthermore,
  the intrinsic metric $d_L$ is called {\it strictly intrinsic} if
 for every two points $x, y \in X$ there exists an admissible path $\gamma$
from $x$ to $y$ with  $L_d(\gamma)=d_L(x, y)$, in which case $(X,\star,d_L)$ is called a {\it pointed forward ${\Theta}$-geodesic space}.
\end{definition}

\noindent \textbf{Convention 3.} For a pointed forward $\Theta$-length/geodesic space $(X,\star,d)$, if $\star$ and $\Theta$ are not relevant or can be naturally deduced  their form from the context, we just use {\it a forward length/geodesic space $(X,d)$} to denote it    for simplicity.
Moreover, if $ {\Theta}$ is a constant $\theta$, then $(X,d)$ is called a {\it $\theta$-length/geodesic space}.

\smallskip



The following result is useful although the proof is trivial.
\begin{proposition}\label{metricsigeodesic}
A forward metric space $(X,d)$  is a forward geodesic space if and only if for any $x,y\in X$, there exists a minimal geodesic
from $x$ to $y$.
\end{proposition}

By Proportions \ref{compctlacalcomapctboundecompact} and  \ref{boundedlyinstrinct}, we have the following result immediately.
\begin{theorem}\label{geodesicexistencethe}
Every forward complete locally compact forward length space is always a forward geodesic space, i.e., every two points in such a space can be connected by a minimal geodesic.
\end{theorem}





By letting $I$ be an arbitrary interval in Definition \ref{shortpathdef},
one can get the definition of geodesic defined on nonclosed intervals.
 Thus we have a generalized version of Hopf-Rinow-Cohn-Vossen Theorem.

\begin{theorem}\label{HopfRinowth} For a locally compact   forward  length space $(X,d)$, the  following  assertions are equivalent:

\begin{itemize}

\item[(i)] $(X,d)$ is forward complete;

\item[(ii)]  $(X,d)$ is forward boundedly compact;

\item[(iii)] Each  constant-speed geodesic $\gamma:[0,1)\rightarrow X$ can be extended to a continuous path $\overline{\gamma}:[0,1]\rightarrow X$;

\item[(iv)] There is a point $p\in X$ such that every constant-speed minimal geodesic $\gamma:[0,1)\rightarrow X$ with $\gamma(0)=p$ can be extended to a continuous path $\overline{\gamma}:[0,1]\rightarrow X$.

\end{itemize}
\end{theorem}

\begin{proof}[Sketch of the proof] Let $(X,\star,d)$ be the corresponding pointed forward $\Theta$-length space.
First,  (i)$\Leftrightarrow$(ii) follows from Proposition \ref{compctlacalcomapctboundecompact} and Proposition \ref{forwardandbackwardrelation}/(iii) directly.

 Now we show (ii)$\Rightarrow$(iii).
 Without loss of generality, we consider a natural parameterized geodesic $\gamma:[0,1)\rightarrow X$. Choose an increasing sequence $(t_i)_i$ convergent to $1$. Since $d(\gamma(t_i),\gamma(t_j))=(t_j-t_i)<1$ for any $j>i$, we see that
 $(\gamma(t_i))_i$ is a forward Cauchy sequence in $\overline{B^+_{\gamma(0)}(2)}$.
  Thanks to the compactness of $\overline{B^+_{\gamma(0)}(2)}$, the sequence $(\gamma(t_i))_i$ must converge to a point $q\in \overline{B^+_{\gamma(0)}(2)}$. Now define $q:=\gamma(1)$. By the triangle inequality, one can easily check this definition is well-defined and particularly,
$\gamma$ is continuous over $[0,1]$.

\smallskip (iii)$\Rightarrow $(iv) is obvious. Thus we have proved (i)$\Leftrightarrow$(ii)$\Rightarrow$(iii)$\Rightarrow $(iv). It remains to show (iv)$\Rightarrow $(ii). In view of Theorem \ref{compactequvitheorem} and Proposition \ref{boundedlyinstrinct},  a similar argument as in the proof of Burago, Burago and Ivanov \cite[Theorem 2.5.28]{DYS}  yields the required conclusion.
\end{proof}

\subsection{Forward metric-measure spaces}
In the paper, all the measures are non-negative locally finite Borel measures. A triple $(X,d,\mu)$ is called a {\it forward metric-measure} (resp., {\it length-measure, geodesic-measure}) {\it space}  if $(X,d)$ is a forward metric (resp., length, geodesic) space  endowed with a measure $\mu$. 

\begin{definition}\label{doulbidefin}
Let $\mathcal {X}:=(X,d, \mu)$ be a forward metric-measure space, where $\mu$ is a nonzero measure.

\begin{itemize}

\item $\mathcal {X}$ is called {\it globally doubling}, if there is a constant $L>0$ with
\[
 \mu\left[ \overline{B^+_x(2r)} \right]\leq L\cdot \mu\left[ \overline{B^+_x(r)} \right],\ \forall\,x\in X, \ \forall\,r>0.
\]

\item $\mathcal {X}$   is   called {\it almost doubling} if for any $r>0$,  there is a constant $L=L(r)>0$ with
\[
 \mu\left[ \overline{B^+_x(2r)} \right]\leq L\cdot \mu\left[ \overline{B^+_x(r)} \right],\ \forall\,x\in X.
\]


\end{itemize}
\end{definition}



\noindent \textbf{Convention 4.}
 Let $\mathcal {X}=(X,d,\mu)$ be a forward metric-measure space.

\begin{itemize}
\item[(1)] If $\mu$ is a nonzero measure, we introduce the notation
\[
\Do(\mathcal {X}):=\inf\left\{ L>0\,\left|\, \mu\left[ \overline{B^+_x(2r)} \right]\leq L\right.\cdot \mu\left[ \overline{B^+_x(r)} \right],\ \forall\,x\in X,\ \forall\,r>0 \right\}.
\]
Clearly, $\mathcal {X}$ is globally doubling if and only if $\Do(\mathcal {X})$ is finite.

\smallskip

\item[(2)] For convenience, we abuse notations
\[
\mu[\mathcal {X}]:=\mu[X],\ \ \ \diam(\mathcal {X}):=\diam(X,d)=\sup_{x,y\in X}d(x,y).
\]

\item[(3)] If $\mathcal {X}$ is bounded (i.e., $\diam(\mathcal {X})<\infty$),  set
\begin{align*}
\Ca_X(\varepsilon)&:= \text{maximum number of disjoint forward
}\frac{\varepsilon}{2}
\text{-balls in } X,\\
\Cov_X(\varepsilon)&:= \text{minimum number of forward }\varepsilon
\text{-balls it takes to cover }X.
\end{align*}
Since $\mathcal {X}$ is bounded, we can assume that $(X,d)$ is a $\theta$-metric space (see Remark \ref{forwardpointspaceandbackwardones}/(a)). Then
\[
\Ca_X(2\varepsilon)\leq \Cov_X(\varepsilon)\leq \Ca_X\left(\frac{\varepsilon}{\theta}\right),\ \forall\varepsilon>0.\tag{2.5}\label{capvscov}
\]
\end{itemize}

\begin{lemma}\label{doublingwithprecompact} Given $D,L,m,M>0$, let $\mathcal {X}=(X,d,\mu)$ be a compact $\theta$-metric-measure space with
\[
\diam(\mathcal {X})\leq D,\ \ \Do(\mathcal {X})\leq L,\ \ m\leq \mu[\mathcal {X}]\leq M.
\]
Thus for any $\varepsilon>0$, there is an $N(\varepsilon)=N(\varepsilon;\theta,D,L,m,M)$  such that $\Cov_X(\varepsilon)\leq \Ca_X\left(\frac{\varepsilon}{\theta}\right)\leq N(\varepsilon)$.
\end{lemma}
\begin{proof} Given $\varepsilon>0$, choose  a maximum  disjoint forward $\varepsilon/(2\theta)$-ball family $\{B^+_{x_i}(\varepsilon/(2\theta))\}_{i=1}^{\Ca_X(\varepsilon/\theta)}$.
On the one hand, by setting $T=\left\lceil  \log_2\frac{2\theta D}{\varepsilon} \right\rceil$, we have
\[
m\leq\mu(B_{x_i}^+(D))\leq L^T \mu(B^+_{x_i}(\varepsilon/(2\theta)))\Longrightarrow \mu(B^+_{x_i}(\varepsilon/(2\theta)))\geq m \cdot L^{-T}.
\]
 On the other hand, one has
\[
0\leq \mu\left[  X\backslash \sqcup_i B^+_{x_i}(\varepsilon/(2\theta)) \right]\leq M-\Ca_X\left(\frac{\varepsilon}{\theta}\right)\cdot m\cdot L^{-T}\Longrightarrow \Ca_X\left(\frac{\varepsilon}{\theta}\right)\leq \frac{M}{m} L^T:=N(\varepsilon).
\]
which concludes the proof of (\ref{capvscov}).
\end{proof}

\subsection{Finsler metric-measure spaces}

\subsubsection{Finsler manifolds.} In this subsection, we study forward length spaces induced by Finsler manifolds. First, we recall some definitions and properties from Finsler geometry; for details see  Bao, Chern and Shen \cite{BCS} and Shen \cite{Shen2013,Sh1}.

%
%
%

 Let $M$ be an $n(\geq2)$-dimensional connected
 smooth manifold without boundary and $TM=\bigcup_{x \in M}T_{x}
M $ be its tangent bundle. The pair $(M,F)$ is a \textit{Finsler
	manifold} if the continuous function $F:TM\to [0,\infty)$ satisfies
the conditions

\begin{itemize}

\item[(a)] $F\in C^{\infty}(TM\setminus\{ 0 \});$

\item[(b)] $F(x,\lambda y)=\lambda F(x,y)$ for all $\lambda\geq 0$ and $(x,y)\in TM;$

\item[(c)] $g_{ij}(x,y)=[\frac12F^{2}%
]_{y^{i}y^{j}}(x,y)$ is positive definite for all $(x,y)\in
TM\setminus\{ 0 \}$ where $F(x,y)=F(y^i\frac{\partial}{\partial x^i}|_x)$.

\end{itemize}

\noindent Note that $g_{ij}$  cannot be defined at $y=0$ unless $F$ is Riemannian. Furthermore, the Euler theorem yields that $F^2(x,y)=g_{ij}(x,y)y^iy^j$ for every $(x,y)\in TM\backslash\{0\}$.

Set $S_xM:=\{y\in T_xM\,|\,F(x,y)=1\}$ and $SM:=\cup_{x\in M}S_xM$. The {\it reversibility} $\lambda_F(M)$ (cf. Rademacher \cite{R, Rademacher}) and the {\it uniformity constant} $\Lambda_F(M)$ (cf. Egloff \cite{E}) of $(M,F)$ are defined as follows:
\[
\lambda_F(M):=\underset{y\in SM}{\sup}F(-y),\quad \Lambda_F(M):=\underset{v,y,z\in SM}{\sup}\frac{g_v(y,y)}{g_z(y,y)},
\]
where $g_v(y,y):=g_{ij}(v)y^iy^j$.
Clearly, ${\Lambda_F}(M)\geq \lambda_F^2(M)\geq 1$. In particular, $\lambda_F(M)=1$ if and only if $F$ is reversible (i.e., symmetric), while $\Lambda_F(M)=1$ if and
only if $F$ is Riemannian.  For convenience, we also introduce the reversibility of a subset $U\subset M$, i.e.,
\[
\lambda_F(U):=\sup_{y\in SU}F(-y), \text{ where }SU:=\cup_{x\in U}S_xM.
\]
In particular, $\lambda_F(x):=\lambda_F(\{x\})$  is a continuous function.
Furthermore, for a Rander manifold $(M,\alpha+\beta)$, there holds
\[
\lambda_F(M)=\frac{1+b }{1-b },\quad \Lambda_F(M)=\left( \frac{1+b }{1-b } \right)^2,\tag{2.6}\label{Randersnormreveruniform}
\]
where $b:=\sup_{x\in M}\|\beta\|_\alpha(x)$.

Given $y\in T_xM$, the {\it Riemannian curvature} $R_y$ of $F$ is a   linear transformation on $T_xM$. More precisely,
$R_y=R^i_k(y)\frac{\partial}{\partial x^i}\otimes dx^k$, where
\begin{align*}
R^i_{\,k}(y)&:=2\frac{\partial G^i}{\partial x^k}-y^j\frac{\partial^2G^i}{\partial x^j\partial y^k}+2G^j\frac{\partial^2 G^i}{\partial y^j \partial y^k}-\frac{\partial G^i}{\partial y^j}\frac{\partial G^j}{\partial y^k},\\
G^i(y)&:=\frac14 g^{il}(y)\left\{2\frac{\partial g_{jl}}{\partial x^k}(y)-\frac{\partial g_{jk}}{\partial x^l}(y)\right\}y^jy^k.
\end{align*}
Let $P:=\text{Span}\{y,v\}\subset T_xM$ be a plane; the \textit{flag curvature} is defined by
\[
\mathbf{K}(y,v):=\frac{g_y\left( R_y(v),v  \right)}{g_y(y,y)g_y(v,v)-g^2_y(y,v)}.
\]
The  {\it Ricci curvature} of $y$
is defined by
\[
\mathbf{Ric}(y):=\underset{i}{\sum}\,\textbf{K}(y,e_i)=\frac{R^i_{\,i}(y)}{F^2(y)},
\]
where $\{e_1,\ldots, e_n\}$ is a $g_y$-orthonormal basis for $T_xM$.

\subsubsection{Forward length space induced by Finsler structure}

Let $\zeta:[0,1]\rightarrow M$ be a piecewise smooth path. The \textit{length} of $\zeta$ is defined by
\[
L_F(\zeta):=\int^1_0 F(\dot{\zeta}(t)){\ddd}t.
\]
Define the {\it distance function} $d_F:M\times M\rightarrow [0,\infty)$ by
\[
d_{F}(p,q):=\inf\{L_F(\gamma)\,|\,\gamma:[0,1]\rightarrow X \text{ is a piecewise smooth path with }
\gamma(0)=p,\gamma(1)=q\}.\tag{2.7}\label{distanceinFinslersetting}
\]
Thus, $d_F:M\times M\rightarrow [0,\infty)$ is a continuous function with

\smallskip

(i) $d_F(p,q)\geq 0,
\mbox{ with equality if and only if } p=q;$
\ \ \ (ii) $ d_F(p,q)\leq d_F(p,t)+d_F(t,q).$

\smallskip

\noindent However, $d_F(p,q)\neq d_F(q,p)$  unless $F$ is reversible.
According to Busemann and Mayer \cite{BM},
if a curve $\gamma: (-\epsilon,\epsilon)\rightarrow M$ is $C^1$, there holds
\[
F(\dot{\gamma}(0))=\lim_{t\rightarrow 0^+}\frac{d_F(\gamma(0),\gamma(t))}{t}=\lim_{t\rightarrow 0^+}\frac{d_F(\gamma(-t),\gamma(0))}{t}.\tag{2.8}\label{BM-1}
\]
Furthermore, $\mathcal {T}_+= \mathcal {T}_-$ is exactly the original topology of $M$.

A smooth curve $t\mapsto \gamma(t)$ in $(M,d_F)$ is  a (constant-speed) \textit{geodesic} if it satisfies
\[
\ddot{\gamma}^i(t)+2G^i\left(\dot{\gamma}\right)=0.\tag{2.9}\label{odegeodesic}
\]
In this paper, we always use $\gamma_y(t)$ to denote  the geodesic with $\dot{\gamma}_y(0)=y$.
 Recall that $(M,F)$ is called a {\it forward complete} Finsler manifold if $(M,d_F)$ is forward complete, in which  case for each $y\in T_xM\backslash\{0\}$, the {\it exponential map} $\exp_x(ty):=\gamma_y(t)$ can be defined on
 $t\in [0,\infty)$.

\begin{theorem}\label{reversibifinslerdd}
Let $(M,F)$ be a forward complete Finsler manifold.  Thus,
\begin{itemize}
\item[(i)] for any $\star\in M$, the triple $(M,\star,d_F)$ is a pointed $\Theta$-forward geodesic space, where
\[
\Theta(r):=\lambda_F\left(\overline{B^+_\star\left( 2r+\lambda_F(B^+_\star(r))\, r\right)}\right);\tag{2.10}\label{FinslerreversThera}
\]

\item[(ii)] for any convex closed set $A\subset M$, one has $\lambda_{d_F}(A)=\lambda_F(A)$. In particular,
\[
\lambda_{d_F}(M)=\lambda_F(M).\tag{2.11}\label{reversibdindentiy}
\]

\end{itemize}
\end{theorem}
\begin{proof}(i) Given any $r>0$ and $x,y\in \overline{B^+_\star(r)}$,
owing to Bao, Chern and Shen  \cite[Proposition 6.5.1]{BCS},
there exists a minimal geodesic $\gamma:[0,1]\rightarrow M$ from $x$ to $y$, i.e., $L_F(\gamma)=d_F(x,y)$. The triangle inequality yields $\gamma([0,1])\subset\overline{B^+_\star(2r+\lambda_F(B^+_\star(r))\, r)}$, which furnishes
\begin{align*}
d_F(y,x)\leq \int^1_0 F(-\dot{\gamma}(1-t)){\ddd} t\leq \Theta(r)\int^1_0 F(\dot{\gamma}(t)){\ddd}t=\Theta(r)\cdot d_F(x,y). \tag{2.12}\label{finslerrevervsdrever}
\end{align*}
Hence, $(M,\star,d_F)$ is a pointed forward $\Theta$-metric space. Moreover, by (\ref{BM-1}) one gets $(d_F)_L(x,y)=L_{d_F}(\gamma)=L_F(\gamma)=d_F(x,y)$, which implies that $d$ is strictly intrinsic.

\smallskip

(ii)
Since $A$ is convex, a similar argument to (\ref{finslerrevervsdrever}) yields $\lambda_{d_F}(A)\leq \lambda_F(A)$. For the reverse inequality, the continuity of  $\lambda_F(x)$  furnishes
  a sequence $(x_i)_i\subset A$ such that $\lambda_F(x_i)\rightarrow \lambda_F(A)$. For each $x_i$, there is a $y_i\in T_{x_i}M\setminus\{0\}$ with
$\lambda_F(x_i)={F(x_i,y_i)}/{F(x_i,-y_i)}$.
Denote by $\gamma_i(s)$, $s\in [0,\varepsilon_i]$  the minimal geodesic from $x_i$ with $\dot{\gamma}_i(0)=y_i$, where $\varepsilon_i\rightarrow 0$. By modifying $(x_i)_i$, we may assume
$\gamma_i([0,\varepsilon_i])\subset A$. Also set $\gamma^{-1}_i(s):=\gamma_i(\varepsilon_i-s)$ for $s\in [0,\varepsilon_i]$. It is not hard to check that
  \begin{align*}
\lambda_{d_F}(A)\geq \frac{d_F(x_i,\gamma_i(\varepsilon_i))}{d_F(\gamma_i(\varepsilon_i),x_i)}\geq \frac{d_F(x_i,\gamma_i(\varepsilon_i))}{L_F(\gamma_i^{-1})}=\frac{\int^{\varepsilon_i}_0 F(\dot{\gamma}_i(s)){\ddd}s}{\int^{\varepsilon_i}_0 F(-\dot{\gamma}_i(\varepsilon_i-s)){\ddd}s}\rightarrow \lambda_F(A), \text{ as }i\rightarrow \infty,
\end{align*}
which concludes the proof.
\end{proof}

We reconsider the metric from \eqref{Funckmeatirc}.
Since  $F$ is projectively
flat (i.e., the geodesics are straight lines), it follows  that
$\overline{B^+_{\mathbf{0}}(r)}=\left\{x\in \mathbb{B}^n\,|\ \|x\|\leq1-e^{-r}\right\}$,
which is a convex closed set. Moreover, a direct calculation together with (\ref{Randersnormreveruniform}) yields $\lambda_F(\overline{B^+_{\mathbf{0}}(r)})=2e^r-1$.
In view of Theorem \ref{reversibifinslerdd}/(ii)(i), the triple $(\mathbb{B}^n,\mathbf{0},d_F)$ is a pointed forward  $(2e^r-1)$-geodesic space.
 \begin{remark}
 It follows from Theorem \ref{reversibifinslerdd}/(i) that $(\mathbb{B}^n,\mathbf{0},d_F)$ is a pointed forward  $(2e^{(2e^r+1)r}-1)$-geodesic space, which is compatible with the above argument because $ 2e^{(2e^r+1)r}-1>2e^r-1$.
 \end{remark}

\subsubsection{Finsler metric-measure manifolds}

Let $\mathfrak{m}$ be a smooth positive measure on $M$; in a local coordinate system $(x^i)$ we
express ${\dd}=\sigma(x)\text{d}x^1\cdots \text{d}x^n$. In particular,
the \textit{Busemann-Hausdorff measure} ${\dd}_{BH}$ and the \textit{Holmes-Thompson measure} ${\dd}_{HT}$ are defined by
\begin{align*}
&{\dd}_{BH}:=\frac{\vol(\mathbb{B}^{n})}{\vol(B_xM)}\text{d}x^1\cdots \text{d}x^n,\\
 &{\dd}_{HT}:=\left(\frac1{\vol(\mathbb{B}^{n})}\ds\int_{B_xM}\det g_{ij}(x,y)\text{d}y\cdots \text{d}y^n \right) \text{d}x^1\cdots \text{d}x^n,
\end{align*}
where $B_xM:=\{y\in T_xM|\, F(x,y)<1\}$ and $\mathbb{B}^{n}$ is the usual Euclidean $n$-dimensional unit ball.

Define the \textit{distortion} of $(M,F,{\dd})$ as
\begin{equation*}
\tau(y):=\log \frac{\sqrt{\det g_{ij}(x,y)}}{\sigma(x)},\ \text{$y\in T_xM\backslash\{0\}$},
\end{equation*}
and the \textit{$S$-curvature} $\mathbf{S}$ is given by
\begin{equation*}
\mathbf{S}(y):=\left.\frac{d}{dt}\right|_{t=0}\tau(\dot{\gamma}_y(t)).
\end{equation*}

%


A triple $(M,F,\mathfrak{m})$ is called a {\it Finsler metric-measure manifold} if $(M, F)$ is a Finsler manifold equipped with a smooth positive measure $\mathfrak{m}$. It should be remarked that for a Finsler metric-measure manifold $(M,F,\mathfrak{m})$, the very space we care about is the forward length-measure space $(M,d_F,\mathfrak{m})$. However, similar to the Riemannian case, we do not mention $(M,d_F,\mathfrak{m})$ particularly.

For a Finsler metric-measure manifold $(M,F,\mathfrak{m})$,
the {\it weighted Ricci curvature} $\mathbf{Ric}_N$, introduced in Ohta and Sturm\cite{Ot}, is defined as follows: given $N\in [n,\infty]$, for any unit vector $y\in SM$,
\begin{align*}\mathbf{Ric}_N(y)=\left\{
\begin{array}{lll}
\mathbf{Ric}(y)+\left.\frac{d}{dt}\right|_{t=0}\mathbf{S}(\dot{\gamma}_y(t))-\frac{\mathbf{S}^2(y)}{N-n}, && \text{ for }N\in (n,\infty),\\
\\
\underset{L\downarrow n}{\lim}\mathbf{Ric}_L(y), && \text{ for }N=n,\tag{2.13}\label{defRicN}\\
\\
\mathbf{Ric}(y)+\left.\frac{d}{dt}\right|_{t=0}\mathbf{S}(\dot{\gamma}_y(t)),  && \text{ for }N=\infty.
\end{array}
\right.
\end{align*}
In the Riemannian case, it is exactly the modified Ricci tensor $\mathbf{Ric}_{N,\mathfrak{m}}$. Moreover, set
\begin{align*}\mathfrak{s}_{K,N}(t):=\left\{
\begin{array}{lll}
\sqrt{\frac{N-1}K}\sin \left(r\sqrt{\frac{K}{N-1}}\right), && \text{ if }K>0,\\
\\
r, && \text{ if }K=0,\tag{2.14}\label{skdefinition}\\
\\
\sqrt{\frac{N-1}{-K}}\sinh \left(r\sqrt{\frac{-K}{N-1}}\right),  && \text{ if }K<0.
\end{array}
\right.
\end{align*}
Thus there holds the following Bishop-Gromov comparison theorem.
\begin{theorem}[Ohta \cite{O1}]\label{Ohtacomparsiontheorem}Let $(M,F,\mathfrak{m})$  be an $n$-dimensional forward complete Finsler metric-measure manifold.
Assume that there are constants $K\in \mathbb{R}$ and
$N\in [n,\infty)$ such that $\mathbf{Ric}_N\geq K$.
Then we have $\diam(M)\leq  \pi\sqrt{(N-1)/K}$
  if $K > 0$ and, for any $x\in  M$
and $0 < r \leq R$ ($\leq \pi\sqrt{(N-1)/K}$ if $K > 0$), there holds
\[
\frac{\mathfrak{m}[B^+_x(R)]}{\mathfrak{m}[B^+_x(r)]}\leq \frac{\int^R_0 \mathfrak{s}_{K,N}(t)^{N-1}{\ddd}t}{\int^r_0 \mathfrak{s}_{K,N}(t)^{N-1}{\ddd}t}\leq e^{R\sqrt{(N-1)|K|}}\left( \frac{R}{r} \right)^N.
\]
\end{theorem}

For the Busemann-Hausdorff measure or the Holmes-Thompson measure, we have the following result.
\begin{theorem}[Zhao and Shen \cite{ZS}]\label{normalRiccicompari}
Let $(M,F,\mathfrak{m})$  be an $n$-dimensional forward complete Finsler metric-measure manifold, where $\mathfrak{m}$ is either the Busemann-Hausdorff measure or the Holmes-Thompson measure.
Assume that there are constants $\theta\geq 1$ and $K\in \mathbb{R}$  such that $\Lambda_F(M)\leq \theta^2$ and $\mathbf{Ric}\geq K$.
Then  for any $x\in  M$
and $0 < r \leq R$ ($\leq \pi\sqrt{(n-1)/K}$ if $K > 0$),
\begin{align*}
\mathfrak{m}[B^+_x(r)]\leq \vol(\mathbb{S}^{n-1})\theta^{2n}\int^r_0\mathfrak{s}_{K,n}(t)^{n-1}{\ddd}t,\ \frac{\mathfrak{m}[B^+_x(R)]}{\mathfrak{m}[B^+_x(r)]}\leq \theta^{4n}\frac{\int^R_0 \mathfrak{s}_{K,n}(t)^{n-1}{\ddd}t}{\int^r_0 \mathfrak{s}_{K,n}(t)^{n-1}{\ddd}t}\leq \theta^{4n}e^{R\sqrt{(n-1)|K|}}\left( \frac{R}{r} \right)^n,
\end{align*}
where $\mathbb{S}^{n-1}$ is the usual Euclidean $(n-1)$-dimensional unit sphere.
\end{theorem}

\section{Gromov-Hausdorff convergence for forward metric spaces}\label{GHDIS}
 \noindent \textbf{Convention 5.} Let $(X,d_X)$ and $(Y,d_Y)$ be two forward metric spaces.
\begin{itemize}
\item[(a)] Given a subset $A\subset X$ and an $\epsilon>0$, set ${A^\epsilon}:=\{ x\in X|\, d_X(A,x)<\epsilon\}$;

\item[(b)] Given a map $f:X\rightarrow Y$, set $$\dis f:=\sup_{x,x'\in X}|d_Y(f(x),f(x'))-d_X(x,x')|.$$
\end{itemize}

\begin{definition}\label{epsilonapproximation}
Let $(X,\star_X,d_X)$ and $(Y,\star_Y,d_Y)$ be two pointed forward $\Theta$-metric spaces. Thus
\begin{itemize}
\item[(i)] an {\it isometry}  $f:X\rightarrow Y$   is a homeomorphism  with $d_Y(f(x),f(x'))=d_X(x,x')$ for any $x,x'\in X$, in which case $(X,d_X)$ is said to be {\it isometric to} $(Y,d_Y)$;

\item[(ii)]  a {\it pointed isometry} $f:X\rightarrow Y$  is an isometry with $f(\star_X)=f(\star_Y)$, in which case $(X,\star_X,d_X)$ is said to be {\it pointed isometric to} $(Y,\star_Y,d_Y)$;

\item[(iii)] given $\epsilon>0$, an {\it $\epsilon$-isometry}  $f:X\rightarrow Y$ is a (not necessarily continuous) map with $\dis f\leq \epsilon$ and $Y\subset \overline{[f(X)]^\epsilon}$.

\item[(iv)] given $\epsilon>0$, a {\it pointed $\epsilon$-isometry}  $f:X\rightarrow Y$ is an $\epsilon$-isometry with $f(\star_X)=\star_Y$.
\end{itemize}

\end{definition}

\subsection{Gromov-Hausdorff topology I: compact spaces}

\begin{definition}[Forward Hausdorff distance]
Let $(X,d)$ be a  forward metric space.  Given two sets $A,B\subset X$,
the Hausdorff distance between them is defined as $d_H(A,B):=\inf\{\epsilon>0\,|\,A\subset B^\epsilon,B\subset
A^\epsilon\}$.
\end{definition}


In order to study the convergence of compact forward metric spaces, we recall the following generalized
Gromov-Hausdorff distance (cf. Shen and Zhao \cite{SZ}).
\begin{definition}[\cite{SZ}]\label{GromoHasdrdelta}
Given $\theta\geq 1$, let $\mathcal {M}^\theta$ denote the collection of all compact $\theta$-metric spaces.
Let $\mathcal {X}_i:=(X_i,d_i)$, $i=1,2$ be two elements in $\mathcal {M}^\theta$. A    {\it $\theta$-admissible metric} $d$  on the disjoint union $X_1\sqcup X_2$ is  {an}  irreversible metric with $\lambda_d(X_1\sqcup X_2)\leq \theta$ and  $ d|_{X_i}=d_i, i=1,2$.
The {\it$\theta$-{Gromov-Hausdorff} distance between
$\mathcal {X}_1$ and $\mathcal {X}_2$} is defined as
\[
d^\theta_{GH}(\mathcal {X}_1,\mathcal {X}_2):=\inf\{d_H(X_1,X_2)|\,\text{$\theta$-admissible
metrics on }X_1\sqcup X_2 \}.
\]
\end{definition}
\begin{remark}
An equivalent definition is $d^\theta_{GH}(\mathcal {X}_1,\mathcal {X}_2):=\inf\{d_H(\mathcal {X}'_1,\mathcal {X}'_2)\}$,
where the infimum is taken over all isometric embeddings $\mathcal {X}'_i$ of $\mathcal {X}_i$,  $i=1,2$ into a common $\theta$-metric space $\mathcal {Y}$. That is, $\mathcal {X}'_i$ is isometric to $\mathcal {X}_i$, $i=1,2$ and each $\mathcal {X}'_i$ is subspaces of $\mathcal {Y}$. The proof is similar to Burago, Burago and Ivanov \cite[Remark 7.3.12]{DYS}.
\end{remark}


Obviously, the $1$-Gromov-Hausdorff distance is the original Gromov-Hausdorff distance in the reversible case.
Moreover, the $\theta$-Gromov-Hausdorff distance is a reversible pseudo-metric on $\mathcal {M}^\theta$.
\begin{proposition}[\cite{SZ}]\label{Gromov-Haudsdiscpro}
Let $\mathcal {X},\mathcal {Y},\mathcal {Z}$ be three elements in $\mathcal {M}^\theta$. Thus, we have

\begin{itemize}
	\item[(i)] $d^\theta_{GH}(\mathcal {X},\mathcal {Y})=d^\theta_{GH}(\mathcal {Y},\mathcal {X})$;

\smallskip

\item[(ii)] $d^\theta_{GH}(\mathcal {X},\mathcal {Z})\leq d^\theta_{GH}(\mathcal {X},\mathcal {Y})+d^\theta_{GH}(\mathcal {Y},\mathcal {Z})$;

\smallskip

\item[(iii)] $d^\theta_{GH}(\mathcal {X},\mathcal {Y})\geq 0$, with equality if and only if $\mathcal {X}$ is isometric to $\mathcal {Y}$.
\end{itemize}

\end{proposition}

\begin{definition}[\cite{SZ}]\label{def-conv}
A sequence $(\mathcal {X}_i)_i\subset \mathcal {M}^\theta$  is said to be {\it convergent to} a space $\mathcal {X}\in \mathcal {M}^\theta$ {\it in the $\theta$-Gromov-Hausdorff topology} if
$\lim_{i\rightarrow \infty}d^\theta_{GH} (\mathcal {X}_i,\mathcal {X})=0$,
in
which case, $\mathcal {X}$ is called the {\it
$\theta$-Gromov-Hausdorff limit} of $(\mathcal {X}_i)_i$.
\end{definition}


\begin{theorem}\label{imporstcompacthteorem} Let $\mathcal {M}^{\theta}_\sim$ denote  the collection of isometric classes of compact $\theta$-metric spaces.
  Then
 $\left(\mathcal {M}^{\theta}_\sim,d^\theta_{GH}\right)$
is a  Polish space (i.e., a complete separable reversible metric space).
\end{theorem}
\begin{proof}

 Proposition \ref{Gromov-Haudsdiscpro} indicates that $\left(\mathcal {M}^{\theta}_\sim,d^\theta_{GH}\right)$ is a reversible metric space. The completeness follows from Zhao \cite[Theorem 2.4]{Z3}.
Hence, it remains to show the separability. Let $\mathcal {N}^\theta$ denote the collection  of  finite subsets of $\mathbb{N}$ with rational-valued distance and rational-valued reversibility in $[1,\theta]$. That is, if $(X,d)\in \mathcal {N}^\theta$, then $X$ is a finite subset of $\mathbb{N}$ satisfying

\smallskip

(1) $d_X(x,y)$ is rational for any $x,y\in X$;\quad  (2) $\lambda_d(X)$ is a rational number in $[1,\theta]$.

\smallskip

\noindent Clearly, $\mathcal {N}^\theta$ is a countable subset of $\left(\mathcal {M}^{\theta},d^\theta_{GH}\right)$. Now we show that $\mathcal {N}^\theta$ is dense.
For any $\mathcal {Y}:=(Y,d_Y)\in \mathcal {M}^\theta$, it is not hard to find a countable dense subset $S=\{x^i\}_{i\in \mathbb{N}}$ of $Y$ and a function $N:\mathbb{N}\rightarrow \mathbb{N}$ such that for each $k\in \mathbb{N}$, the first $N(k)$ points of $S$, say  $S^k:=\{x^1,\ldots,x^{N(k)}\}$, is a forward $1/k$-net of $Y$.
Fixing $k\in \mathbb{N}$, for each $\alpha\in \mathbb{N}$, it is not hard to find a $\theta$-metric space $(S^k_\alpha, d_{S^k_\alpha})$ such that
\[
S^k_\alpha:=\{1,\ldots,N(k)\},\  \lambda_{d_{S_\alpha^k}}(S_\alpha^k)\in [1,\theta]\cap \mathbb{Q}, \ d_{S_\alpha^k}(i,j)\in \mathbb{Q},\ |d_{S_\alpha^k}(i,j)-d_Y(x^i,x^j)|\leq\frac{1}{\alpha}, \,1\leq i,j\leq N(k).
\]
Consider the diagonal  sequence  $\mathcal {Y}_k:=(S^k_k,d_{S^k_k})\in \mathcal {N}^\theta$. A standard argument (cf. Petersen \cite[Example 56, p.295]{PP})
 yields
$d^{\theta}_{GH}(\mathcal {Y}_k,\mathcal {Y})\leq d^{\theta}_{GH}(\mathcal {Y}_k,S^k)+d^{\theta}_{GH}( S^k,\mathcal {Y})\leq \frac{2}{k}\rightarrow 0$,
which concludes the proof.
\end{proof}

In the following, we present several examples to show the advantage of the $\theta$-Gromov-Hausdorff topology. Note that every irreversible metric can be symmetrized to be a reversible metric and moreover, one can easily show the following result by Lemma \ref{inprotantisometry}.
\begin{proposition}\label{convergnoncmpamesa}
If $(X_i,d_i)_i$ converges to $(X,d)$ in the $\theta$-Gromov-Hausdorff topology, then the sequence of symmetrized spaces $(X_i,\hat{d}_i)_i$ converges to $(X,\hat{d})$ in the $1$-Gromov-Hausdorff topology $($i.e., the original Gromov-Hausdorff topology$).$
\end{proposition}

One may expect  to hold the reverse of the above statement and if so, to replace the $\theta$-Gromov-Hausdorff topology by the $1$-Gromov-Hausdorff topology in order to study the convergence.  Unfortunately, the following example shows that it is usually impossible  because the $1$-Gromov-Hausdorff topology is ``rougher" than the $\theta$-Gromov-Hausdorff topology.

\begin{example}\label{symmetricexample}
Let $\mathbb{T}^n=\mathbb{S}\times\cdots\times\mathbb{S}$ be an $n(\geq 2)$-flat tours. Define a sequence of Randers-Berwald metrics $F_i$ on $\mathbb{T}^n$ by
\[
F_i:=\alpha+\beta_i:=\left\{
\begin{array}{lll}
\sqrt{(dt^1)^2+\cdots+(dt^n)^2} + \frac12 dt^1, && \text{ if $i$ is even},\\
\\
\sqrt{(dt^1)^2+\cdots+(dt^n)^2} + \frac13 dt^1, && \text{ if $i$ is odd},
\end{array}
\right.
\]
where $\alpha$ is the standard Riemannian metric on $\mathbb{T}^n$ and $(t_1,\ldots,t_n)\in (0,2\pi)\times \cdots\times (0,2\pi)$ is a  local coordinate system of $\mathbb{T}^n$. In particular, $F_i$'s are globally defined on $\mathbb{T}^n$.

Let $d_i$ and $d_\alpha$ be the metrics induced by $F_i$ and $\alpha$, respectively.
 Note that the geodesics of $(\mathbb{T}^n,F_i)$ are the same as those of $(\mathbb{T}^n,\alpha)$, see e.g. Bao, Chern and Shen  \cite{BCS}, and the  symmetrized  metric of $d_i$, say $\hat{d}_i$, is exactly $d_\alpha$.
 Therefore, the sequence of symmetrized spaces $(\mathbb{T}^n,\hat{d}_i)_i$ converges to $(\mathbb{T}^n,d_\alpha)$ under the $1$-Gromov-Hausdorff topology.

 On the other hand, due to (\ref{Randersnormreveruniform}), each $(\mathbb{T}^n,d_i)$ is a compact $3$-metric space. However,
 for any $\theta\geq 3$, the sequence $(\mathbb{T}^n,{d}_i)_i$ is divergent under the $\theta$-Gromov-Hausdorff topology, since the odd and even sequences converge to different limits.
\end{example}

By Definition \ref{GromoHasdrdelta}, it is easy to check that different $\theta$-Gromov-Hausdorff topologies are compatible in the following sense.
\begin{proposition}
If a sequence $(\mathcal {X}_i)_i$ is convergent in the $\theta_1$-Gromov-Hausdorff topology, then it must converge to the same limit in the $\theta_2$-Gromov-Hausdorff topology for any $\theta_2\geq \theta_1$.
\end{proposition}
A natural question arises at this point: can we eliminate the  uniformly upper bound $\theta$ for  reversibilities? The following example presents a negative answer.
\begin{example}[A flaw Gromov-Hausdorff distance]\label{flawgroha}
Let $\mathcal {M}^\infty$ be a collection of compact forward metric space with finite reversibilities.
Let $\mathcal {X}_i:=(X_i,d_i)$, $i=1,2$ be two elements in $\mathcal {M}^\infty$. An $\infty$-admissible metric $d$  on the disjoint union $X_1\sqcup X_2$ is a metric with finite reversibility such that  $ d|_{X_i}=d_i, i=1,2$.
The $\infty$-{Gromov-Hausdorff} distance between
$\mathcal {X}_1$ and $\mathcal {X}_2$ is defined as
\[
d^\infty_{GH}(\mathcal {X}_1,\mathcal {X}_2):=\inf\{d_H(X_1,X_2)|\,\infty\text{-admissible
metrics on }X_1\sqcup X_2 \}.
\]
Although $d^\infty_{GH}$ satisfies all the properties in Proposition \ref{Gromov-Haudsdiscpro},
 the space $\left( \mathcal {M}^\infty, d_{GH}^\infty\right)$ is not complete. In fact, the limit of a Cauchy sequence in this space might be not  {an}  irreversible metric space.

For instance, consider  a sequence of Finsler metrics $F_i=\alpha+e^{-\frac{1}{i}}dt^1$, $i\in \mathbb{N},$ globally defined on an $n(\geq 2)$-flat torus $\mathbb{T}^n$,
where $\alpha$ and $t^1$ are  as in Example \ref{symmetricexample}.
Let $d_i$ denote the metric induced by $F_i$. Thus, $\mathcal {X}_i:=(\mathbb{T}^n,d_i)$, $i\in\mathbb{N}$ is a sequences of compact $\theta_i$-metric space, where
$\theta_i=({e^{{1}/{i}}+1})/({e^{{1}/{i}}-1})\rightarrow \infty$ as $i\to \infty$. Furthermore, the property of Berwald-Randers metrics implies
\[
|d_{i+1}(p,q)-d_{i}(p,q)|\leq 2\pi \left[ e^{-\frac{1}{i+1}}- e^{-\frac{1}{i}}\right]=:\epsilon_i, \text{ for any }p,q\in \mathbb{T}^n.
\]
Hence, we can define an $\infty$-admissible metric $d$ on $M_i\sqcup M_{i+1}$; more precisely, for $p\in M_i$ and $q\in M_{i+1}$,
\begin{align*}
d(p,q):=\underset{x\in \mathbb{T}^n}{\min} \left\{
d_i(p,x)+d_{i+1}(x,q)\right\}+\epsilon_i,\
d(q,p):=\underset{x\in \mathbb{T}^n}{\min}\left\{
d_{i+1}(q,x)+d_i(x,p)\right\}+\epsilon_i,
\end{align*}
which implies $d^\infty_{GH}(\mathcal {X}_i,\mathcal {X}_{i+1})\leq \epsilon_i$ and hence,
 $(\mathcal {X}_i)_i$ is a Cauchy sequence with respect to $d^\infty_{GH}$.
Obviously, $d_i\rightarrow d_\infty$, which is induced by $F_\infty=\alpha+dt^1$. However, $d_\infty$ is not a  metric because there exist two different points $p,q\in \mathbb{T}^n$ with $d_\infty(p,q)=0$ and $d_\infty(q,p)>0$.
\end{example}

From above, the $\theta$-Gromov-Hausdorff metric is an optimal tool to study the convergence of irreversible metric spaces. In particular, the collection of compact $\theta$-geodesic spaces is closed in $\left(\mathcal {M}^{\theta},d^\theta_{GH}\right)$ (cf. Shen and Zhao \cite[Theorem 5.2]{SZ}).
\begin{theorem}[\cite{SZ}]\label{lenghtclsoedghtopo}
The limit of a sequence of compact
 $\theta$-geodesic spaces under the $\theta$-Gromov-Hausdorff topology  is  a compact $\theta$-geodesic space as well.
\end{theorem}
Owing to Theorem \ref{geodesicexistencethe}, a compact
 $\theta$-length space is   a  compact $\theta$-geodesic space and vice versa. Hence, the collection of compact $\theta$-length spaces is also closed in $\left(\mathcal {M}^{\theta},d^\theta_{GH}\right)$.

 We recall a Gromov type compactness theorem (cf.   Zhao \cite[Theorem 2.6]{Z3}), which is an irreversible version of Petersen \cite[Corollary 8]{PP2}.

\begin{proposition}[\cite{Z3}]\label{capcityprecompact}Given a decreasing function $N:(0,1)\rightarrow \mathbb{N}$,
the collection
\[
\mathscr{C}(N):=\left\{\left.(X,d)\in \mathcal {M}^{\theta}\right|\Cov_X(\epsilon)\leq N(\epsilon), \text{ for any } \epsilon\in (0,1)\right\}
\]
is compact in the $\theta$-Gromov-Hausdorff topology.
\end{proposition}

Theorem \ref{Ohtacomparsiontheorem} together with (\ref{capvscov}), Proposition \ref{capcityprecompact} and (\ref{reversibdindentiy}) furnishes the following result immediately.
\begin{theorem}\label{basisprecomp}Given $N\in [2,\infty)$, $\theta\in [1,\infty)$, $D\in (0,\infty)$ and $K\in \mathbb{R}$, the collection of closed
 Finsler metric-measure manifolds $(M,F,\mathfrak{m})$ with
 \[
 \dim(M)\leq N,\ \lambda_{F}(M)\leq \theta,\ \mathbf{Ric}_N\geq K, \  \diam(M)\leq D
 \]
 is pre-compact in the $\theta$-Gromov-Hausdorff topology. Here, $\dim(M)$ denotes the dimension of $M$.
 \end{theorem}

\begin{remark}\label{precompactnessfail}The uniform upper bound $\theta$ in Theorem \ref{basisprecomp} is necessary. In view of Example  \ref{flawgroha}, the Finsler manifolds $(M_i,F_i):=(\mathbb{T}^n,F_i)$ endowed with the Holmes-Thompson measure satisfy $\dim=n$, $\mathbf{Ric}_N=0$, $\diam\leq 2(n+1)\pi$ but $\lambda_{F_i}(M_i)\nearrow \infty$ as $i\to \infty$.
\end{remark}

Following Lott and Villani \cite{LV}, we use $\epsilon$-isometry (see Definition \ref{epsilonapproximation}/(iii)) to study the convergence of forward metric-measure spaces. Firstly, we point out that every $\epsilon$-isometry admits an  approximate inverse.
\begin{proposition}\label{deltanprooxima} Let $(X,d_X)$ and $(Y,d_Y)$ be two compact $\theta$-metric spaces.
Given  an $\epsilon$-isometry  $f:X\rightarrow Y$, there exists a $(2+\theta)\epsilon$-isometry $f_r: Y\rightarrow X$. In particular, we have
\[
d_X(f_r\circ f(x),x)\leq 2\epsilon, \ d_Y(f\circ f_r(y),y)\leq \epsilon, \ \forall\,x\in X,\ y\in Y.
\]
\end{proposition}
\begin{proof}
Given any $y\in Y$, there exists $x_y\in X$ with $d_Y(f(x_y),y)\leq \epsilon$. Define $f_r: Y\rightarrow X$ by $f_r(y):=x_y$.
Then the triangle inequality yields
\begin{align*}
&|d_X(f_r(y), f_r(y'))-d_Y(y,y')|\\
\leq& |d_X(f_r(y), f_r(y'))-d_Y(f\circ f_r(y), f\circ f_r(y'))|+|d_Y(f\circ f_r(y), f\circ f_r(y'))-d_Y(y,y')|\\
\leq &\epsilon+\max\left\{d_Y(f\circ f_r(y), y)+d_Y(y',f\circ f_r(y')),\,d_Y(y,f\circ f_r(y))+d_Y(f\circ f_r(y'), y')\right\}\leq (2+\theta)\epsilon,
\end{align*}
which implies $\dis f_r\leq (2+\theta)\epsilon$.
A similar argument furnishes $d_X(f_r\circ f(x),x)\leq 2\epsilon$ for any $x\in X$, which indicates $X\subset \overline{[f_r(Y)]^{(2+\theta)\epsilon}}$. Hence, $f_r$ is a $(2+\theta)\epsilon$-isometry, which concludes the proof.
\end{proof}

The relation between $\epsilon$-isometry and the $\theta$-Gromov-Hausdorff convergence is as follows.

\begin{lemma}\label{inprotantisometry}
Given $\mathcal {X}=(X,d_X),\mathcal {Y}=(Y,d_Y)\in \mathcal {M}^\theta$, we have

\begin{itemize}
\item[(i)]  if $d_{GH}^\theta(\mathcal {X},\mathcal {Y})<\epsilon$, then there exists a $(1+\theta)\epsilon$-isometry $f:\mathcal {X}\rightarrow \mathcal {Y}$;

\smallskip

\item[(ii)] if there is an $\epsilon$-isometry $f:\mathcal {X}\rightarrow \mathcal {Y}$, then $d_{GH}^\theta(\mathcal {X},\mathcal {Y})\leq2\epsilon$.
\end{itemize}
\end{lemma}
\begin{proof} (i) Choose
 a $\theta$-admissible metric $d$ on $X\sqcup Y$ such that $d_H(Y,X)<\epsilon$.
By this fact, one can define a map $f:X\rightarrow Y$ with $d(f(x),x)<\epsilon$ for any $x\in X$. Then the triangle inequality yields
\begin{align*}
|d(f(x_1),f(x_2))-d(x_1,x_2)|<(1+\theta)\epsilon\Longrightarrow \dis f\leq (1+\theta)\epsilon.
\end{align*}
On the other hand, for any $y\in Y$, there exists $x_y\in X$ such that $d(x_y,y)<\epsilon$, which implies
\[
d(f(x_y),y)\leq d(f(x_y),x_y)+d(x_y,y)<2\epsilon\Longrightarrow Y\subset \overline{[f(X)]^{(1+\theta)\epsilon}}.
\]
Thus, $f:X\rightarrow Y$ is a $(1+\theta)\epsilon$-isometry.

\smallskip

(ii)  Given an $\epsilon$-isometry $f:\mathcal {X}\rightarrow \mathcal {Y}$, we can define a $\theta$-admissible metric $d$ on $X\sqcup Y$ as
\[
d(x,y):=\inf_{x'\in X}\left[ d(x,x')+d(f(x'),y) \right]+\epsilon,\ d(y,x):=\inf_{x'\in X}\left[ d(y,f(x'))+d(x',x) \right]+\epsilon.
\]
Since $Y\subset \overline{[f(X)]^\epsilon}$, it is easy to check that $X\subset \overline{Y^\epsilon}$ and $Y\subset \overline{X^{2\epsilon}}$ in $(X\sqcup Y,d)$. Thus, Statement (ii) follows from
$d_{GH}^\theta(\mathcal {X},\mathcal {Y})\leq d_H(X,Y)\leq 2\epsilon$.
\end{proof}


Since the  convergence of  compact $\theta$-metric spaces can be equivalently defined by  $\epsilon$-isometries,
we can introduce the following definition.
\begin{definition}\label{measureconvergence} Let $(X_i,d_i,\mu_i)$, $i\in \mathbb{N}$ and $(X,d,\mu)$ be compact $\theta$-metric-measure spaces. We say that  $(X_i,d_i,\mu_i)_i$ {\it converges to}
$(X,d,\mu)$ {\it in the measured  $\theta$-Gromov-Hausdorff topology} if there are $\epsilon_i$-isometries $f_i:X_i\rightarrow X$, which are Borel maps, such that $\lim_{i\rightarrow \infty}\epsilon_i=0$ and
$\lim_{i\rightarrow \infty}(f_i)_\sharp \mu_i=\mu$
in the weak topology of measures.
\end{definition}

An equivalent convergence   is defined by the Gromov-Hausdorff-Prokhorov distance. See Theorem \ref{GHPequvialecnn} below.
On the other hand, the globally doubling property is preserved in the  measured  $\theta$-Gromov-Hausdorff topology, whose proof will be given in Subsection \ref{GHPTOPOLOGY}.
\begin{lemma}\label{DoublingcontrollconverinGromovHameameaser}
Let $(X_i,d_i,\mu_i)_i$ be a sequence of compact $\theta$-metric-measure spaces converging to a compact $\theta$-metric-measure space $(X_\infty,d_\infty,\mu_\infty)$ in the measured $\theta$-Gromov-Hausdorff topology. If each $\mu_i$ is globally doubling for a uniform constant $L$, so is $\mu_\infty$.
\end{lemma}


\begin{theorem}\label{compactmeasuredgromvhaus}
Given $D,L>0$ and $0<m\leq M$,  let $\mathscr{C}$ be a collection of compact $\theta$-metric-measure spaces $\mathcal {X}=(X,d,\mu)$ with
\[
\diam(\mathcal {X})\leq D,\ \Do(\mathcal {X})\leq L,\ m\leq \mu[\mathcal {X}]\leq M.
\]
Then $\mathscr{C}$ is pre-compact in the measured $\theta$-Gromov-Hausdorff topology. In particular, any weak cluster space $(X_\infty,d_\infty,\mu_\infty)$ satisfies $\supp\mu_\infty=X_\infty$, where $\supp\mu_\infty$ is the support of $\mu_\infty$.
\end{theorem}
\begin{proof}
Choose an arbitrary sequence $(X_i,d_i,\mu_i)_i$ in $\mathscr{C}$. Owing to Lemma \ref{doublingwithprecompact}, Proposition \ref{capcityprecompact} and Lemma \ref{inprotantisometry}, by passing a subsequence, we can assume
 \begin{itemize}
 \item[(1)]  $(X_i,d_i)_i$ converges to a compact $\theta$-metric space $(X_\infty,d_\infty)$ in the $\theta$-Gromov-Hausdorff topology;

 \item[(2)]  there exists a sequence of $\epsilon_i$-isometries $f_i:X_i\rightarrow X_\infty$ with $\lim_{i\rightarrow \infty}\epsilon_i=0$;

 \item[(3)] $(\mu[X_i])_i$ is  convergent to a positive constant $\Xi\in [m,M]$.

 \end{itemize}
Set $\nu_i:=\mu_i/\mu_i[X_i]$. By Proposition \ref{ProtheoremGromHauss}, we may suppose that $({f_i}_\sharp)\nu_i$ converges weakly to a probability measure $\nu\in P(X_\infty)$. Thus, one gets
\[
\lim_{i\rightarrow \infty}\int_X g{\ddd}({f_i}_\sharp\mu_i)=\lim_{i\rightarrow \infty}\mu_i[X_i]\int_X g{\ddd}({f_i}_\sharp\nu_i)=\Xi\lim_{i\rightarrow \infty}\int_X g{\ddd}({f_i}_\sharp\nu_i)= \int_X g{\ddd}(\Xi\nu),\ \forall\,g\in C(X_\infty),
\]
which implies ${f_i}_\sharp\mu_i\rightarrow \Xi\nu=:\mu_\infty$ in the weak topology.

It remains to show $\supp\mu_\infty=X_\infty$. Suppose by contradiction that $\mathcal {O}:=X_\infty\backslash \supp \mu_\infty\neq\emptyset$.
Since $\mathcal {O}$ is a non-empty open subset, we can choose $x\in \mathcal {O}$ and a small $r>0$ such that $B^+_x(r)\subset \mathcal {O}$. Obviously,
\[
\overline{B^+_x(r/2)}=\{y\in X:\, d(x,y)\leq r/2\}\subset B^+_x(r)\subset \mathcal {O}.
\]
Owing to $\mathcal {O}\subset X_\infty\backslash \text{Spt}\mu_\infty$, we have $\mu_\infty[\mathcal {O}]=0$ and hence, $\mu_\infty\left[\overline{B^+_x(r/2)}\right]=0$.
 On the other hand, it follows from Lemma \ref{DoublingcontrollconverinGromovHameameaser} that $\mu_\infty$ is $L$-doubling. By choosing $N:=\lceil\log_2\frac{2D}{r} \rceil$ we have
 \[
0<\Xi=\mu_\infty[X_\infty]=\mu_\infty\left[  \overline{B^+_x(D)} \right]\leq L^N \cdot\mu_\infty\left[\overline{B^+_x(r/2)}\right]=0,
\]
which is a contradiction. Therefore, $\supp\mu_\infty=X_\infty$.
\end{proof}

\begin{theorem}\label{RiccompactprecompactFinser}Let $\mathfrak{m}$ denote either the Busemann-Hausdorff measure or the Holmes-Thompson measure. Thus, for any
$N\in [2,\infty)$, $\theta\in [1,\infty)$, $D\in (0,\infty)$ and $K\in \mathbb{R}$, the collection of closed
 Finsler metric-measure manifolds $(M,F,\mathfrak{m})$ with
 \[
 \dim(M)\leq N,\ \Lambda_{F}(M)\leq \theta^2,\ \mathbf{Ric}\geq K, \  \diam(M)\leq D\tag{3.1}\label{Ricnormmeasureprecompact},
 \]
 is pre-compact in the measured $\theta$-Gromov-Hausdorff topology.
\end{theorem}
\begin{proof}
Let $(M_i,F_i,\mathfrak{m}_i)_i$ be   a sequence  satisfying (\ref{Ricnormmeasureprecompact}). Theorem \ref{normalRiccicompari} implies
\[
\mathfrak{m}_i[M_i]\leq \vol(\mathbb{S}^{N-1})\,\theta^{2N}\int^D_0\mathfrak{s}_{-|K|,N}^{N-1}(t){\ddd}t,\quad \Do(M_i)\leq 2^N \theta^{4N}e^{D\sqrt{(N-1)|K|}}.
\]
If $\inf_i\mathfrak{m}_i[M_i]>0$, then the theorem follows from  Theorem \ref{compactmeasuredgromvhaus} and (\ref{reversibdindentiy}) directly. If $\inf_i\mathfrak{m}_i[M_i]=0$, by a modification to the proof of Theorem \ref{compactmeasuredgromvhaus}, we may assume that $(M_i,d_{F_i})_i$   converges to a compact $\theta$-metric space $(X_\infty, d_\infty)$ by means of $\epsilon_i$-isometries $f_i:M_i\rightarrow X_\infty$ with $\lim_{i\rightarrow \infty}\epsilon_i=0$ and  $\lim_{i\rightarrow \infty}\mathfrak{m}_i[M_i]=0$. Thus,  for any $g\in C(X_\infty)$, we have
\[
\lim_{i\rightarrow \infty}\left|\int_{X_\infty} g{\ddd}({f_i}_\sharp\mathfrak{m}_i)\right|\leq\max_{x\in X_\infty}|g(x)|\lim_{i\rightarrow \infty}\mathfrak{m}_i[M_i]=0,
\]
which implies that $(f_i)_{\sharp}\mathfrak{m}_i$ converges weakly to a null measure $\mathfrak{m}_\infty$. That is, $(M_i,d_{F_i},\mathfrak{m_i})$ converges to $(X_\infty, d_\infty,\mathfrak{m}_\infty)$ in the measured $\theta$-Gromov-Hausdorff topology.
\end{proof}

A similar argument together with Theorem \ref{Ohtacomparsiontheorem} yields the following result.
\begin{theorem}\label{RicNmeasurecompact}
Given $D,V\in (0,\infty)$, $\theta\in [1,\infty)$, $N\in [2,\infty)$ and $K\in \mathbb{R}$, the collection of closed
 Finsler metric-measure manifolds $(M,F,\mathfrak{m})$ with
 \[
 \dim(M)\leq N,\ \lambda_{F}(M)\leq \theta,\ \mathbf{Ric}_N\geq K, \  \diam(M)\leq D,\ \mathfrak{m}[M] \leq V
 \]
 is pre-compact  in the measured $\theta$-Gromov-Hausdorff topology.
\end{theorem}

\subsection{Gromov-Hausdorff topology II: noncompact spaces}

We begin with the following example.
\begin{example}\label{preex} Let $\mathbf{e}$ be a fixed unit vector in $\mathbb{R}^n$ and define a sequence of Randers metrics $F_i=\alpha+\beta_i$
 on the open unit ball $\mathbb{B}^n\subset \mathbb{R}^n$ by
\[
\alpha:=\frac{\sqrt{\|y\|^2-(\|x\|^2\|y\|^2-\langle x,y\rangle^2)}}{1-\|x\|^2},\quad \beta_i:=\frac{\langle x,y\rangle}{1-\|x\|^2}+\frac{\langle a_i,y\rangle}{1+\langle a_i,x\rangle},\quad a_i=\frac{\mathbf{e}}{i^2+1}.
\]
Set $\mathbb{B}_\mathbf{0}(r):=\{x\in \mathbb{R}^n|\,\|x\|<r\}$. A direct  calculation together with (\ref{Randersnormreveruniform}) yields
\[
\lambda_{F_i}\left(\overline{\mathbb{B}_\mathbf{0}(r)}\right)= \frac{1+\sqrt{1-\frac{(1-{r}^2)(1-\|a_i\|^2)}{(1+\|a_i\| {r})^2}}}{1-\sqrt{1-\frac{(1-{r}^2)(1-\|a_i\|^2)}{(1+\|a_i\| {r})^2}}}\leq \frac{1-\sqrt{1-\frac{3(1-{r}^2)}{(2+{r})^2}}}{1-\sqrt{1+\frac{3(1-{r}^2)}{(2+{r})^2}}}=:\Upsilon(r),\ 0<r<1.
\]
Since $F_i$ is projectively flat, we have $\mathbb{B}_\mathbf{0}\left( \frac{e^r-1}{e^r+\|a_i\|} \right)\subset B^+_{\mathbf{0}}(r)\subset \mathbb{B}_\mathbf{0}\left( \frac{e^r-1}{e^r-\|a_i\|}\right)$, which implies
\[
\lambda_{F_i}\left( \overline{B^+_\mathbf{0}(r)}\right)\leq\Upsilon\left(  \frac{e^r-1}{e^r-1/2} \right)  =:\Theta(r).\tag{3.2}\label{new3.1}
\]
Thus,   Theorem \ref{reversibifinslerdd} implies that for each $i$, $\mathcal {X}_i:=(X_i,\star_i,d_i):=(\mathbb{B}_\mathbf{0}(1),\mathbf{0},d_{F_i})$ is a pointed forward $\Theta$-metric/geodesic space.
Despite $\lambda_{d_i}(X_i)=\infty$, it is easy to see that $d_i\rightarrow d_\infty$, where $d_\infty$ is the metric induced by $F_\infty=\alpha+\beta_\infty$ with $\beta_\infty=\frac{\langle x,y\rangle}{1-\|x\|^2}$. In particular,   $(\mathbb{B}_\mathbf{0}(1),\mathbf{0},d_\infty)$ is also a forward pointed $\Theta$-metric/geodesic space.
\end{example}

Inspired by this example,
a sequence of noncompact forward metric spaces could approach a ``limit  space" even if their reversibilities are infinite. Hence, we introduce the following definition, which serves as an irreversible version of Burago, Burago and Ivanov \cite[Definition 8.1.1]{DYS}.

\begin{definition}\label{noncomapctGHdef} Given a nondecreasing function $\Theta:(0,\infty)\rightarrow [1,\infty)$, let $\mathcal {M}^{\Theta}_*$ be the {\it collection of forward boundedly compact pointed forward $\Theta$-metric spaces}.
A sequence $(X_i,\star_i,d_i)_i \subset \mathcal {M}^{\Theta}_*$ is said to
  {\it converge to} $(X,\star,d)\in\mathcal {M}^{\Theta}_*$ {\it in the  forward $\Theta$-Gromov-Hausdorff topology} if   for every $r > 0$ and $\epsilon> 0$, there exists a natural number $N=N(r,\epsilon)$ such that for
every $i > N$ there is a (not necessarily continuous) map $f_i : \overline{B^+_{\star_i}(r)}\rightarrow X$
satisfying
\[
f_i(\star_i)=\star, \ \ \dis f_i\leq \epsilon, \ \ \overline{B^+_\star(r-\epsilon)}\subset \left[ f_i\left( \overline{B^+_{\star_i}(r)} \right) \right]^\epsilon.\tag{3.3}\label{noncfincazse}
\]
In particular, $(X,\star,d)$ is called the  {\it forward $\Theta$-Gromov-Hausdorff limit} of  $(X_i,\star_i,d_i)_i$.
\end{definition}

\begin{remark}\label{remarkcompleteseparted}
 If $(X,\star,d)\in \mathcal {M}^{\Theta}_*$, then it is separable and forward complete (see Proposition \ref{forwardandbackwardrelation}/(iii)).
Owing to Theorem \ref{topologychara}/(ii), the corresponding  symmetrized space $(X,\star,\hat{d})$ is a locally compact Polish space. 
\end{remark}

\begin{remark}   $f_i : \overline{B^+_{\star_i}(r)}\rightarrow X$ in Definition \ref{noncomapctGHdef} is similar to a pointed $\epsilon$-isometry. In fact,
\[
\overline{B^+_\star(r-\epsilon)}\subset \left[ f_i\left( \overline{B^+_{\star_i}(r)} \right) \right]^\epsilon\subset\overline{B^+_\star(r+2\epsilon)}.\tag{3.4}\label{noncompactballconverge}
\]
\end{remark}

 Definition \ref{noncomapctGHdef} is a very weak way to define the convergence of noncompact spaces (compared with Theorem \ref{newnoncp} below).
However, the following result shows that Definition \ref{noncomapctGHdef} is well-defined and compatible with the compact case, whose proof will be given in Appendix \ref{appendixGro-Hauscon}.
\begin{proposition}\label{welldefinednoncompactGHCONVER} Suppose that $(X_i,\star_i,d_i)$, $i\in \mathbb{N}$, $(X,\star,d)$ and $(X,'\star',d')$ belong to $\mathcal {M}^{\Theta}_*$. The following statements are true:

\begin{itemize}
\item[(i)] Let $(X,\star,d)$ and $(X,'\star',d')$ be two forward $\Theta$-Gromov-Hausdorff limits of a sequence $(X_i,\star_i,d_i)_i$. Then $(X,\star,d)$ is pointed isometric to $(X,'\star',d')$.

\item[(ii)] If $(X_i,\star_i,d_i)_i$ converges in the forward $\Theta$-Gromov-Hausdorff topology, then it converges to the same limit in  the forward ${\Theta'}$-Gromov-Hausdorff sense for any nondecreasing function $\Theta'\geq \Theta$.

\item[(iii)] If $(X_i,\star_i,d_i)_i$ converges to $(X,\star,d)$ in the forward $\Theta$-Gromov-Hausdorff topology and satisfies $\diam(X_i)\leq D$, then $(X_i,d_i)_i$ converges to $(X,d)$ in the $\theta$-Gromov-Hausdorff topology for any $\theta\geq \Theta(D)$.

\smallskip

\item[(iv)] If $(X_i,d_i)_i\subset \mathcal {M}^\theta$ converges to $(X,d)\in \mathcal {M}^\theta$ in the $\theta$-Gromov-Hausdorff topology, then there exist points $\star_i\in X_i$ and $\star\in X$ such that $(X_i,\star_i,d_i)_i$ converges to $(X,\star,d)$ in the forward $\Theta$-Gromov-Hausdorff topology for any nondecreasing function $\Theta\geq \theta$.
\end{itemize}
\end{proposition}

Now we continue to study the convergence in the model Example \ref{preex}.

\begin{proposition}\label{preex2}
Let $ (X_i,\star_i,d_i):=(\mathbb{B}^n,\mathbf{0},d_i)$ and $(X_\infty,\star_\infty,d_\infty):=(\mathbb{B}^n,\mathbf{0},d_\infty)$ be as in Example \ref{preex} and let $\Theta(r)$  be as in \eqref{new3.1}. Then  $ (X_i,\star_i,d_i)_i$ converges to $(X_\infty,\star_\infty,d_\infty)$ in the forward $\Theta$-Gromov-Hausdorff topology.
\end{proposition}
\begin{proof}
Since the metrics $F_i$ and $F_\infty$ are projectively flat,  a direct calculation yields
\[
\left| d_i(p,q)-d_\infty(p,q)  \right|  \leq  \frac{2}{i^2},\ \forall\,p,q\in \mathbb{B}^n.\tag{3.5}\label{dismodefunk}
\]
Moreover, for each $r>0$, we have
\[
\overline{\mathbb{B}_\textbf{0}\left( \frac{e^r-1}{e^r+\|a_i\|}  \right)}\subset\overline{B^+_{\star_i}(r)},\quad   \overline{B^+_{\star_\infty}(r)}=\overline{\mathbb{B}_\textbf{0}\left( \frac{e^r-1}{e^r}  \right)}.\tag{3.6}\label{dis2modefunk}
\]
Now consider the identity map $f_i=\id|_{\overline{B^+_{\star_i}(r)}}:\overline{B^+_{\star_i}(r)}\rightarrow X_\infty$. Thus, (\ref{dismodefunk}) and (\ref{dis2modefunk}) imply
\[
f_i(\star_i)=\star_\infty, \quad \dis f_i\leq  \frac{2}{i^2}, \quad  \overline{B^+_{\star_\infty}(r-\epsilon_i)} \subset \left[f_i\left( \overline{B^+_{\star_i}(r)} \right)\right]^{\epsilon_i}, \quad \epsilon_i=\ln\left( \frac{i^2+2}{i^2+1} \right),
\]
which concludes the proof.
\end{proof}

It is well-known that every tangent space is a tangent cone of a Riemannian manifold (cf. Gromov \cite{GLP}). This result remains valid in the Finsler setting. More precisely,
let $T_\star M$ be the tangent space of a Finsler manifold $(M,F)$ at $\star$. The distance from   $y_1\in T_\star M$ to $y_2\in T_\star M$ is defined by $F_\star(y_1,y_2):=F(\star,y_2-y_1)$. Then
$(T_\star M,0,F_\star)$ is a forward boundedly compact pointed $\lambda_F(\star)$-metric space. Now we have the following result.

\begin{proposition}\label{tanggespacecone}
Let $(M,F)$ be a forward complete Finsler manifold and let $\star$ be a fixed point in $M$. For each $k\in \mathbb{N}$, set $(X_k,\star_k, d_k):=(M,\star,k\cdot d_F)$. Thus
$(X_k,\star_k, d_k)_k$ converges to $(T_\star M,0,F_\star)$ in the  forward $\Theta$-Gromov-Hausdorff topology, where $\Theta(r):=\lambda_F(B^+_\star(r))$.
\end{proposition}
\begin{proof} Set $\mathcal {B}^+_0(r):=\{y\in T_\star M|\, F_\star(0,y)<r\}$, $B^+_\star(r):=\{x\in M|\, d_F(\star,x)<r\}$ and $B^+_{\star_k}(r):=\{x\in X_k|\, d_k(\star,x)<r\}$. Choose
a small $R_0>0$ such that $\exp_\star:\overline{\mathcal {B}^+_0(R_0)}\rightarrow \overline{B^+_\star(R_0)}$ is a $C^1$-coordinate system (cf. Bao, Chern and Shen \cite[p.126]{BCS}).
According to Busemann and Mayer \cite[p.186]{BM}, we have the uniform convergence
\[
\lim_{k\rightarrow\infty} d_k\left( \exp_\star\left( \frac{y_1}{k} \right),\exp_\star\left( \frac{y_2}{k} \right) \right)=F_\star(y_1, y_2),\ \forall \,y_1,y_2\in \overline{\mathcal {B}^+_0(R_0)}.
\]
Thus, for any $r>0$ and any $\epsilon>0$, there exists $N=N(r,\epsilon)$ such that for any $k>N$,
\[
\overline{ {B}^+_{\star_k}(r)}\subset \overline{B^+_\star(R_0)}, \  \left| F_\star (f_k(x_1), f_k(x_2))-d_k(x_1,x_2)\right|<\epsilon,\ \forall\, x_1,x_2\in \overline{{B}^+_{\star_k}(r)},
\]
where $f_k: \overline{{B}^+_{\star_k}(r)}\rightarrow T_\star M$ is a map defined by $f_k(x):=k\cdot\exp^{-1}_\star(x)$. Clearly, $f_k(\star)=0$. Moreover, for any $y\in \overline{ \mathcal {B}^+_0(r-\epsilon)}$, we have $f_k(x_y)=y$, where
$x_y:=\exp_\star(\frac{y}{k})\in \overline{{B}^+_{\star_k}(r)}$, which implies $\overline{\mathcal {B}^+_0(r-\epsilon)}\subset f_k\left(\overline{{B}^+_{\star_k}(r)}  \right)$. Thus, the result immediately follows by Definition \ref{noncomapctGHdef}.
\end{proof}

Moreover, we also have a Gromov type pre-compactness theorem in the noncompact case, whose proof is given in Appendix \ref{appendixGro-Hauscon}.
\begin{theorem}\label{noncompactprecompact}
Let $\mathscr{C}\subset \mathcal {M}^{\Theta}_*$
be a class satisfying
for every $r>0$ and every $\epsilon>0$, there exists a natural number
$N=N(r,\epsilon)$ such that $\Cov_{\overline{B^+_\star(r)}}(\epsilon)\leq N(r,\epsilon)$ for all
$(X,\star,d)\in \mathscr{C}$.
 Then $\mathscr{C}$ is pre-compact in the forward $\Theta$-Gromov-Hausdorff topology.
\end{theorem}

Moreover,
the collection of forward geodesic spaces is still closed under such a convergence, whose proof is postponed to Appendix \ref{appendixGro-Hauscon}.

\begin{proposition}\label{lenthcompactnoncompact}
 Suppose that  $(X_i,\star_i,d_i)_i\subset \mathcal {M}^\Theta_\star$ converges to
$(X,\star,d)\in \mathcal {M}^\Theta_\star$  in the pointed
forward $\Theta$-Gromov-Hausdorff topology. If each $(X_i,\star_i,d_i)$ is a forward geodesic space, so is $(X,\star,d)$.
\end{proposition}

\begin{remark}\label{geodeisclenghequv}
In view of Theorems \ref{geodesicexistencethe} \& \ref{HopfRinowth}, a   forward boundedly compact forward geodesic  space is always
  a  forward complete locally compact forward length space and vice versa. Hence, the above result also implies the compactness of the collection of forward length spaces in $\mathcal {M}^\Theta_\star$.
\end{remark}

In what follows, we use another method to study the convergence of pointed spaces.

\begin{definition}
Let $\mathcal {X}_i:=(X_i,\star_i,d_i)$, $i=1,2$ be two compact pointed $\theta$-metric spaces. The {\it pointed
$\theta$-Gromov-Hausdorff distance}  between $\mathcal {X}_i$, $i=1,2$ is
defined by
\begin{align*}
d^{\,\theta}_{pGH}(\mathcal {X}_1,\mathcal {X}_2):=\inf\left\{\left.d_H(X_1,X_2)+\frac{d(\star_1,\star_2)+d(\star_2,\star_1)}{2}\right|\,\theta\text{-admissible
metrics on }X_1\sqcup X_2 \right\}.
\end{align*}
\end{definition}

By the proof of Lemma \ref{inprotantisometry}/(i) and Proposition \ref{Ascloghtipo}, one can show the following result.
\begin{proposition}
 The pointed
$\theta$-Gromov-Hausdorff distance is a reversible pseudo-metric on the collection of compact pointed $\theta$-metric spaces. In particular, $d^{\,\theta}_{pGH}(\mathcal {X}_1,\mathcal {X}_2)=0$ if and only if they are pointed isometric.
\end{proposition}

\begin{theorem}\label{newnoncp}
Suppose that $(X_i,\star_i,d_i)_i$ and  $(X,\star,d)$ belong to $\mathcal {M}^{\Theta}_*$.
The following statements hold:
\begin{itemize}
\item[(i)] The sequence $(X_i,\star_i,d_i)_i$ converges to
$(X,\star,d)$  in the pointed
forward $\Theta$-Gromov-Hausdorff topology if
for all $r>0$,
\[
\lim_{i\rightarrow\infty}d^{\,\Theta(r)}_{pGH}\left((\overline{B^+_{\star_i}(r)},\star_i),(\overline{B^+_{\star}(r)},\star)\right)=0.\tag{3.8}\label{lengtmarked}
\]

\item[(ii)]
Suppose that $(X_i,\star_i,d_i)$'s are forward geodesic spaces. Then  $(X_i,\star_i,d_i)_i$ converges to
$(X,\star,d)$  in the pointed
forward $\Theta$-Gromov-Hausdorff topology if and only if   (\ref{lengtmarked}) is valid
for all $r>0$.
\end{itemize}
\end{theorem}
\begin{proof} Statement (i) directly  follows by the proof of Lemma \ref{inprotantisometry}/(i). We now show (ii); since the ``$\Leftarrow$" part follows by (i), it remains to prove the ``$\Rightarrow$" part.
Using the same argument of  Lemma \ref{inprotantisometry}/(ii), for every $r > 0$ and $\epsilon> 0$, there exists a natural $N=N(r,\epsilon)$ such that for
every $i > N$, we have
\[
d^{\Theta(r)}_{pGH}\left( \overline{B^+_{\star_i}(r)}, \overline{B^+_\star(r-\epsilon)} \right)<2\epsilon.
\]
On the other hand, Proposition \ref{lenthcompactnoncompact} implies that $(X,d)$ is a forward geodesic  space and hence,
\[
d^{\Theta(r)}_{pGH}\left( \overline{B^+_\star(r-\epsilon)}, \overline{B^+_\star(r)}\right)\leq \epsilon,
\]
which together with the triangle inequality for $d^{\Theta(r)}_{p GH}$ concludes the proof.
\end{proof}

Note that (\ref{lengtmarked}) is  a sufficient but not necessary condition if the reference spaces are not forward geodesic spaces. However,
in the reversible case, it is frequently used to define the convergence in many references (cf. \cite{PP,Vi}) just for convenience.
For the same reason, we utilize (\ref{lengtmarked}) to define the convergence of pointed forward metric-measure spaces.




\begin{definition}\label{noncompactconvergedefi}
Let $\mathcal {X}_i:=(X_i,\star_i,d_i,\mu_i)$, $i\in \mathbb{N}$ and $\mathcal {X}:=(X,\star,d,\mu)$ be  forward boundedly compact pointed forward $\Theta$-metric-measure spaces.  The sequence  $(\mathcal {X}_i)_i$ is said to {\it converge to $\mathcal {X}$ in the pointed measured forward $\Theta$-Gromov-Hausdorff topology} if there are sequences $R_i\rightarrow \infty$ and $\epsilon_i\rightarrow 0$ and measured pointed $\epsilon_i$-isometrics $f_i:\overline{B^+_{\star_i}(R_i)}\rightarrow  \overline{B^+_{\star}(R_i)}$ such that $(f_i)_\sharp  \mu_i\rightarrow \mu$, where the convergence is in the weak-$*$ topology (i.e., convergence against compactly supported continuous test functions).
\end{definition}
\begin{remark}Owing to Theorem \ref{newnoncp}/(i), the pointed measured forward $\Theta$-Gromov-Hausdorff convergence indicates the pointed forward $\Theta$-Gromov-Hausdorff convergence.
\end{remark}

The Cantor diagonal argument together with Theorem \ref{compactmeasuredgromvhaus}, Lemma  \ref{inprotantisometry} and Proposition \ref{noncompactmesurecon} yields the following result.
\begin{theorem}\label{noncomapctdoubling} Given two constants $0<m \leq M$ and a positive function $L(r)$,
let  $\mathscr{C}$ be a collection of pointed forward $\Theta$-metric-measure spaces $(X,\star,d,\mu)\in \mathcal {M}^{\Theta}_*$ satisfying
\[
 m\leq \mu\left[\overline{B^+_\star(1)}\right]\leq M,\quad \Do \left(\overline{B^+_\star(r)}\right)\leq L(r), \ \forall\,r>0.
\]
Then $\mathscr{C}$ is pre-compact in the pointed measured forward $\Theta$-Gromov-Hausdorff topology. In particular, any weak cluster space $(X_\infty,d_\infty,\mu_\infty)$ satisfies $\supp\mu_\infty=X_\infty$.
\end{theorem}

Similar to the compact case (see Theorems \ref{RiccompactprecompactFinser}-\ref{RicNmeasurecompact}), one can easily derive the following result from Theorem \ref{noncomapctdoubling}.
\begin{theorem}
Given any $N\in [2,\infty)$, $V\in (0,\infty)$ and $K\in \mathbb{R}$, the following classes are pre-compact in the pointed measured forward $\Theta$-Gromov-Hausdorff topology:

\begin{itemize}
\item[(1)] the collection of pointed forward complete Finsler metric-measure manifolds $(M,\star,F,\mathfrak{m})$ with
\[
\dim(M)\leq N,\quad \mathbf{Ric}\geq K,\quad  \Lambda_F({B^+_\star(r)})\leq \Theta^2(r),\ \forall\,r>0,
\]
and $\mathfrak{m}$ is either the Busemann-Hausdorff measure or the Holmes-Thompson measure;

\item[(2)] the collection of pointed forward complete
 Finsler metric-measure manifolds $(M,\star,F,\mathfrak{m})$ with
 \[
 \dim(M)\leq N,\quad \mathbf{Ric}_N\geq K,\quad  \mathfrak{m}\left[{B^+_\star(1)}\right] \leq V,\quad \lambda_{F}({B^+_\star(r)})\leq \Theta(r),\   \forall\,r>0.
 \]
 \end{itemize}
\end{theorem}

\subsection{Gromov-Hausdorff-Prokhorov topology}\label{GHPTOPOLOGY}
In this subsection, we investigate the convergence of forward metric-measure spaces by another equivalent method.
\begin{definition}[Prokhorov distance]
Given a forward metric space $(X,d)$,  let   $\mathfrak{M}(X)$ denote   the collection of (totally) finite
Borel measures on $X$.
For any  $\mu,\nu\in \mathfrak{M}(X)$, define the {\it Prokhorov distance} between them as
\[
d_P(\mu,\nu):=\inf\left\{ \epsilon>0|\, \mu(A)\leq \nu(A^\epsilon)+\epsilon,\  \nu(A)\leq \mu(A^\epsilon)+\epsilon,\ \forall\,A\subset X \text{ is closed} \right\}.
\]
\end{definition}

A standard argument (cf.\,Daley and Vere-Jones \cite[pp.\,398]{DV}) combined with
Theorem \ref{topologychara}/(ii) yields the following result.
\begin{proposition}\label{measureweakconvergence}
Let $(X,d)$ be a forward complete separable forward metric space. Thus,
\begin{itemize}
	\item[(i)]  $(\mathfrak{M}(X),d_P)$ is a complete separable reversible metric space;

\item[(ii)]  a sequence $(\mu_i)_i\subset \mathfrak{M}(X)$ satisfies $d_P(\mu_i,\mu)\rightarrow 0$ if and only if  $\mu_i$ converges weakly to $\mu$.
\end{itemize}
\end{proposition}

\begin{proposition}\label{dghpproperties}
Let $\mathscr{M}^\theta$ be the collection of  compact $\theta$-metric-measure spaces. Given $\mathcal {X}_i:=(X_i,d_i,\mu_i)\in \mathscr{M}^\theta$, $i=1,2$, the {\it $\theta$-Gromov-Hausdorff-Prokhorov distance} between them is defined by
\[
d_{GHP}^\theta(\mathcal {X}_1,\mathcal {X}_2):=\inf\left\{ d_H(X_1,X_2)+d_P\left((\Phi_1)_{\sharp}\mu_1,(\Phi_2)_{\sharp}\mu_2\right)| \,\text{$\theta$-admissible
metrics on }X_1\sqcup X_2 \right\},
\]
where $\Phi_i: X_i\rightarrow X_1\sqcup X_2$, $i=1,2$ are the natural isometric embeddings. Then we have:
\begin{itemize}
\item[(i)] $d_{GHP}^\theta$ is a reversible pseudo-metric on $\mathscr{M}^\theta$, i.e., $d_{GHP}^\theta$ is non-negative, reversible and satisfies the triangle inequality;

\item[(ii)] if $\mathcal {X}_i:=(X_i,d_i,\mu_i)\in\mathscr{M}^\theta $, $i=1,2$ satisfies $d_{GHP}^\theta(\mathcal {X}_1,\mathcal {X}_2)=0$ if and only if there exists an isometry $f:X_1\rightarrow X_2$ with $f_\sharp \mu_1=\mu_2$.

\end{itemize}
\end{proposition}
\begin{proof}
Since  (i) is easy to verify, we only focus to  (ii).  If $d_{GHP}^\theta(\mathcal {X}_1,\mathcal {X}_2)=0$,
 there exist  a sequence $(\epsilon_\alpha)_\alpha$ and a sequence of $\theta$-admissible metrics $(d^\alpha)_\alpha$ on $X_1\sqcup X_2$ such that
 \[
 d^\alpha_H(X_1,X_2)+d^\alpha_P\left((\Phi_1)_{\sharp}\mu_1,(\Phi_2)_{\sharp}\mu_2\right)<\epsilon_\alpha\rightarrow 0, \text{ as }\alpha\rightarrow \infty.\tag{3.9}\label{tiangleGPH}
 \]
 For each $\alpha$, the proof of Lemma \ref{inprotantisometry}/(i) furnishes  a  $(1+\theta)\epsilon_\alpha$-isometry $f_\alpha:X_1\rightarrow X_2$ with $d^\alpha(f_\alpha(x),x)\leq \epsilon_\alpha$ for any $x\in X_1$, which implies
  \[
 d^\alpha_P\left((\Phi_2\circ f_\alpha)_\sharp\mu_1,(\Phi_1)_\sharp\mu_1\right)\leq C(\theta)\,\epsilon_\alpha,
 \]
 where $C(\theta)$ is a finite constant only dependent on $\theta$. This inequality combined with
 (\ref{tiangleGPH}) yields
$\lim_{\alpha\rightarrow\infty}(f_\alpha)_\sharp \mu_1=\mu_2$ (see Proposition \ref{measureweakconvergence}/(ii)). On the other hand, by Proposition \ref{Ascloghtipo}, one can assume that $f_\alpha$ converges uniformly to an isometry $f:X_1\rightarrow X_2$ and hence, $f_\sharp\mu_1=\mu_2$.
\end{proof}

\begin{theorem}\label{GHPequvialecnn}
Let $\mathcal {X}_i:=(X_i,d_i,\mu_i)$, $i\in \mathbb{N}$ and $\mathcal {X}:=(X,d,\mu)$ be compact $\theta$-metric-measure spaces. Then $(\mathcal {X}_i)_i$ converges to
$\mathcal {X}$ in the measured  $\theta$-Gromov-Hausdorff topology if and only if
\[
\lim_{i\rightarrow \infty}d_{GHP}^\theta(\mathcal {X}_i,\mathcal {X})=0.
\]
\end{theorem}
\begin{proof}
In view of Theorem \ref{imporstcompacthteorem}, Lemma \ref{inprotantisometry}/(i) and Proposition \ref{ProtheoremGromHauss}, one can prove
 the ``$\Leftarrow$"  part by a modification of the proof of Proposition \ref{dghpproperties}. Hence, we only show the
``$\Rightarrow$" part.
Let $f_i:X_i\rightarrow X$ be a sequence of $\epsilon_i$-isometries satisfying $\lim_{i\rightarrow\infty}\epsilon_i=0$ and $\lim_{i\rightarrow\infty}(f_i)_\sharp\mu_i=\mu$. By repeating the proof of Lemma \ref{inprotantisometry}/(ii), one can define a $\theta$-admissible metric $d^i$ on $X_i\sqcup X$ such that
\[
d^i_H(X_i,X)\leq 2\epsilon_i,\ d^i(f_i(x),x)=\epsilon_i, \ \forall\,x\in X_i.\tag{3.10}\label{approximiateballtocontrol}
\]
Let $\Phi_i:X_i\hookrightarrow X_i\sqcup X$ and $\Phi:X\hookrightarrow X_i\sqcup X$ denote the natural isometric embeddings. Thus, the second equality in (\ref{approximiateballtocontrol}) yields
\begin{align*}
d^i_{P}\left((\Phi_i)_\sharp\mu_i,(\Phi\circ f_i)_\sharp\mu_i\right)\leq C(\theta)\,\epsilon_i,\tag{3.11}\label{orginalmeasureandpushmeasure}
\end{align*}
where $C(\theta)$ is a finite constant only dependent on $\theta$.
Since $(f_i)_{\sharp}\mu_i\rightarrow \mu$ weakly, for each $\delta>0$, there exists $I>0$ such that
for $i>I$, $d^i_{P}(\Phi_\sharp({f_i})_\sharp \mu_i,\Phi_\sharp\mu)<\delta$, which
together with (\ref{orginalmeasureandpushmeasure}) and (\ref{approximiateballtocontrol})  yields  $\lim_{i\rightarrow \infty}d^\theta_{GHP}(\mathcal {X}_i,\mathcal {X})= 0$.
\end{proof}

Define a equivalent relation $\sim$ on $\mathscr{M}^\theta $ by
$(X_1,d_1,\mu_1)\sim (X_2,d_2,\mu_2)$ if and only if there exists an isometry $f:X_1\rightarrow X_2$ with $f_\sharp\mu_1=\mu_2$.
Let $\mathscr{M}^\theta_\sim$ be the   quotient set of $\mathscr{M}^\theta$ by $\sim$.
\begin{theorem}\label{polishtheorem}
 $(\mathscr{M}^\theta_\sim,d^\theta_{GHP})$ is a Polish space.
\end{theorem}
\begin{proof}[Sketch of the proof]   Proposition \ref{dghpproperties} implies that $(\mathscr{M}^\theta_\sim,d^\theta_{GHP})$ is a reversible metric space. Moreover, thanks to  Theorem \ref{GHPequvialecnn}, Theorem \ref{imporstcompacthteorem} and
 Proposition \ref{ProtheoremGromHauss}, it is not hard to check the completeness.
 It remains to show that $(\mathscr{M}^\theta_\sim,d^\theta_{GHP})$ is separated. Given $\mathcal {Y}=(Y,d,\mu)\in \mathscr{M}^\theta$, for each $k\in \mathbb{N}$, one can choose a finite forward $1/k$-net $S^k:=\{x^i\}_{i=1}^{N(k)}$ and a sequence of Borel subsets $\{B^i\}_{i=1}^{N(k)}\subset Y$ such that $x^i\in B^i$, $Y=\sqcup_iB^i$ and $\diam(B^i)\leq k^{-1}$ for all $i$ (cf.\,Villani \cite[p.105]{Vi}).

Define a finite Borel measure $\mu^k$ on $S^k$ by $\mu^k:=\sum_{i=1}^{N(k)} \mu(B^i)\, \delta_{x^i}$, where $\delta_{x^i}$ is  the Dirac measure at $x^i$. Thus, $d_P(\mu^k,\mu)<1/k$.
On the other hand, given fixed $k\in \mathbb{N}$, for each $\alpha\in \mathbb{N}$, define a compact $\theta$-metric-measure space $(S^k_\alpha, d_{S^k_\alpha},\mu^k_\alpha)$ such that $(S^k_\alpha, d_{S^k_\alpha})$ is the one defined in the proof of Theorem \ref{imporstcompacthteorem} and
\[
\mu^k_\alpha=\sum_{i=1}^{N(k)} b_i\, \delta_{x^i},\ b_i\in \mathbb{Q}, \ \left|b_i-\mu(B^i)\right|<\frac1{\alpha\, N(k)}, \text{ for }1\leq i\leq N(k).
\]
By considering the diagonal sequence $\mathcal {Y}_k:=(S^k_k, d_{S^k_k},\mu^k_k)$, one has $\lim_{k\rightarrow \infty}d^\theta_{GHP}(\mathcal {Y},\mathcal {Y}_k)=0$.
\end{proof}

\begin{lemma}\label{Doublingcontrollconver}
Let $\mathcal {X}_i:=(X_i,d_i,\mu_i)$, $i\in \mathbb{N}$ be a sequence in   $\mathscr{M}^\theta$ converging to a space $\mathcal {X}_\infty:=(X_\infty,d_\infty,\mu_\infty)$ $\in \mathscr{M}^\theta$ in the $\theta$-Gromov-Hausdorff-Prokhorov sense. If there exists a constant $L>0$ such that $\Do(\mathcal {X}_i)\leq L$ for all $i$, then  $\Do(\mathcal {X}_\infty)\leq L$ as well.
\end{lemma}
\begin{proof}By passing a subsequence,  for each $i\in \mathbb{N}$, there is a $\theta$-admissible $d^i$ on $X_i\sqcup X_\infty$ such that
\[
d^i_H(X_i,X_\infty)+d^i_P((\Phi_{i})_\sharp\mu_i,(\Phi_{\infty})_\sharp\mu_\infty)\leq \epsilon_i,\ \lim_{i\rightarrow \infty}\epsilon_i=0,
\]
where $\Phi_{i}:X_i\rightarrow X_i\sqcup X_\infty$ and $\Phi_{\infty}:X_\infty\rightarrow X_i\sqcup X_\infty$ are the isometric embeddings.
 Given $x\in X_\infty$, choose $x_i\in X_i$ with $d^i(x_i,x)\leq\epsilon_i$. Thus, for any $r>0$ we have
 \begin{align*}
 &\mu_\infty\left[\overline{B^+_x(2r)}\cap X_\infty\right]=(\Phi_\infty)_\sharp \mu_\infty\left[\overline{B^+_x(2r)}\right]\leq  (\Phi_i)_\sharp \mu_i\left[\overline{B^+_{x_i}(2r+2\epsilon_i)}\right]+\epsilon_i\\
 =& \mu_i\left[\overline{B^+_{x_i}(2r+2\epsilon_i)}\cap X_i\right]+\epsilon_i\leq L \,\mu_i\left[\overline{B^+_{x_i}(r+\epsilon_i)}\cap X_i\right]+\epsilon_i=L \,(\Phi_i)_\sharp\mu_i\left[\overline{B^+_{x_i}(r+\epsilon_i)}\right]+\epsilon_i\\
 \leq &L\, (\Phi_i)_\sharp\mu_i\left[\overline{B^+_{x}(r+(1+\theta)\epsilon_i)}\right]+\epsilon_i\leq L (\Phi_\infty)_\sharp\mu_\infty\left[\overline{B^+_{x}(r+(2+\theta)\epsilon_i)}\right]+(1+L)\epsilon_i\\
 =& L \, \mu_\infty\left[\overline{B^+_{x}(r+(2+\theta)\epsilon_i)}\cap X_\infty\right]+(1+L)\epsilon_i,
 \end{align*}
 which implies $\Do(\mathcal {X}_\infty)\leq L$.
\end{proof}

\begin{proof}[Proof of Lemma \ref{DoublingcontrollconverinGromovHameameaser}]
The statement immediately follows by Theorem \ref{GHPequvialecnn} and Lemma \ref{Doublingcontrollconver}.
\end{proof}



By a method similar to the one used in Abraham, Delmas and Hoscheit \cite{ADH} (or Khezeli \cite{AK}), one can extend the Gromov-Hausdorff-Prokhorov distance to pointed forward $\Theta$-metric-measure spaces. In particular,
the corresponding  topology is equivalent to the measured forward $\Theta$-Gromov-Hausdorff topology. We leave the formulation of such statements to the
interested reader.



\section{Optimal transport on forward metric spaces}\label{optimaltarsn-0}

\subsection{Wasserstein spaces}
\begin{definition}
Let $(X,d)$ be a forward boundedly compact forward metric space and let $P(X)$ denote the collection of Borel probability measures.
Given $p\in [1,\infty)$ and $\mu,\nu\in P(X)$, the {\it Wasserstein distance of order $p$ from $\mu$ to $\nu$} is defined as
\[
W_p(\mu,\nu):=\left(  \inf_{\pi\in \Pi(\mu,\nu) } \int_{X\times X} d(x,y)^p {\ddd}{\pi}(x,y)                  \right)^{\frac1p},
\]
where $\Pi(\mu,\nu)$ is the collection of transference plans/couplings from $\mu$ to $\nu$.
\end{definition}

\begin{theorem}\label{lowercontinuouspfWp} Let $(X,d)$ be a forward boundedly compact forward metric space. Given $p\in [1,\infty)$ and $\mu,\nu\in P(X)$, there exists a coupling $\tilde{\pi}\in \Pi(\mu,\nu)$ such that
\[
W_p(\mu,\nu)=\left(\int_{X\times X} d(x,y)^p {\ddd}\tilde{\pi}(x,y)                  \right)^{\frac1p}.
\]
Such a $\tilde{\pi}$ is called an optimal transference plan (or optimal coupling) from $\mu$ to $\nu$ (with respect to $d^p$).
\end{theorem}
\begin{proof} Theorem \ref{topologychara}/(ii) shows that
 the symmetrized space $(X,\hat{d})$  is a Polish space. Since $d^{\,p}:X\times X\rightarrow \mathbb{R}$ is a continuous function under the product topology of $(X,\hat{d})$ (i.e., Theorem \ref{topologychara}), the statement directly follows by Villani \cite[Theorem 4.1]{Vi}.
\end{proof}

A standard argument (cf.\,Villani \cite[Theorem 7.3]{V}) together with the H\"older inequality furnishes the following result immediately.
\begin{theorem}\label{wassdistacbasisthe}Given a
  forward boundedly compact pointed forward metric space $(X,\star,d)$, define
\[
P_p(X):=\left\{ \mu\in P(X):\, \int_X \hat{d}(\star,x)^p {\ddd}\mu(x)<\infty  \right\}.
\]
Then $P_p(X)$ is independent of the choice of $\star$. Moreover,
$(P_p(X),W_p)$ is an  irreversible metric space, i.e., for any $\mu,\nu,\upsilon\in P_p(X)$,
\begin{itemize}
\item[(i)]  $W_p(\mu,\nu)$ is finite$;$

\item[(ii)] $W_p(\mu,\nu)\geq 0,$ with equality if and only if $\mu=\nu;$

\item[(iii)] $W_p(\mu,\nu)\leq W_p(\mu,\upsilon)+W_p(\upsilon,\nu);$

\end{itemize}
Furthermore, for any $1\leq q\leq p$, there hold $W_q\leq W_p$ and  $P_p(X)\subset P_q(X)$.
\end{theorem}


In the sequel, we set $\Theta^\infty:=\sup_{x\in X}\Theta(d(\star,x))$; moreover,  $\Theta^\infty$ is said to be {\it concave} if it is finite.

\begin{lemma}\label{reversbilityofW}
Given $1\leq q\leq p<\infty$ and $(X,\star,d)\in \mathcal {M}^\Theta_*$, if $\Theta^{\frac{qp}{p-q}}$ is a concave function, then
\[
 W_{q}(\nu,\mu)\leq  \Theta\bigg(  W_p(\delta_\star,\mu)+W_p(\mu,\nu) \bigg)\,W_p(\mu,\nu),\ \forall\,\mu,\nu\in P_p(X).
 \]
\end{lemma}
\begin{proof}
Let $\tilde{\pi}$ be an optimal transference plan from $\mu$ to $\nu$ with respect to $d^{\,p}$ and set $\breve{\pi}({\ddd}y{\ddd}x):=\tilde{\pi}({\ddd}x{\ddd}y)$. The H\"older inequality together with Theorem \ref{wassdistacbasisthe} then yields
\begin{align*}
\int_{X\times X}\left[d(\star,x)+d(x,y)\right] {\ddd}\tilde{\pi}(x,y)
\leq W_1(\delta_\star,\mu)+\left(\int_{X\times X}d(x,y)^p {\ddd}\tilde{\pi}(x,y)\right)^{1/p}\leq W_p(\delta_\star,\mu)+W_p(\mu,\nu).
\end{align*}

On the other hand,
since $d(y,x)\leq \Theta(d(\star,x)+d(x,y))\,d(x,y)$ for all $x,y\in X$,  Jensen's inequality combined with the above inequality yields
\begin{eqnarray*}
W_{q}(\nu,\mu)^q&\leq& \int_{X\times X}d(y,x)^{q}{\ddd}\breve{\pi}(y,x)\leq \int_{X\times X} \Theta^q(d(\star,x)+d(x,y))\,d(x,y)^q\, {\ddd}\tilde{\pi}(x,y)\\&
\leq &\left( \int_{X\times X}\Theta^{\frac{qp}{p-q}}(d(\star,x)+d(x,y)) {\ddd}\tilde{\pi}(x,y)\right)^\frac{p-q}{p}\left(\int_{X\times X}d^p(x,y) {\ddd}\tilde{\pi}(x,y)\right)^{\frac qp}\\
&\leq &\left( \Theta^{\frac{qp}{p-q}}\left(\int_{X\times X}\left[d(\star,x)+d(x,y)\right]{\ddd}\tilde{\pi}(x,y)\right) \right)^\frac{p-q}{p}  W_p(\mu,\nu)^q\\&\leq& \Theta^q\Big( W_p(\delta_\star,\mu)+W_p(\mu,\nu)\Big)\,W_p(\mu,\nu)^q,
\end{eqnarray*}
which concludes the proof.
\end{proof}

 In the sequel we always endow $P_p(X)$ with the forward topology $\mathscr{T}_p$ induced by $W_p$.

\begin{theorem}\label{wasstherforqs}
Given $1\leq q\leq p<\infty$ and $(X,\star,d)\in \mathcal {M}^\Theta_*$, if $\Theta^{\frac{qp}{p-q}}$ is a concave function, then
 \begin{itemize}

 \item[(i)] $(P_p(X),W_p)$ is a separable irreversible metric space. Moreover, any probability measure can be approximated by a sequence of probability measures with finite support (w.r.t. $\mathscr{T}_p$);

\item[(ii)] any forward Cauchy sequence in  $(P_p(X),W_p)$ is convergent in $(P_s(X),W_s)$ for all $s\in [1,q];$

\item[(iii)] for all $s\in [1,q]$, $W_s$ is continuous on $(P_p(X),W_p)$, i.e., if $(\mu_k)_k\subset P_p(X)$ (resp., $(\nu_k)_k\subset P_p(X)$) converges to $\mu\in P_p(X)$ (resp., $\nu\in P_p(X)$) under $\mathscr{T}_p$, then
\[
\lim_{k,l\rightarrow \infty}W_s(\mu_k,\nu_l)=W_s(\mu,\nu).
\]
\end{itemize}
\end{theorem}
\begin{proof}(i) The proof is similar to that of Villani \cite[Theorem 6.18]{Vi} and hence, we just give the sketch.
Due to Theorem \ref{wassdistacbasisthe}, it suffices to prove the separability.  Let
$\mathcal {D}$ be a dense sequence in $X$ (since $X$ is separable), and let $\mathcal {P}$ be the collection of probability measures that can be written as $\sum_j b_j \delta_{x_j}$, $b_j\in \mathbb{Q}$ and the $x_j$'s are finitely many elements in $\mathcal {D}$. Now we show that $ \mathcal {P}$ is dense in $P_p(X)$, i.e., given any $\mu\in P_p(X)$, for any $\epsilon\in (0,1)$, there is  $\nu\in \mathcal {P}$ such that
$W_p( \mu, \nu)\leq\epsilon$.

For convenience, let $\epsilon:=\varepsilon/(1+2^{1/p})$.
Since $\int_X\hat{d}(\star,x)^p{\ddd}\mu(x)$ is finite, there exists a compact set $K\subset X$ such that
$\int_{X\setminus K} d(x,\star)^p{\ddd}\mu(x)\leq \varepsilon^p$. The compactness of  $\overline{K^1}=\{x\in X|\,d(K,x)\leq 1\}$ implies $\theta:=\sup_{x\in \overline{K^1}}  \Theta(d(\star,x))<\infty$. Owing to the density of $\mathcal {D}$, one can
cover $K$ by a finite family of forward balls $B^+_{x_k}(\varepsilon/\theta)$, $1\leq k\leq N$, with centers $x_k\in \mathcal {D}\cap \overline{K^1}$. Set
\[
B'_k:=B^+_{x_k}(\varepsilon/\theta)\backslash \bigcup_{j<k}B^+_{x_j}(\varepsilon/\theta).
\]
and define a function $f$ on $X$ by $f\left( B'_k\cap K \right):=\{x_k\}$ and $f\left( X\backslash K \right):=\{x_0\}:=\{\star\}$. Thus, for any $x\in K$, we have $f(x)\in K^1$ and hence, $d(x,f(x))\leq \theta d(f(x),x)<\varepsilon$.
Then an easy argument as in Villani \cite[p.105]{Vi} yields  $W_p(\mu,f_\sharp\mu)\leq 2^{1/p}\varepsilon$.
Since $f_\sharp \mu=\sum_{k=0}^N a_k\delta_{x_k}$,
we see that $\mu$ can be approximated by a finite combination of Dirac masses.
On the other hand, Theorem \ref{totallvacontwp} yields
\begin{align*}
W_p\left( f_\sharp \mu,\, \sum_{k=0}^Nb_k\delta_{x_k} \right)=W_p\left( \sum_{k=0}^N a_k\delta_{x_k},\, \sum_{k=0}^Nb_k\delta_{x_k} \right)
\leq 4\max_{0\leq k\leq N}\hat{d}(x_0,x_k)\left(\sum_{k=0}^N|a_k-b_k|^{\frac1p}\right).
\end{align*}
Hence, we choose $b_k\in \mathbb{Q}$ such that $|a_k-b_k|$ is small enough, then $W_p( f_\sharp \mu,\, \sum_{k=0}^Nb_k\delta_{x_k} )\leq \varepsilon$. By choosing $\nu:=\sum_{k=0}^Nb_k\delta_{x_k}$, we obtain $W_p(\mu,\nu)\leq W_p(\mu,f_\sharp \mu,)+W_p(f_\sharp \mu,\nu)\leq \epsilon$.

\smallskip

 (ii) The assumption implies that $\Theta^{\frac{sp}{p-s}}$ is  concave for any $s\in [1,q]$.
Let $(\mu_k)_k$ be a forward Cauchy sequence in $P_p(X)$. Corollary \ref{backwardconvergence} yields a   measure $\mu\in P(X)$ with $\lim_{k\rightarrow \infty}W_p(\mu_k,\mu)=0$. Hence, for any $\varepsilon\in (0,1)$, there exists $N=N(\varepsilon)>0$ such that for any $k>N$, $W_p(\mu_{N+1},\mu_k)<\varepsilon$ and $W_p(\mu_k,\mu)<\varepsilon$.
 Moreover, Theorem \ref{wassdistacbasisthe} and Lemma \ref{reversbilityofW} furnish $W_s(\mu_k,\mu)\leq W_p(\mu_k,\mu)$ and
\begin{align*}
W_s(\mu,\mu_k)&\leq \Theta \left( W_p(\delta_\star,\mu_{N+1})+W_p(\mu_{N+1},\mu_k)+W_p(\mu_k,\mu) \right)\,W_p(\mu_k,\mu)\\
&\leq \Theta \left( W_p(\delta_\star,\mu_{N+1})+2 \right)\,W_p(\mu_k,\mu).
\end{align*}
Therefore, we have
\[
\lim_{k\rightarrow\infty}W_s(\mu_k,\mu)=0=\lim_{k\rightarrow\infty}W_s(\mu,\mu_k).\tag{4.1}\label{5.5}
\]
Recall $(\mu_k)_k\subset P_s(X)$ (see Theorem \ref{wassdistacbasisthe}). Thus, by the triangle inequality for $W_s$ and (\ref{5.5}), one can easily check that both $W_s(\mu,\delta_\star)$ and $W_s(\delta_\star,\mu)$ are finite, which implies $\int_X \hat{d}(\star,x)^s {\ddd}\mu(x)<\infty$, i.e.,
 $\mu\in P_s(X)$.

 \smallskip

 (iii)
  Let $(\mu_k)_k, (\nu_k)_k$ be two sequences in $P_p(X)$ with $W_p(\mu,\mu_k)\rightarrow 0$ and $W_p(\nu,\nu_k)\rightarrow 0$. Since $\mu,\nu\in P_p(X)$, a similar argument  to that of  (\ref{5.5}) yields
 \[
 \lim_{k\rightarrow\infty}W_s(\mu_k,\mu)=0=\lim_{k\rightarrow\infty}W_s(\mu,\mu_k),\ \lim_{k\rightarrow\infty}W_s(\nu_k,\nu)=0=\lim_{k\rightarrow\infty}W_s(\nu,\nu_k).
 \]
 The proof is complete by using the triangle inequality for $W_s$.
 \end{proof}

\begin{corollary}\label{compactwasserin}
Suppose that $(X, d)$ is a forward boundedly compact $\theta$-metric space. Thus, for any $p\in [1,\infty)$, the following statements hold:
\begin{itemize}
\item[(i)] $(P_p(X),W_p)$ is a separable complete $\theta$-metric space;

\item[(ii)] if $(X,d)$ is compact, then   $(P_p(X),W_p)=(P(X),W_p)$ is a compact $\theta$-metric space satisfying $\diam(P(X),W_p)=\diam(X,d)$, in which case $i:(X,d)\rightarrow (P(X),W_p)$, $x \mapsto \delta_x$ is an isometric embedding.

\end{itemize}
\end{corollary}
\begin{proof} Statement (i) directly follows by  Theorem \ref{wasstherforqs} and Lemma \ref{reversbilityofW}.
As for (ii), obviously $P_p(X)=P(X)$. Now we show $(P(X),W_p)$ is forward totally bounded, i.e., for any $\epsilon>0$, there exists a finite forward $\epsilon$-net of $(P(X),W_p)$. In order to do this, choose a finite forward $\varepsilon/\theta$-net $(x_k)_{k=1}^N$ of $X$, where $\varepsilon:=\epsilon/(1+2^{1/p})$. Now we define a finite subset $\mathcal {D}$ of $P(X)$ by
  \[
  \mathcal {D}:=\left\{\left.\sum_{k=1}^N b_k \delta_{x_k}\right|\, b_k\in\left\{0,\eta,2\eta,\ldots,\lceil 1/\eta\rceil \eta \right\}, \ \sum_{k=1}^Nb_k=1\right\},\quad {\eta:=\left(\frac{\varepsilon}{4\theta N\diam(X,d)}\right)^p.}
  \]
 An  argument similar to that of Theorem \ref{wasstherforqs}/(i) yields that $\mathcal {D}$ is
   a finite forward $\epsilon$-net   of $P(X)$. Thus, the compactness of $(P(X),W_p)$ follows from Theorem \ref{compactequvitheorem}/(iii). Moreover, it follows from $W_p(\delta_x,\delta_y)=d(x,y)$ that $\diam(P(X),W_p)\geq \diam(X,d)$ and $i:(X,d)\rightarrow (P(X),W_p)$ is an isometric embedding. On the other hand,
 for any $\mu,\nu\in P(X)$, it is easy to verify $W_p(\mu,\nu)\leq \diam(X,d)$,
which implies $\diam(P(X),W_p)=\diam(X,d)$.
\end{proof}

The same (simpler) argument as in Villani \cite[Theorem 7.12]{V} yields the following result.

\begin{proposition}\label{weaktopologyvswstatopo}
Let $(X,d)$ be a compact $\theta$-metric space. Then the weak topology of $P(X)$ coincides with the topology induced by the Wasserstein metric $W_p$ for any $p\in [1,\infty)$; equivalently, given $(\mu_k)_k\subset P(X)$ and $\mu\in P(X)$, $W_p(\mu,\mu_k)\rightarrow 0$ as $k\rightarrow \infty$ if and only if $(\mu_k)_k$ converges weakly to $\mu$.
\end{proposition}

\subsection{Displacement interpolation}
In this subsection, $(X,d)$ always denotes a   forward boundedly compact forward geodesic  space, while $C([s,t];X)$ denotes the set of curves from $[s,t]$ to $X$, where $0\leq s<t\leq 1$. And $C([s,t];X)$ is equipped with the uniform topology (see Definition \ref{unformconverge}).

In particular, owing to Remark \ref{geodeisclenghequv} and Theorem \ref{geodesicexistencethe}, our assumption on the reference space $(X,d)$ coincides with the usual assumption in the reversible case (cf.\,Villani \cite[Section 7]{Vi}).



\begin{theorem}\label{turespapcelagranceaction}
Let $(X,\star,d)\in \mathcal {M}^\Theta_*$  be a   pointed forward geodesic space. Given $p\in [1,\infty)$, for any $0\leq s<t\leq 1$, define a  functional $\mathcal {A}^{s,t}_p$ on  $C([s,t];X)$ as
\[
\mathcal {A}^{s,t}_p(\gamma):=\sup_{N\in \mathbb{N}}\sup_{s\leq t_0\leq t_1\leq \cdots \leq t_N=t}\sum_{i=0}^{N-1}\frac{d(\gamma(t_i),\gamma(t_{i+1}))^p}{(t_{i+1}-t_i)^{p-1}},
\]
and define a function $c^{s,t}$ on $X\times X$ as
\[
c^{s,t}(x,y)=\frac{d(x,y)^p}{(t-s)^{p-1}}.
\]
Then $ \mathcal {A}^{0,1}_p$ is a Lagrangian action and $c^{s,t}$'s are the responding cost functions,
i.e.,
\begin{itemize}
\item[(i)] $ \mathcal {A}^{s,t}_p$'s are lower semicontinuous (w.r.t. the  uniform convergence)  and satisfy
\[
 \mathcal {A}^{t_1,t_2}_p+\mathcal {A}^{t_2,t_3}_p=\mathcal {A}^{t_1,t_3}_p,\  \forall\,0\leq t_1< t_2< t_3\leq 1;\tag{4.2}\label{addivity}
\]

\item[(ii)] for any $x,y\in X$, there holds
\[
c^{s,t}(x,y):=\inf\left\{\mathcal {A}^{s,t}(\gamma)\,|\, \gamma\in C([s,t];X) \text{ with } \gamma(s)=x,\,\gamma(t)=y\right\}.
\]
In particular, a curve $\gamma\in C([s,t];X)$ is an  action-minimizing curve $($i.e., it minimizes $\mathcal {A}^{s,t}_p$ among all curves having the same endpoints$)$ if and only if
\begin{itemize}
\item it is a minimal geodesic when $p=1;$

\item  it is a constant-speed minimal geodesic when $p>1;$
\end{itemize}

\smallskip

\item[(iii)] for any curve $\gamma\in C([s,t];X)$, there holds
\[
\mathcal {A}^{s,t}_p(\gamma)=\sup_{N\in \mathbb{N}}\sup_{s=t_0\leq t_1\leq \cdots \leq t_N=t} \sum_{i=0}^{N-1} c^{t_i,t_{i+1}}(\gamma(t_i),\gamma(t_{i+1})).
\]

\end{itemize}
Moreover,  the action is coercive, i.e.,
\begin{itemize}
\item[(a)] $\inf_{s<t}\inf_{\gamma\in C([s,t];X)}\mathcal {A}^{s,t}(\gamma)>-\infty$;

\smallskip

\item[(b)] if $s<t$ are two intermediate times, and $K_s,K_t\subset X$ are any compact sets such that $c^{s,t}(x,y)<\infty$ for all $x\in K_s, y\in K_t$, then $\Gamma^{s,t}_{K_s\rightarrow K_t}$, the set of action-minimizing curves staring in $K_s$ at time $s$ and ending in $K_t$ at time $t$, is compact and nonempty in the uniform topology.
\end{itemize}
\end{theorem}
\begin{proof} (iii) is obvious.
  We first show
(ii). Given any $\gamma\in C([s,t];X)$, for any partition $\{t_i\}_{i=0}^N$ of $[s,t]$,
the convexity of $|\cdot|^p$ together with the triangle inequality yields
\begin{align*}
\left(\frac{d(\gamma(s),\gamma(t))}{t-s}\right)^p&\leq \left(\sum_{i=0}^{N-1}\frac{t_{i+1}-t_i}{t-s}\,\frac{d(\gamma(t_i),\gamma(t_{i+1}))}{t_{i+1}-t_i}\right)^p\leq \sum_{i=0}^{N-1} \frac{t_{i+1}-t_i}{t-s}\left(\frac{d(\gamma(t_i),\gamma(t_{i+1}))}{t_{i+1}-t_i}\right)^p\\
&=\frac{1}{t-s}\sum_{i=0}^{N-1}\frac{d(\gamma(t_i),\gamma(t_{i+1}))^p}{(t_{i+1}-t_i)^{p-1}}\leq \frac{\mathcal {A}^{s,t}_p(\gamma)}{t-s},\tag{4.3}\label{realtrinC}
\end{align*}
which  implies that the action-minimizing curve is a constant-speed minimal geodesic if $p>1$ (resp., a minimal geodesic if $p=1$). Since such a curve always exists,
 $\inf{\mathcal {A}^{s,t}_p(\gamma)}=\frac{d(x,y)^p}{(t-s)^{p-1}}=c^{s,t}(x,y)$, where the infimum is taken over all curves $\gamma\in C([s,t];X)$ with $\gamma(s)=x$ and $\gamma(t)=y$.   Thus  (ii) follows.  In particular,  (\ref{realtrinC}) implies
\[
c^{s,t}(\gamma(t_0),\gamma(t_N))\leq \sum_{i=0}^{N-1} c^{t_i,t_{i+1}}(\gamma(t_i),\gamma(t_{i+1})).\tag{4.4}\label{triangleineqlC}
\]

Now we show (i).  Let $(\gamma_\alpha)_\alpha\subset C([s,t];X)$ be a sequence of paths convergent uniformly to $\gamma\in C([s,t];X)$. Set $\theta:=\max_{s\leq \tau\leq t}\Theta\big(d(\star,\gamma(\tau))+1\big)<\infty$.
Given $\varepsilon\in (0,1)$, there exists $J_1=J_1(\varepsilon)$ such that $\rho(\gamma,\gamma_\alpha)=\max_{s\leq \tau\leq t}d(\gamma(\tau),\gamma_\alpha(\tau))<\varepsilon$  for any $\alpha>J_1$.

Firstly we suppose $\mathcal {A}^{s,t}_p(\gamma)<\infty$. Thus
  (iii) yields  a partition $Y=\{t_i\}_{i=0}^N$ of $[s,t]$ for $\gamma$ such that $t_i\neq t_{i+1}$ and $\mathcal {A}^{s,t}_p(\gamma)-\Sigma^p(Y)<\varepsilon$, where $\Sigma^p(Y):=\sum_{i=0}^{N-1} c^{t_i,t_{i+1}}(\gamma(t_i),\gamma(t_{i+1}))$. Also set
$L:=\max_{0\leq i\leq N-1}d(\gamma(t_i),\gamma(t_{i+1}))<\infty$.
Since $Y$ is fixed, there exists $J_2=J_2(\varepsilon)(>J_1)$ such that
\[
\rho(\gamma,\gamma_\alpha)=\max_{s\leq \tau\leq t}d(\gamma(\tau),\gamma_\alpha(\tau))<\varepsilon \frac{\min_{1\leq i\leq N-1}(t_{i+1}-t_i)^{p-1}}{2^{p}N(1+\theta)(L+(1+{\theta}))^{p-1}}=:\epsilon,\ \text{ for any }\alpha>J_2.\tag{4.5}\label{contrallmesp}
\]
Since $d(\gamma_\alpha(\tau),\gamma(\tau))\leq \theta d(\gamma(\tau),\gamma_\alpha(\tau))$ for $\tau\in[s,t]$ and $\alpha>J_2$, the triangle inequality yields
\[
  d(\gamma_\alpha(t_i),\gamma_\alpha(t_{i+1}))\leq d(\gamma(t_i),\gamma(t_{i+1}))+(1+{\theta})\epsilon\leq L+(1+{\theta})\epsilon=:M<\infty,\ \forall \,0\leq i\leq N.\tag{4.6}\label{contralldrj}
  \]
Let   $\Sigma^p_\alpha(Y):=\sum_{i=0}^{N-1} c^{t_i,t_{i+1}}(\gamma_\alpha(t_i),\gamma_\alpha(t_{i+1}))$. Owing to $\epsilon\leq\varepsilon\ll L<M$,
    it follows from the triangle inequality, (\ref{contralldrj}) and (\ref{contrallmesp}) that for  $\alpha>J_2$,
\begin{align*}
\mathcal {A}^{s,t}_p(\gamma)\leq \Sigma^p(Y)+\varepsilon\leq {\Sigma}^p_\alpha(Y)+\sum_{i=0}^{N-1} \frac{2^pM^{p-1}(1+{\theta})\epsilon}{(t_{i+1}-t_i)^{p-1}}+\varepsilon\leq \mathcal {A}^{s,t}_p(\gamma_\alpha)+2\varepsilon,
\end{align*}
which implies  the lower semicontinuity of $\mathcal {A}^{s,t}_p$. Secondly, if $\mathcal {A}^{s,t}_p(\gamma)=\infty$,  choose $Y$ with $\Sigma^p(Y)>1/\varepsilon$. Then the same argument as above yields
$\mathcal {A}^{s,t}_p(\gamma_\alpha)\geq \Sigma^p_\alpha(Y)\geq 1/\varepsilon-\varepsilon$. Therefore, $\mathcal {A}^{s,t}_p$ is always lower semicontinuous, while (\ref{addivity}) follows directly by   (\ref{triangleineqlC}) and (iii).

We now prove  the action is coercive. It suffices to show that
   $\Gamma^{s,t}_{K_s\rightarrow K_t}$ is compact in the uniform topology. Given a minimal geodesic $\gamma$ from $K_s$ to $K_t$,
   its length satisfies
   \[
   L_d(\gamma)=d(\gamma(s),\gamma(t))\leq \max_{x\in K_s,y\in K_t}d(x,y)<\infty.
   \]
   Hence,
  one can
choose a finite constant $R>0$ such that all the minimal geodesics from $K_s$ to $K_t$ are contained in the compact ball $\overline{B^+_{\star}(R)}$. In view of (ii), the compactness follows immediately by  Theorem \ref{Arzela-Ascoli Theorem} and Proposition \ref{shorttoshort}.
\end{proof}

\begin{remark}\label{reversiblelagrage}
It follows from Theorem \ref{turespapcelagranceaction}, Remark \ref{unfromequaive},
Theorem \ref{topologychara}/(ii) and Remark \ref{remarkcompleteseparted}  that
$\mathcal {A}^{0,1}$ is a coercive Lagrange action and   $c^{s,t}$ is the cost function on the Polish space $(X,\hat{d})$.
\end{remark}

 \begin{remark}\label{Finslerspeed}
 Let $(X,d)$ be a forward metric space and let $\gamma:[0,1]\rightarrow X$ be an {\it absolutely continuous curve}, i.e., there exists $f\in L^1([0,1])$ such that
 $d(\gamma(s),\gamma(t))\leq \int^t_s f(\tau)d\tau$  for any $0\leq s\leq t\leq 1$.
Thus a suitable modification to the proof
of Ambrosio, Gigli and Savar\'e \cite[Theorem 1.1.2]{AGS} furnishes that the  following limit (called {\it speed} or {\it metric derivative}) always exists for $\mathscr{L}$-a.e. $t\in [0,1]$,
 \[
v_\gamma(t):=\underset{\epsilon\rightarrow 0^+}{\lim} \frac{d(\gamma(t),\gamma(t+\epsilon))}{\epsilon}=\underset{\epsilon\rightarrow 0^+}{\lim}\frac{d(\gamma(t-\epsilon),\gamma(t))}{\epsilon}.
\]
 Moreover, by the same method as employed in Burago, Burago and Ivanov  \cite[Theorem 2.7.6]{DYS}, one can show
\[
L_d(\gamma)=\int^1_0 v_\gamma(\tau) {\ddd}\tau,\quad \mathcal {A}^{s,t}_p(\gamma)=\int^t_s [v_\gamma(\tau)]^p {\ddd}\tau.
\]
If $(X,d)$ is induced by a Finsler manifold and $\gamma$ is piecewise smooth, (\ref{BM-1}) yields $v_\gamma(t)=F(\dot{\gamma}(t))$.
\end{remark}



The following result follows from Remark \ref{reversiblelagrage} and Villani \cite[Proposition 7.16]{Vi}/(vi) immediately.
\begin{proposition}\label{basispropofLargprop} Given a forward boundedly compact forward geodesic space $(X,d)$, let $\mathcal {A}^{0,1}$ be the Lagrange action and let $c^{s,t}$ be the cost function as in Theorem \ref{turespapcelagranceaction}. Then
for all times $s < t$, there exists a Borel map $S_{s\rightarrow t} : X \times X \rightarrow
C([s, t];X)$  such that for all $x, y \in X$, $S_{s\rightarrow t}(x, y)$ belongs to $\Gamma^{s,t}_{x\rightarrow y}$.
\end{proposition}

\begin{definition}
Given a forward boundedly compact forward geodesic space $(X,d)$, let the Lagrangian action $\mathcal {A}^{0,1}$ and the cost function $c^{s,t}$ be as in Theorem \ref{turespapcelagranceaction}. Denote by $\Gamma(X)$  the {\it set of constant-speed minimal geodesics in $X$ (defined on $[0,1]$)}.
\begin{itemize}

\item  A {\it dynamical transference plan} $\Pi$ is a probability measure on $C([0,1];X)$ and the {\it evaluation} at time $t$ is a map $e_t:C([0,1];X)\rightarrow X$ defined as $e_t (\gamma):=\gamma(t)$.

\smallskip

\item  Given two measures $\mu_0,\mu_1\in P(X)$, a {\it dynamical optimal transference plan} of $(\mu_0,\mu_1)$ is a probability measure $\Pi$ on $\Gamma(X)$ such that $\pi=(e_0,e_1)_\sharp \Pi$ is an optimal transference plan from $\mu_0$ to $\mu_1$ with respect to $c^{0,1}(=d^{\,p})$.
If $\Pi$ is a dynamical optimal transference plan, then $\mu_t:=(e_t)_{\sharp}\Pi$, $0\leq t\leq 1$ is called a {\it displacement interpolation}.
\end{itemize}
\end{definition}

\begin{remark}\label{existencedynamicaloptimal}  Given  any two measures $\mu_0,\mu_1\in P(X)$, let $\pi$ be an optimal transference plan from $\mu_0$ to $\mu_1$ (w.r.t. $c^{0,1}=d^p$) and
 $\Pi:=S_{\sharp}\pi$, where  $S:=S_{0\rightarrow 1}:X\times X\rightarrow \Gamma(X)\subset C([0,1];X)$ is the map  in Proposition \ref{basispropofLargprop}. It follows from  $(e_0,e_1)\circ S=\id_{X\times X}$ that $\Pi$ is a dynamical optimal transference plan of $(\mu_0,\mu_1)$ with $\pi=(e_0,e_1)_\sharp\Pi$.
\end{remark}

\begin{remark}\label{continumut}Let $\mu_t=(e_t)_{\sharp}\Pi$, $0\leq t \leq 1$ be a displacement interpolation. Then
the dominated convergence theorem  implies that $(\mu_t)_{0\leq t\leq 1}$ is continuous in $P(X)$ (w.r.t. the weak topology).
\end{remark}

In the following, we view $P(X)$ as a reversible metric space $(P(X),d_P)$. On account of Proposition \ref{measureweakconvergence}, it is a {complete  separable reversible} metric space equipped with the weak topology. Endow
$C([0,1];P(X))$ with the uniform topology induced by $d_P$ and equip $P(C([0,1];X))$ with the weak topology. Then we have the following result, whose proof  is given in Appendix \ref{optimaltarsn}.

\begin{lemma}\label{compactgeodeisc}
Let $(X,\star,d)\in \mathcal {M}^\Theta_\star$ be a forward geodesic space.  Then the map
\[
\mathfrak{E}:P(C([0,1];X))\rightarrow C([0,1];P(X)), \quad \Pi\mapsto ((e_t)_\sharp \Pi)_{0\leq t\leq 1}
\]
is continuous.
\end{lemma}

\begin{theorem}\label{lengspacewass} Given a pointed forward geodesic space $(X,\star,d)\in \mathcal {M}^\Theta_\star$, let $(P_p(X),W_p)$ be the Wasserstein space defined as in Theorem \ref{wassdistacbasisthe} for some $p>1$.
Thus, given any two measures $\mu_0,\mu_1\in P_p(X)$, and a continuous curve $(\mu_t)_{0\leq t\leq 1}$ valued in $P(X)$, the following properties are equivalent:
\begin{itemize}
\item[(i)] $\mu_t$ is a displacement interpolation, i.e., there exists a dynamical optimal transference plan $\Pi$ with
\[
\mu_t=(e_t)_\sharp\Pi, \quad 0\leq t\leq 1.\tag{4.7}\label{conmut}
\]
\item[(ii)] $(\mu_t)_{0\leq t\leq 1}$ is a constant-speed minimal geodesic from $\mu_0$ to $\mu_1$ in the space $(P_p(X),W_p)$, i.e.,
\[
W_p(\mu_s,\mu_t)=(t-s)\,W_p(\mu_0,\mu_1), \text{ for any }0\leq s\leq t\leq 1.\tag{4.8}\label{minimalwpgeodes}
\]
\end{itemize}

\smallskip

In particular, if $\mathcal {K}_0,\mathcal {K}_1\subset P_p(X)$ are two compact subsets of $P(X)$, then the set of dynamical optimal transference plans $\Pi$ with $(e_0)_\sharp \Pi\in \mathcal {K}_0$  and $(e_1)_\sharp \Pi\in \mathcal {K}_1$ is compact and nonempty (w.r.t. the weak topology of $P(C([0,1];X))$), and the set of constant-speed minimal geodesics $(\mu_t)_{0\leq t\leq 1}$ such that $\mu_0\in \mathcal {K}_0$ and $\mu_1\in \mathcal {K}_1$ is compact and nonempty (w.r.t. the uniform topology on $C([0,1];P(X))$).

\end{theorem}

\begin{proof} Given $\mu_1,\mu_2\in P_p(X)$, for any $0\leq s\leq t\leq 1$, let $C^{s,t}(\mu_1,\mu_2)$ be the optimal transport cost from $\mu_1$ to $\mu_2$, that is,
\begin{align*}
C^{s,t}(\mu_1,\mu_2):=\inf_{\pi\in \Pi(\mu_1,\mu_2)}\int_{X\times X} c^{s,t}(x,y)d\pi(x,y)=\frac{W_p(\mu_1,\mu_2)^p}{(t-s)^{p-1}},
\end{align*}
where $c^{s,t}(x,y)$ is defined as in Theorem \ref{turespapcelagranceaction}. In particular,
$C^{0,1}(\mu_1,\mu_2)=W_p(\mu_1,\mu_2)^p<\infty$.  Let $(\mu_t)_{0\leq t\leq 1}$ be a continuous path from $\mu_0$ to $\mu_1$ in $P(X)$ (w.r.t. the weak topology).
Now it follows from  Remark \ref{remarkcompleteseparted} and
Villani \cite[Theorem 7.21]{Vi} that the following statements are equivalent:

\smallskip

(a) $\mu_t$ is a displacement interpolation;

\smallskip

(b) for any three intermediate times $t_1<t_2<t_3$ in $[0,1]$, one has
\[
C^{t_1,t_2}(\mu_{t_1},\mu_{t_2})+C^{t_2,t_3}(\mu_{t_2},\mu_{t_3})=C^{t_1,t_3}(\mu_{t_1},\mu_{t_3}).\tag{4.9}\label{intermedminaltogeodesic}
\]

Hence,
in order to show that  (i) is equivalent  to (ii), it suffices to prove
that (ii)$\Leftrightarrow$(b).  The ``$\Rightarrow$" part is obvious. For
  the ``$\Leftarrow$" part,
the convexity of $|\cdot|^p$ and the triangle inequality yield
\begin{align*}
\left(\frac{ W_p(\mu_{t_1},\mu_{t_3})}{t_3-t_1}  \right)^p \leq \left( \frac{W_p(\mu_{t_1},\mu_{t_2})+W_p(\mu_{t_2},\mu_{t_3})}{t_3-t_1} \right)^p\leq \frac{1}{t_3-t_1}\left(\frac{W_p(\mu_{t_1},\mu_{t_2})^p}{(t_2-t_1)^{p-1}}+\frac{W_p(\mu_{t_2},\mu_{t_3})^p}{(t_3-t_2)^{p-1}}\right),
\end{align*}
which together with (\ref{intermedminaltogeodesic})
then furnishes
\[
W_p(\mu_{t_1},\mu_{t_2})+W_p(\mu_{t_2},\mu_{t_3})=W_p(\mu_{t_1},\mu_{t_3}),\quad \frac{W_p(\mu_{t_1},\mu_{t_2})}{t_2-t_1}=\frac{W_p(\mu_{t_2},\mu_{t_3})}{t_3-t_2}.
\]
Hence, we obtain $W_p(\mu_s,\mu_t)=(t-s)\, W_p(\mu_0,\mu_1)$, which together with the triangle inequality of $W_p$ implies $\mu_t\in P_p(X)$ for any $t\in [0,1]$. Therefore,
(b) follows.

We now show the second part of the theorem.
Let $\mathcal {K}_0,\mathcal {K}_1\subset P_p(X)$ be two compact subsets of $P(X)$. Theorem \ref{wassdistacbasisthe}/(i) implies $C^{0,1}(\mu_0,\mu_1)<+\infty$ for all $\mu_0\in \mathcal {K}_0$, $\mu_1\in \mathcal {K}_1$. Hence,
Remark \ref{remarkcompleteseparted} together with
Villani \cite[Theorem 7.21]{Vi}  again furnishes that the set of dynamical optimal transference plans $\Pi$ with $(e_0)_\sharp \Pi\in \mathcal {K}_0$  and $(e_1)_\sharp \Pi\in \mathcal {K}_1$ is compact. The non-emptyness follows by Remark \ref{existencedynamicaloptimal}.
 Since the continuous image of a compact and nonempty set is compact and nonempty,  it follows from Lemma \ref{compactgeodeisc} that the set of constant-speed minimal geodesics $(\mu_t)_{0\leq t\leq 1}$ such that $\mu_0\in \mathcal {K}_0$ and $\mu_1\in \mathcal {K}_1$ is compact and nonempty  as well.
\end{proof}


\begin{corollary}\label{lengthpropWASSER}
Let $(X,d)$ be a compact $\theta$-geodesic space. Then for each $p\in [1,\infty)$, the Wasserstein space $(P(X),W_p)$ is a compact $\theta$-geodesic space as well.
\end{corollary}
\begin{proof}  Due to Corollary \ref{compactwasserin} and Proposition \ref{metricsigeodesic}, it suffices to show that for any $\mu_0,\mu_1\in P(X)$, there is a geodesic from $\mu_0$ to $\mu_1$. The case when $p>1$ follows directly by
     Theorem \ref{lengspacewass}. On the other hand,
if $\gamma(t)$, $t\in[0,1]$ is a constant-speed minimal geodesic in $X$, we have
$d(e_s(\gamma),e_t(\gamma))=(t-s)\,d(e_0(\gamma),e_1(\gamma))$ for any $[s,t]\subset[0,1]$.
Based on this fact,
the case $p=1$ follows by a standard argument combined with Remark \ref{existencedynamicaloptimal}.
\end{proof}



\begin{definition}
Given a forward boundedly compact forward geodesic space $(X, d)$, let $\Pi\in P(\Gamma(X))$ be a dynamical  transference plan {with respect to $c=d^{\,p}$}. For each $t\in (0,1)$, define the {\it associated kinetic energy} $\varepsilon_t$ (which is a measure on $X$) by
\[
\varepsilon_t:=(e_t)_\sharp \left( \frac{L^2_d}{2}\Pi \right),
\]
where $L_d$ is the length defined by Definition \ref{dfelength}.
If $\varepsilon_t$ is absolutely continuous w.r.t.  $\mu_t:=(e_t)_\sharp \Pi$, define the {\it speed field} $|v|(t,x)$ by
\[
|v|(t,x):=\sqrt{2\frac{{\ddd} \varepsilon_t}{{\ddd}\mu_t}}.
\]
\end{definition}

\begin{example}[Compact case]\label{compactkinetic} Let $(X,d)$ be a compact $\theta$-geodesic space. Thus there holds
\begin{align*}
\varepsilon_t\leq \frac{\diam^2(X)}{2}(e_t)_{\sharp}\Pi=\frac{\diam^2(X)}{2}\mu_t,
\end{align*}
 which implies $ |v|(t,x)\leq \diam(X)$.
 Hence, $|v|(t,\cdot)$ is well-defined and is bounded by $\diam(X)$  (up to modification on a set of zero $\mu_t$-measure).
\end{example}

\begin{example}[Finsler case] Let $(M,d_F)$ be a  forward geodesic space induced by a forward complete Finsler manifold $(M,F)$. Assume that the geodesics in  $\supp\Pi$ do not cross at intermediate times; by a similar argument to Villani \cite[Corollary 8.2, Theorem 8.5]{Vi}, one can show that this case will happen when $\Pi$ is an dynamical optimal transference plan (w.r.t. $c=d^{\,2}$). Then for each $t\in (0,1)$ and $x\in M$, there is at most one geodesic $\gamma$ with $\gamma(t)=x$. Thus,
\begin{align*}
\varepsilon_t({\ddd}x)=\left(\frac{F^2(\dot{\gamma}(t))}{2}  \right)[(e_t)_\sharp\Pi]({\ddd}x)=\left(\frac{F^2(\dot{\gamma}(t))}{2}  \right)\mu_t({\ddd}x).
\end{align*}
Thus, $|v|(t,x)$ is exactly $F(\dot{\gamma}(t))$, the speed at time $t$ and position $x$, see Remark \ref{Finslerspeed}.
\end{example}

\begin{theorem}\label{vectorfieldscontrol}
Let $(X,d)$ be a compact $\theta$-geodesic space, let $\Pi\in P(\Gamma(X))$ be a dynamical optimal transference plan with respect to $c=d^2$, let $(\mu_t)_{0\leq t\leq 1}$ be the associated displacement interpolation and let $|v|=|v|(t,x)$ be the associated speed field. Thus, for each $t\in (0,1)$, one can modify $|v|(t,\cdot)$ on a $\mu_t$-negligible set in such a way that for all $x,y\in X$,
\[
\left||v|(t,x)-|v|(t,y)   \right|\leq  \sqrt{\frac{3\theta(1+\theta)\diam(X)}{2t(1-t)}}\sqrt{d(x,y)}.
\]
\end{theorem}
\begin{proof} Let $\pi:=(e_0,e_1)_\sharp \Pi$ be the corresponding optimal transference plan from $\mu_0$ to $\mu_1$. Since $(X,d)$ is compact,  the optimal cost $C(\mu_0,\mu_1)=\int_{X\times X} c\, {\ddd}\pi$ is finite.
 Moreover, since $c$ is continuous on the product of Polish spaces $(X,\hat{d})\times (X,\hat{d})$, it follows from Villani \cite[Theorem 5.10/(ii)]{Vi} that $\pi$ is concentrated on a $c$-cyclically monotone set. Hence,
 for any two geodesics $\gamma_i$, $i=1,2$  in  $\supp\Pi$,
 \[
d(\gamma_1(0),\gamma_1(1))^2+d(\gamma_2(0),\gamma_2(1))^2\leq d(\gamma_1(0),\gamma_2(1))^2+d(\gamma_2(0),\gamma_1(1))^2,
\]
with together with the  same argument as in Villani \cite[Theorem 8.22]{Vi} yields
\[
\left| L_d(\gamma_1)-L_d(\gamma_2) \right|\leq  \sqrt{\frac{3(1+\theta)}{2}}\sqrt{\frac{\diam(X)}{t(1-t)}}\sqrt{d(\gamma_1(t),\gamma_2(t))}, \ \forall t\in (0,1).
\]
The rest of the proof goes in a similar way as in Villani \cite[Theorem 28.5]{Vi}, thus we omit it.
\end{proof}

\subsection{Stability of optimal transport}
\begin{proposition}\label{strongconverges}
Let $(X_i,d_i)$, $i=1,2$ be two compact $\theta$-metric spaces. If  $f:(X_1,d_1)\rightarrow (X_2,d_2)$ is an $\epsilon$-isometry, then for each $p\in [1,\infty)$, $f_\sharp:P_p(X_1)\rightarrow P_p(X_2)$ is an $\tilde{\epsilon}$-isometry, where
\[
\tilde{\epsilon}=2(1+\theta)\epsilon+(2+\theta)\sqrt[p]{\epsilon p (\diam(X_2)+ \epsilon)^{p-1}}.
\]
\end{proposition}
\begin{proof}
The assumption implies
$\diam(X_2)-\epsilon\leq \diam(X_1)\leq \diam(X_2)+\epsilon$.
Hence, for any $x_1,y_1\in X_1$, by the usual mean value theorem we obtain that
\begin{align*}
|d_2(f(x_1),f(y_1))^p-d_1(x_1,y_1)^p |
&\leq \epsilon p \left( \diam(X_2)+\epsilon \right)^{p-1}.\tag{4.10}\label{connectedmeditem}
\end{align*}

Given $\mu_1,\mu'_1\in P_p(X_1)$, let $\pi_1$ be an optimal transference plan from $\mu_1$ to $\mu_1'$ and set $\pi_2:=(f, f)_\sharp\pi_1$. Since $\pi_2$ is a transference plan from $f_\sharp\mu_1$ to $f_\sharp\mu_1'$, by (\ref{connectedmeditem}) we have
\begin{align*}
W_{p}(f_{\sharp}\mu_1,f_{\sharp}\mu_1')^p\leq& \int_{X_2\times X_2}d_2(x_2,y_2)^p {\ddd}\pi_2(x_2,y_2)=\int_{X_1\times X_1}d_2(f(x_1),f(y_1))^p {\ddd}\pi_1(x_1,y_1)\\
\leq& \int_{X_1\times X_1}\left[ d_1(x_1,y_1)^p+\epsilon p \left( \diam(X_2)+\epsilon \right)^{p-1} \right] {\ddd}\pi_1(x_1,y_1)\\
=&W_{p}(\mu_1,\mu_1')^p+\epsilon p \left( \diam(X_2)+\epsilon \right)^{p-1}.
\end{align*}
Hence, we have
\begin{align*}
W_{p}(f_{\sharp}\mu_1,f_{\sharp}\mu_1') \leq W_{p}(\mu_1,\mu_1') +\sqrt[p]{\epsilon p (\diam(X_2)+\epsilon)^{p-1}}\tag{4.11}\label{needseoncdin}
\end{align*}
 A similar argument together with Proposition \ref{deltanprooxima}  furnishes
\begin{align*}
W_{p}((f_r)_{\sharp}f_{\sharp}\mu_1,(f_r)_{\sharp}f_{\sharp}\mu_1')
\leq W_{p}(f_{\sharp}\mu_1,f_{\sharp}\mu_1') +(2+\theta)\sqrt[p]{\epsilon p (\diam(X_2)+ \epsilon)^{p-1}},\tag{4.12}\label{neddtrangleinequal}
\end{align*}
where $f_r$ is an approximate inverse of $f$. Since $d_1(f_r\circ f(x),x)\leq 2\epsilon$ (see  Proposition \ref{deltanprooxima}), we have
\begin{align*}
W_p\left((f_r\circ f)_{\sharp}\mu_1,\mu_1 \right)^p
\leq \int_{X_1} d_1((f_r\circ f)(x_1),x_1)^p {\ddd}\mu_1(x)\leq (2\epsilon)^p.
\end{align*}
Similarly, one has $ W_{p}\left((f_r\circ f)_{\sharp}\mu'_1,\mu'_1 \right)\leq 2\epsilon$.
Thus, (\ref{neddtrangleinequal}) together with Corollary \ref{compactwasserin}/(i) yields
\begin{align*}
W_{p}(\mu_1,\mu_1')\leq & W_{p}(\mu_1,(f_r\circ f)_{\sharp}\mu_1)+W_{p}((f_r\circ f)_{\sharp}\mu_1,(f_r\circ f)_{\sharp}\mu'_1)+W_{p}((f_r\circ f)_{\sharp}\mu'_1, \mu_1')\\
\leq & W_{p}(f_{\sharp}\mu_1,f_{\sharp}\mu_1') +2(1+\theta)\epsilon+(2+\theta)\sqrt[p]{\epsilon p (\diam(X_2)+ \epsilon)^{p-1}}.
\end{align*}
which together with (\ref{needseoncdin}) implies $\dis f_{\sharp}\leq \tilde{\epsilon}$.

On the other hand, given $\mu_2\in P_p(X_2)$, consider the Monge transport $(f\circ f_r,\id)$ from $(f\circ f_r)_{\sharp}\mu_2$ to $\mu_2$. Since $d_2(f\circ f_r(x_2),x_2)\leq \epsilon$, we have
\begin{align*}
W_{p}\left((f\circ f_r)_{\sharp}\mu_2,\mu_2\right)^p\leq \int_{X_2} d_2((f\circ f_r)(x_2),x_2)^p\, {\ddd}\mu_2(x)\leq  \epsilon^p\leq \tilde{\epsilon}^p,
\end{align*}
which implies $P_p(X_2)\subset\overline{f_\sharp(P_p(X_1))^{\tilde{\epsilon}}}$.
\end{proof}


The following theorem immediately follows by Corollary \ref{compactwasserin}, Proposition \ref{strongconverges} and Lemma \ref{inprotantisometry}.

\begin{theorem}\label{compactwassspace}
If a sequence  $ (X_i,d_i)_i\subset \mathcal {M}^\theta$
converges   to  $(X, d)\in \mathcal {M}^\theta$ in the $\theta$-Gromov-Hausdorff topology,
then for every $p\in [1,\infty)$, $(P_p(X_i))_i$
 converges  to $P_p(X)$ in the $\theta$-Gromov-Hausdorff topology.
\end{theorem}

\begin{remark}
Theorem \ref{compactwassspace} cannot be extended to the noncompact case (under the  forward $\Theta$-Gromov-Hausdorff topology) because the corresponding Wasserstein spaces are usually not locally compact and hence, not forward boundedly compact, cf. Villani \cite[Remark 28.14]{Vi}.
\end{remark}

We conclude this section by showing that the quantities of optimal transport are also stable under the $\theta$-Gromov-Hausdorff topology, which is an irreversible $p$-power version of Villani \cite[Theorem 28.9]{Vi}; we notice that in the latter result only the $p=2$ case  is considered.

Given  a forward geodesic space $(X,d)$, let $\mathbb{M}([0,1];X)$ be the set of measurable paths from $[0,1]$ to $X$.
Define a map $i:\mathbb{M}([0,1];X)\rightarrow P([0,1]\times X)$ as
\[
i(\gamma):=\overline{\gamma}=(\id,\gamma)_{\sharp}(\mathscr{L}),
\]
where $\mathscr{L}$ is the Lebesgue measure on $[0,1]$. Moreover,
\[
(\id,\gamma)^{-1}(t,x)= \left\{
	\begin{array}{lll}
	\{t\}, & \ \ \ \text{if }x=\gamma(t), \\
	\\
	\emptyset, & \ \ \ \text{if }x \neq\gamma(t);
	\end{array}
	\right.\Longrightarrow \ \ \overline{\gamma}({\ddd}t\,{\ddd}x)= \delta_{x=\gamma(t)}{\ddd}t.
\]

Note that
the restriction  ${j:=}i|_{C([0,1];X)}: C([0,1];X)\rightarrow P([0,1]\times X)$ is injective. Thus,
 each measure $\Pi\in P(\Gamma(X))$ can be identified with its push-forward $i_\sharp \Pi\in P(P([0,1]\times X))$.
In fact,
for each $\nu\in \supp i_\sharp\Pi$,  there is a unique $\gamma\in \Gamma(X)$ with $\overline{\gamma}=\nu$ and hence, $i_\sharp\Pi({\ddd}\overline{\gamma})=\Pi({\ddd}\gamma)$.

\begin{theorem}\label{stabilityoptimal}
Let $\mathcal {X}_k:=(X_k,d_k)$, $k\in \mathbb{N}$ and $\mathcal {X}:=(X,d)$ be compact $\theta$-geodesic spaces such that $(\mathcal {X}_k)_k$ converges to $\mathcal {X}$ in the $\theta$-Gromov-Hausdorff topology, by means of $\epsilon_k$-isometries $f_k:\mathcal {X}_k\rightarrow \mathcal {X}$.

 For each $k\in \mathbb{N}$, let $\Pi_k$ be a Borel probability measure on $\Gamma(\mathcal {X}_k)$. Further, let $\pi_k=(e_0,e_1)_\sharp\Pi_k$, $\mu_{k,t}=(e_t)_\sharp \Pi_k$ and $\varepsilon_{k,t}=(e_t)_\sharp[L_{d_k}^2 \Pi_k/2]$.
Then after passing to a subsequence, still denoted with the index $k$ for simplicity, there is a dynamical transference plan $\Pi$ on $\mathcal {X}$, with associated transference plan $\pi=(e_0,e_1)_\sharp \Pi$, measure-valued path $\mu_t=(e_t)_\sharp\Pi$ for $t\in[0,1]$, and kinetic energy $\varepsilon_t:=(e_t)_\sharp[L^2_d\Pi/2]$ satisfying

\smallskip

\begin{itemize}
\item[(i)]
$\lim_{k\rightarrow\infty}(i\circ f_k\circ)_{\sharp}\Pi_k=i_\sharp\Pi$ in the weak topology on $P(P([0,1]\times X))$, where  $f_k\circ:C([0,1];X_k)\rightarrow  \mathbb{M}([0,1];X)$ is the map defined as $\gamma\longmapsto f_k\circ\gamma;$

\smallskip

\item[(ii)] $\lim_{k\rightarrow \infty}(f_k,f_k)_{\sharp}\pi_k=\pi$ in the weak topology on $P(X\times X);$

\smallskip

\item[(iii)] $\lim_{k\rightarrow \infty}(f_k)_\sharp \mu_{k,t}=\mu_t$ in $P_p(X)$ uniformly in $t$ for every $p\in [1,\infty)$. More precisely, we have
\[
\lim_{k\rightarrow \infty}\sup_{t\in [0,1]}W_p(\mu_t,(f_k)_\sharp\mu_{k,t})=0 \ \text{ for any }p\in [1,\infty);
\]

\item[(iv)] $\lim_{k\rightarrow \infty} (f_k)_\sharp \varepsilon_{k,t}=\varepsilon_t$ in the weak topology of measures, for each $t\in (0,1)$.

\end{itemize}

\smallskip

\noindent Additionally suppose that each $\Pi_k$ is a dynamical optimal transference plan w.r.t.\ the $p$-power distance cost function for some $p\in [1,\infty)$. Thus:

\begin{itemize}

\smallskip

\item[(v)] The limit $\Pi$ is a dynamical optimal transference plan; hence, $\pi$ is an optimal transference plan and $(\mu_t)_{0\leq t\leq 1}$ is a displacement interpolation$;$

\smallskip

\item[(vi)]  If $p=2$, then for each $t\in (0,1)$, there is a choice of the speed fields $|{v_{k,t}}|:=|v_k|(t,\cdot)$  associated with the plans $\Pi_k$  such that $\lim_{k\rightarrow \infty}|v_{k,t}| \circ f_{r,k}=|v_t|$ in the uniform topology, where $|v_t|:=|v|(t,\cdot)$ and $f_{r,k}$ is the approximate inverse of $f_k$ defined in Proposition \ref{deltanprooxima}$.$

\end{itemize}
\end{theorem}

\begin{proof}Since the proof is a modification of the arguments from  Villani \cite[Theorem 28.9]{Vi}, which {arises} from the irreversible character of the spaces, we just focus on the differences.

\noindent\textbf{Step 1.} (1) Since $[0,1]\times X$  endowed with the product metric $\sqrt{d^2_\mathbb{R}+d_X^2}$ is a compact $\theta$-metric space,  Corollary \ref{lengthpropWASSER} implies that  $P([0,1]\times X)$ and $P(P([0,1]\times X))$ equipped with the Wasserstein metrics  are compact $\theta$-metric/geodesic spaces. In view of Theorem \ref{compactequvitheorem}/(i)(ii) and Proposition \ref{weaktopologyvswstatopo},
by  passing to a subsequence, we can assume that  $\left((i\circ f_k\circ)_{\sharp}\Pi_k \right)_{k}$ (resp., $((f_k,f_k)_\sharp \pi_k)_k$) converges to some $\widehat{\Pi}\in P(P([0,1]\times X))$ (resp., $\pi\in P(X\times X)$) in the weak topology. For convenience, let ``$\rightharpoonup$" denote  the weak convergence (as $k\rightarrow \infty$). Then we have
\begin{align*}
(i\circ f_k\circ)_\sharp \Pi_k\rightharpoonup\widehat{\Pi} \in P(P([0,1]\times X)),\ \ \
(f_k,f_k)_\sharp \pi_k\rightharpoonup \pi\in P(X\times X).\tag{4.13}\label{frsittwoweakconverge}
\end{align*}

(2) Since $(\mathcal {X}_k)_k$ converges to $\mathcal {X}$ in the $\theta$-Gromov-Hausdorff topology,  the diameters of $\mathcal {X}_k$'s  are bounded above by a constant $C=C(\theta, \diam(\mathcal {X}))$. Hence,
for any geodesic $\gamma_k\in \Gamma(\mathcal {X}_k)$, we have
\[
d_k(\gamma_k(t),\gamma_k(s))\leq \theta\,|t-s|\, d_k(\gamma(0),\gamma(1))\leq  \theta\, C\,|t-s|,\ \forall\,s,t\in [0,1],\tag{4.14}\label{diamcontralength}
\]
which implies for any $s,t\in [0,1]$,
\begin{align*}
W_p (\mu_{k,s},\mu_{k,t})&\leq  \left(\int_{\Gamma} d_k(e_s(\gamma),e_t(\gamma))^p \,{\ddd} \Pi_k(\gamma)\right)^{\frac1p}\leq\theta C|t-s|.\tag{4.15}\label{wasdissmu}
\end{align*}
Hence, $(\mu_{k,t})_{0\leq t\leq 1}$ are uniformly continuous in $t$ with a uniform modulus of continuity.

It follows from  Theorem \ref{compactwassspace} and Proposition \ref{strongconverges} that $(P_p(X_k))_k$ converges to $P_p(X)$ in the $\theta$-Gromov-Hausdorff topology w.r.t. $(f_k)_\sharp$ for every $p\in [1,\infty)$.
Thus,
  by Proposition \ref{Ascloghtipo}, (\ref{wasdissmu}) and a Cantor's diagonal argument, we may assume that $\left((f_k)_\sharp\mu_{k,t}\right)_{t\in [0,1]}$  converges uniformly to a continuous curve $(\mu_t)_{t\in [0,1]}\in C([0,1]; P_q(X))$ for every $q\in \mathbb{N}$.
That is,
$\lim_{k\rightarrow \infty}\sup_{t\in [0,1]}W_q(\mu_t,(f_k)_\sharp\mu_{k,t})=0$ for any $q\in \mathbb{N}$.
Thus, for each $p\in [1,\infty)$, choose a $q\in \mathbb{N}$ with $p\leq q$. Since $W_p\leq W_q$,  there holds
\[
\lim_{k\rightarrow \infty}\sup_{t\in [0,1]}W_p(\mu_t,(f_k)_\sharp\mu_{k,t})=0.\tag{4.16}\label{ththridweakconverge}
\]


(3) For each $t\in (0,1)$ and each $k\in \mathbb{N}$, since $\varepsilon_{k,t}[X_k]\leq \diam(\mathcal {X}_k)^2/2\leq C^2/2$ (see Example \ref{compactkinetic}), we have $(f_k)_\sharp \varepsilon_{k,t}[X]\leq C^2/2$. Thus, by Prohorov's Theorem (cf. Daley and Vere-Jones \cite[Theorem A2.4.I]{DV}) and Cantor's diagonal argument,
 we may assume  that for  $t\in (0,1)\cap \mathbb{Q}$, there exists a finite measure $\varepsilon_t$ on $X$ such that $(f_k)_\sharp \varepsilon_{k,t}\rightharpoonup \varepsilon_t$.

\noindent\textbf{Step 2.} According to (\ref{frsittwoweakconverge}) and (\ref{ththridweakconverge}),
in order to prove  (i)-(iv), it suffices to show the following:

\smallskip

 \textbf{(a)}  $\widehat{\Pi}$ induces a measure $\Pi$ on $P(\Gamma(X))$ such that $\widehat{\Pi}=i_\sharp\Pi$; \ \textbf{(b)} $\pi=(e_0,e_1)_\sharp\Pi$; \  \textbf{(c)} $\mu_t=(e_t)_\sharp \Pi$;

\textbf{(d)} $\varepsilon_t=(e_t)_\sharp(L_d^2\Pi/2)$ and $(f_k)_\sharp \varepsilon_{k,t}\rightharpoonup \varepsilon_t$ for every $t\in [0,1]$.

\bigskip

\noindent\textbf{(a)}
For $\delta\in (0,1/2)$, set $\varphi^\delta_+(s):=\varphi^\delta(s-\delta)$ and  $\varphi^\delta_-(s):=\varphi^\delta(s+\delta)$, where
\[
\varphi^\delta(s):=\frac{\delta+s}{\delta^2}\textbf{1}_{-\delta\leq s\leq 0}+\frac{\delta-s}{\delta^2}\textbf{1}_{0\leq s\leq \delta}
\]
Thus, $\suppor \varphi^\delta_+\subset [0,2\delta]$ and $\suppor \varphi^\delta_-\subset [-2\delta,0]$. Note that $\int_\mathbb{R}\varphi^\delta_+(s){\ddd}s=\int_\mathbb{R}\varphi^\delta_-(s){\ddd}s=1$ and $\varphi^\delta_\pm$ converges to the Dirac mass $\delta_0$ in the weak topology as $\delta\rightarrow 0^+$.

Given $0\leq s_0< t_0\leq 1$,
for any $\gamma\in \mathbb{M}([0,1];X)$, define
\begin{align*}
\mathcal {L}^\delta_{s_0\rightarrow t_0}(\gamma)&:=\int^1_0 \int^1_0 d(\gamma(s),\gamma(t))\,\varphi^\delta_+(s-s_0)\,\varphi^\delta_-(t-t_0)\,{\ddd}s\, {\ddd}t.
\end{align*}
On the one hand, $\lim_{\delta\rightarrow 0}\mathcal {L}^\delta_{s_0\rightarrow t_0}(\gamma)=d(\gamma(s_0),\gamma(t_0))$ if $\gamma\in C([0,1];X)$.
On the other hand,
since $f_k$ is an $\epsilon_k$-isometry, for any $\gamma_k\in \Gamma(X_k)$, we have
\[
\mathcal {L}^\delta_{s_0\rightarrow t_0}(f_k\circ \gamma_k)=\int^1_0 \int^1_0 d_k(\gamma_k(s),\gamma_k(t)) \,\varphi^\delta_+(s-s_0)\,\varphi^\delta_-(t-t_0)\,{\ddd}s\, {\ddd}t+O(\epsilon_k).\tag{4.17}\label{basisestimateLf}
\]

By (\ref{basisestimateLf}) and (\ref{diamcontralength}), we have
\begin{align*}
\mathcal {L}^\delta_{s_0\rightarrow t_0}(f_k\circ \gamma_k)
\leq &\theta C \int^1_0 \int^1_0|t-s|\,\varphi^\delta_+(s-s_0)\,\varphi^\delta_-(t-t_0)\,{\ddd}s\, {\ddd}t+O(\epsilon_k)\\
\leq & \theta C\left[ (t_0-s_0)+O(\delta) \right]+O(\epsilon_k)\leq C'[(t_0-s_0)+\delta+\epsilon_k],\tag{4.18}\label{firstestimforlispmeas}
\end{align*}
where $C'$ is a constant independent of $t_0,s_0,\delta$ and $\epsilon_k$.

Furthermore, for  $\delta\in (0, (t_0-s_0)/4]$ (i.e., $s_0+2\delta\leq t_0-2\delta$), (\ref{basisestimateLf})  yields
\begin{align*}
\mathcal {L}^\delta_{s_0\rightarrow t_0}(f_k\circ \gamma_k)
=&\int^{t_0}_{t_0-2\delta} \int^{s_0+2\delta}_{s_0} d_k(\gamma_k(s),\gamma_k(t)) \,\varphi^\delta_+(s-s_0)\,\varphi^\delta_-(t-t_0)\,{\ddd}s\, {\ddd}t+O(\epsilon_k)\\
=&d_k(\gamma_k(0),\gamma_k(1))\int^{t_0}_{t_0-2\delta} \int^{s_0+2\delta}_{s_0}(t-s) \,\varphi^\delta_+(s-s_0)\,\varphi^\delta_-(t-t_0)\,{\ddd}s\, {\ddd}t+O(\epsilon_k)\\
=&d_k(\gamma_k(0),\gamma_k(1))\int^1_0 \int^1_0|t-s| \,\varphi^\delta_+(s-s_0)\,\varphi^\delta_-(t-t_0)\,{\ddd}s\, {\ddd}t+O(\epsilon_k)\\
=&d_k(\gamma_k(0),\gamma_k(1))\left[ (t_0-s_0)+O(\delta) \right]+O(\epsilon_k).\tag{4.19}\label{medl1}
\end{align*}
Likewise, we get
\[
\mathcal {L}^\delta_{0\rightarrow 1}(f_k\circ \gamma_k)=d_k(\gamma_k(0),\gamma_k(1))(1+O(\delta))+O(\epsilon_k),\ \text{ if }\delta\leq \frac{1}{4},
\]
which together with (\ref{medl1}) furnishes (by choosing a larger $C'$)
\begin{align*}
&\left| \mathcal {L}^\delta_{s_0\rightarrow t_0}(f_k\circ \gamma_k)-(t_0-s_0)\mathcal {L}^\delta_{0\rightarrow 1}(f_k\circ \gamma_k) \right|\leq C'(\delta+\epsilon_k),\ \text{ if }\delta\leq \frac{t_0-s_0}{4}.\tag{4.20}\label{secondestimforlispmeas}
\end{align*}

Similarly, given $0\leq s_0<t_0\leq 1$, for any $\mu\in P([0,1]\times X)$, define
\[
\mathfrak{L}^\delta_{s_0\rightarrow t_0}(\mu):=\int_{[0,1]\times X}\int_{[0,1]\times X} d(x,y)\,\varphi^\delta_+(s-s_0)\,\varphi^\delta_-(t-t_0)\,{\ddd}\mu(s,x)\,{\ddd}\mu(t,y).
\]
From (\ref{firstestimforlispmeas}) and (\ref{secondestimforlispmeas}), we introduce the following set: given any $\epsilon,\delta>0$, let $\Gamma_{\epsilon,\delta}(X)$ be the collection of $\mu\in P([0,1]\times X)$ satisfying the following three conditions:
\begin{itemize}

\smallskip

\item[(1)] $\mathfrak{p}_\sharp(\mu)=\Leb$, where $\mathfrak{p}:[0,1]\times X\rightarrow [0,1]$ is the natural projection;

\smallskip

\item[(2)] $\mathfrak{L}^\delta_{s_0\rightarrow t_0}(\mu)\leq C'[(t_0-s_0)+\delta+\epsilon] \text{ for any }0\leq s_0<t_0\leq 1$;

\smallskip

\item[(3)] $\left| \mathfrak{L}^\delta_{s_0\rightarrow t_0}(\mu)-(t_0-s_0)\mathfrak{L}^\delta_{0\rightarrow 1}(\mu) \right|\leq C'(\delta+\epsilon) \text{ for any }0\leq s_0<t_0\leq 1 \text{ with }4\delta\leq t_0-s_0$.

\end{itemize}

The continuity of $\mathfrak{L}^\delta_{s_0\rightarrow t_0}$ implies that  $\Gamma_{\epsilon,\delta}(X)$ is a closed set in $P([0,1]\times X)$. Moreover,  since  $\mathfrak{L}^\delta_{s_0\rightarrow t_0}(i(f_k\circ \gamma_k))=\mathcal {L}^\delta_{s_0\rightarrow t_0}(f_k\circ \gamma_k)$, it follows from (\ref{firstestimforlispmeas}) and (\ref{secondestimforlispmeas}) that if
  $\epsilon_k\leq \epsilon$,
\[
i(f_k\circ \gamma_k)\in \bigcap_{\delta>0} \Gamma_{\epsilon,\delta}(X),\ \forall\, \gamma_k\in \Gamma(X_k).
\]
Therefore, for large $k$,
\[
i\circ f_k\circ:\Gamma(X_k)\rightarrow \bigcap_{\delta>0} \Gamma_{\epsilon,\delta}(X) \subset \Gamma_{\epsilon,\delta}(X),
\]
which implies $(i\circ f_k\circ)_{\sharp}\Pi_k\in P(\bigcap_{\delta>0}\Gamma_{\epsilon,\delta}(X))$. By passing a limit, we see $\widehat{\Pi}\in P(\bigcap_{\delta>0}\Gamma_{\epsilon,\delta}(X))$ (See Step 1/(1)). Since $\epsilon,\delta$ are both arbitrarily small, we have
\[
\widehat{\Pi}\in P\left(\bigcap_{\epsilon,\delta>0}\Gamma_{\epsilon,\delta}(X)\right).
\]

Recall that ${j:=}i|_{C([0,1];X)}$ is an injective map. Now we claim
\[
\bigcap_{\epsilon,\delta>0}\Gamma_{\epsilon,\delta}(X)=j(\Gamma(X))\cong \Gamma(X).\tag{4.21}\label{inclupigammareat}
\]
Note that if  (\ref{inclupigammareat}) is true, then  we can define a measure $\Pi$ on $\Gamma(X)$ by $\Pi:=(j^{-1})_\sharp\widehat{\Pi}$,
which concludes the proofs of (a)  and  Theorem/(i) at the same time.

In fact, for any $\mu\in j(\Gamma(X))$, there is some $\gamma\in \Gamma(X)$ with $\mu=j(\gamma)=\overline{\gamma}=(\id,\gamma)_\sharp \Leb$ and hence, $\mathfrak{L}^\delta_{s_0\rightarrow t_0}(\mu)=\mathcal {L}^\delta_{s_0\rightarrow t_0}(\gamma)$,
which implies
 $j(\Gamma(X))\subset \cap_{\epsilon,\delta>0}\Gamma_{\epsilon,\delta}(X)$.
On the other hand, given $\mu\in  \cap_{\epsilon,\delta>0}\Gamma_{\epsilon,\delta}(X)$,  for any $\delta>0$ we have
\[
\mathfrak{L}^\delta_{s_0\rightarrow t_0}(\mu)\leq C'[(t_0-s_0)+\delta],\ \text{ for any }0\leq s_0<t_0\leq 1.\tag{4.22}\label{noabsoulvauenofmu}
\]
Since  $W_1(\nu_1,\nu_2)\leq \theta W_1(\nu_2,\nu_1)$ for any $\nu_1,\nu_2\in P(X)$,
 the same argument as in Villani \cite[Lemma 28.11]{Vi} together with (\ref{noabsoulvauenofmu}) yields that
  $\mu$ can be written as $j(\gamma)=\bar{\gamma}=(\id,\gamma)_\sharp \Leb$ for some Lipschitz continuous curve $\gamma:[0,1]\rightarrow X$. Hence, for small $\delta>0$,
\[
\mathfrak{L}^\delta_{s_0\rightarrow t_0}(\mu)=\mathcal {L}^\delta_{s_0\rightarrow t_0}(\gamma)=d(\gamma(s_0),\gamma(t_0))+O(\delta),\ \text{ for any }0\leq s_0<t_0\leq 1.
\]
which together with Condition (3) of $\Gamma_{\epsilon,\delta}$ implies (letting $\epsilon,\delta\rightarrow 0$)
\[
d(\gamma(t_0),\gamma(s_0))=(t_0-s_0) \, d(\gamma(0),\gamma(1)),\ \text{ for any }0\leq s_0<t_0\leq 1.
\]
So $\gamma$ is a minimal geodesic in $\Gamma(X)$. Therefore, $\cap_{\epsilon,\delta>0}\Gamma_{\epsilon,\delta}(X)\subset j(\Gamma(X))$.
Thus, (\ref{inclupigammareat}) follows.

\smallskip

 \textbf{(b),(c)}  The proofs of (b) and (c) are the same as those of Villani \cite[p.784--786]{Vi}.

\smallskip

  \textbf{(d)} In view of the argument in Step 1/(1)(3),  $\nu_k:=(i\circ f_k\circ)_\sharp (L^2_{d_k}\Pi_k/2)$, $k\in \mathbb{N}$ is  a sequence of measures defined on the compact space $P([0,1]\times X))$ with uniformly bounded total masses.  Owing to Prokhorov's theorem  (cf. \cite[Theorem A2.4.I]{DV}), we may assume that $\nu_k$ converges weakly to some measure $\nu$. The last part of the proof of (a) also implies $\supp \nu\subset j(\Gamma(X))$.
Thus, for any $\gamma\in  \Gamma(X)$, if $\overline{\gamma}\in \supp\nu_k$, then $(i\circ f_k\circ)^{-1}(\overline{\gamma})\subset \supp\Pi_k\subset \Gamma(X_k)$. Hence, for any $\zeta\in (i\circ f_k\circ)^{-1}(\overline{\gamma})$,
  \begin{align*}
  |L_{d_k}^2(\zeta)-L_d^2(\gamma)|=|d_k(\zeta(0), \zeta(1))^2-d(f_k(\zeta(0)),f_k(\zeta(1)))^2|\leq 2C\epsilon_k,  \tag{4.23}\label{lengthcontollfordominate}
  \end{align*}
  where $C$ is defined as in Step 1/(2).
    Recall $(i\circ f_k\circ)_\sharp \Pi_k\rightharpoonup \widehat{\Pi}=j_{\sharp}\Pi$, which together with (\ref{lengthcontollfordominate}) implies $(\nu_k)|_{j(\Gamma(X))}\rightharpoonup j_\sharp \left({L^2_d}\Pi/2 \right)$ on $j(\Gamma(X))$. Since both $\supp\nu$ and $\supp j_\sharp \left({L^2_d}\Pi/2 \right)$ are contained in $j(\Gamma(X))$, it is not hard to check that
    \[
  \nu_k=(i\circ f_k\circ)_\sharp \left(\frac{L^2_{d_k}}2\Pi_k\right)\rightharpoonup j_\sharp \left( \frac{L^2_d}{2}\Pi \right)=\nu \text{ on }P([0,1]\times X).\tag{4.24}\label{ldpicoinvergent}
  \]

 Given $t_0\in (0,1)\cap \mathbb{Q}$, for any $\Phi\in C(X)$, define a function on $\mathbb{M}([0,1];X)$ by
\[
\Phi^\delta_{t_0}(\gamma)=\int^1_0 \Phi(\gamma(t))\varphi^\delta_+(t-t_0){\ddd}t,
\]
which extends a continuous function $\Psi^\delta_{t_0}$ on $P([0,1]\times X)$, i.e.,
\[
\Psi^\delta_{t_0}(\mu)=\int_{[0,1]\times X} \Phi(x)\varphi^\delta_+(t-t_0){\ddd}\mu(t,x).
\]
Thus, (\ref{ldpicoinvergent}) combined with $\Psi^\delta_{t_0}(i\circ\gamma)=\Phi^\delta_{t_0}(\gamma)$
  yields
\begin{align*}
&\lim_{k\rightarrow\infty}\int_{\Gamma(X_k)}\Psi^\delta_{t_0}(i\circ f_k\circ \gamma_k)\,{\ddd}\left(\frac{L_{d_k}^2}{2}\Pi_k\right)(\gamma_k)
=\lim_{k\rightarrow\infty}\int_{P([0,1]\times X)}\Psi^\delta_{t_0}(\mu)\,{\ddd}\nu_k(\mu)\\
=&\int_{j(\Gamma(X))}\Psi^\delta_{t_0}(\mu)\,{\ddd}\nu(\mu)=\int_{\Gamma(X)}\Psi^\delta_{t_0}(i\circ\gamma)\,{\ddd}\left( \frac{L^2_{d}}2 \Pi\right)(\gamma)=\int_{\Gamma(X)}\Phi^\delta_{t_0}(\gamma)\,{\ddd}\left( \frac{L^2_{d}}2 \Pi\right)(\gamma).\tag{4.25}\label{partcfirsteq}
\end{align*}
Moreover, since $\lim_{\delta\rightarrow 0}\Phi^\delta_{t_0}(\gamma)=\Phi(e_{t_0} (\gamma))$ for any $\gamma\in \Gamma(X)$, the dominated convergence theorem yields
\begin{align*}
\lim_{\delta\rightarrow 0}\int_{\Gamma(X)}\Phi^\delta_{t_0}(\gamma)\,{\ddd}\left( \frac{L^2_{d}}2 \Pi\right)(\gamma)
=\int_X \Phi(x)\, {\ddd}\left[(e_{t_0})_\sharp\frac{L^2_d}{2}\Pi\right](x).\tag{4.26}\label{partcseconeq}
\end{align*}

On the other hand, the uniform continuity of $\Phi$ implies $\sup_{\gamma_k\in \Gamma(X_k)}| \Phi^\delta_{t_0}(f_k(\gamma_k(t)))-\Phi(f_k(\gamma_k(t_0)))|\rightarrow 0$ as $\delta\rightarrow 0$ and $k\rightarrow \infty$. Since $(f_k)_\sharp \varepsilon_{k,t_0}\rightharpoonup \varepsilon_{t_0}$, one has
\begin{align*}
&\lim_{k\rightarrow\infty,\delta\rightarrow0}\int_{\Gamma(X_k)}\Psi^\delta_{t_0}(i\circ f_k\circ \gamma_k)\,{\ddd}\left(\frac{L_{d_k}^2}{2}\Pi_k\right)(\gamma_k)=\lim_{k\rightarrow\infty,\delta\rightarrow0}\int_{\Gamma(X_k)}\Phi^\delta_{t_0}( f_k\circ \gamma_k )\,{\ddd}\left(\frac{L_{d_k}^2}{2}\Pi_k\right)(\gamma_k)\\
=&\lim_{k\rightarrow\infty}\int_{\Gamma(X_k)}\Phi(  f_k\circ e_{t_0}\circ \gamma_k )\,{\ddd}\left(\frac{L_{d_k}^2}{2}\Pi_k\right)(\gamma_k)=\lim_{k\rightarrow\infty}\int_X\Phi(x)\,{\ddd} (f_k)_\sharp\varepsilon_{k,t_0}(x)= \int_X\Phi(x)\,{\ddd} \varepsilon_{t_0}(x),
\end{align*}
which together with (\ref{partcfirsteq}) and (\ref{partcseconeq})   yields
\[
\int_X \Phi(x) \,{\ddd}\left[(e_{t_0})_\sharp\frac{L^2_d}{2}\Pi\right](x)=\int_X\Phi(x)\,{\ddd}\varepsilon_{t_0}(x).
\]
Hence, $\varepsilon_{t_0}=(e_{t_0})_\sharp [{L^2_d}\Pi/{2}]$.  Therefore,  Step 1/(3) implies  for $t\in (0,1)\cap \mathbb{Q}$,
 \[
 (f_k)_\sharp \varepsilon_{k,t}\rightharpoonup(e_{t})_\sharp [{L^2_d}\Pi/{2}]=:\varepsilon_t.\tag{4.27}\label{uniformlyeplstit}
 \]
We claim that (\ref{uniformlyeplstit}) holds
    for each $t\in (0,1)$. Suppose by contradiction that for some $t_0\in (0,1)$, $(f_k)_\sharp \varepsilon_{k,t_0}$ does not converge weakly to  $\varepsilon_{t_0}$. Thus, by the same argument as Step 1/(3),  one could choose a subsequence $((f_{k_l})_\sharp\varepsilon_{k_l,t_0})_l$ convergent weakly to some measure $\tilde{\varepsilon}_{t_0}(\neq \varepsilon_{t_0})$. However,  using the same argument as above, one could get
$\tilde{\varepsilon}_{t_0}=(e_{t_0})_\sharp (L^2_d\Pi/2)=\varepsilon_{t_0}$, which is a contradiction.

\smallskip

\noindent \textbf{Step 3.}  (v) follows by the same method used in Villani \cite[p.786-787]{Vi} together with Proposition \ref{strongconverges} and Statement (iii), while (vi) follows by a similar argument as in Villani \cite[p.786]{Vi} together with Theorem \ref{vectorfieldscontrol} and Proposition \ref{Ascloghtipo}.
\end{proof}

In view of Corollary \ref{compactwasserin}, the structure of Wasserstein spaces over {an}  irreversible space of finite reversibility is similar to the one over a reversible space, in which case
one can consider Strum's $\mathbb{D}$-distance (cf. \cite{Sturm-1}). In particular, the equivalence between $\mathbb{D}$-convergence and $W_2$-convergence (cf. Ambrosio, Gigli and Savar\'e  \cite[Proposition 2.7]{AGS2}) remains valid. We leave the formulation of such statements to the
interested reader.

\section{Synthetic Ricci curvature on forward metric-measure spaces}\label{Ricccurv}

\subsection{Weak ${\rm{CD}}(K,N)$ spaces}

\begin{definition}\label{dispconvexfirst}
Let $(X,d,\nu)$ be a locally compact forward metric-measure space, where $\nu$ is locally finite; let $U$ be a continuous convex function with $U(0)=0$. Given a measure $\mu$ on $X$ with compact support, let
$\mu=\rho \nu+\mu_s$
be the Lebesgue decomposition of $\mu$ with resect to $\nu$, i.e., the absolutely continuous part $\rho \nu$ and singular part $\mu_s$. Then

\begin{itemize}

\item the integral functional $U_\nu$, with nonlinearity $U$ and reference measure $\nu$, is defined as
\[
U_\nu(\mu):=\int_X U(\rho(x))\,{\ddd}\nu(x)+U'(\infty)\,\mu_s[X],
\]
where $U'(\infty):=\lim_{r\rightarrow \infty}U(r)/r$.

\smallskip

\item given $\pi\in P(X\times X)$ and  a measurable function $\beta:X\times X\rightarrow (0,\infty]$, define the integral functional $U^\beta_{\pi,\nu}$, with nonlinearity $U$, reference measure $\nu$, coupling $\pi$ and distortion coefficient $\beta$ as
\[
U^\beta_{\pi,\nu}(\mu):=\int_{X\times X}U\left( \frac{\rho(x)}{\beta(x,y)} \right)\beta(x,y)\pi({\ddd}y|x)\nu({\ddd}x)+U'(\infty)\,\mu_s[X],
\]
 where $\pi({\ddd}y|x)$ denotes the disintegration of $\pi$ with respect to $\mu$  (i.e., $\pi({\ddd}x {\ddd}y)=\mu({\ddd}x)\,\pi({\ddd}y|x)$), which is a family of probability measures on $X$.
\end{itemize}
\end{definition}

\begin{remark}\label{firstUv} First, if $\beta\equiv1$, then $U^\beta_{\pi,\nu}(\mu)=U_\nu(\mu)$. Second, if $\mu$ is absolutely continuous w.r.t. $\nu$ (i.e., $\mu_s=0$), then Villani \cite[Lemma 29.6]{Vi} furnishes
\begin{align*}
U^{\beta}_{\pi,\nu}(\mu)=\int_{X\times X}U\left(  \frac{\rho(x)}{\beta(x,y)}\right)\frac{\beta(x,y)}{\rho(x)}\pi({\ddd}x{\ddd}y),
\end{align*}
where we use the convention $U(0)/0=U'(0)\in [-\infty,\infty)$ and $U(\infty)/\infty=U'(\infty)\in (-\infty,\infty]$.
\end{remark}


\begin{definition}\label{defbeta} Let $(X,d)$ be a forward metric space.
\begin{itemize}

\item[(1)] Given $N\in (1,\infty)$, $K\in \mathbb{R}$ and $t\in (0,1]$,  define
\[
\beta^{(K,N)}_t(x,y):=\left\{
	\begin{array}{lll}
	& \infty, & \ \ \ \text{if } K>0 \text{ and } d(x,y)\geq\pi\sqrt{\frac{N-1}{K}}, \\
	\\
	&\left(\frac{\mathfrak{s}_{K,N}(td(x,y))}{t\mathfrak{s}_{K,N}(d(x,y))}\right)^{N-1}, & \ \ \ \text{if it is well-defined},
	\end{array}
	\right.
\]
where $\mathfrak{s}_{K,N}(t)$ is defined as in (\ref{skdefinition}). For $t=0$, set $\beta^{(K,N)}_0(x,y)=1$.

\smallskip

\item[(2)]
For $N\in\{1,\infty\}$ and $t\in [0,1]$, define
\[
\beta^{(K,1)}_t(x,y)=\lim_{N\downarrow 1}\beta^{(K,N)}_t(x,y),\ \ \beta^{(K,\infty)}_t(x,y)=\lim_{N\uparrow \infty}\beta^{(K,N)}_t(x,y).
\]
In particular, if $t\in (0,1]$, then
\[
\beta^{(K,1)}_t(x,y)=\left\{
	\begin{array}{lll}
	& \infty, & \ \ \ \text{if } K>0, \\
	\\
	&1, & \ \ \ \text{if }K\leq 0,
	\end{array}
	\right.
\ \ \ \ \ \ \ \beta^{(K,\infty)}_t(x,y)=e^{\frac{K}{6}(1-t^2)d(x,y)^2}.
\]
\end{itemize}
\end{definition}

\begin{remark}\label{limitdefinitionofbeta}Since $\beta^{(K,N)}_t(x,y)$ is usually  asymmetric, we introduce a notation $\breve{\beta}^{(K,N)}_t:=\beta^{(K,N)}_t\circ\mathcal {E}$, where $\mathcal {E}(x,y):=(y,x)$. Similarly, for
 a probability measure $\pi\in P(X\times X)$, set $\breve{\pi}=\mathcal {E}_\sharp \pi$, i.e., the probability measure obtained
from $\pi$ by exchanging $x$ and $y$.
Moreover, given a  continuous convex function $U$ with $U(0)=0$, set
\[
U^{\beta_t^{(K,N)}}_{\pi,\nu}(\mu):=\lim_{N'\downarrow N}U^{\beta_t^{(K,N')}}_{\pi,\nu}(\mu),\quad U^{\breve{\beta}_t^{(K,N)}}_{\breve{\pi},\nu}(\mu):=\lim_{N'\downarrow N}U^{\breve{\beta}_t^{(K,N')}}_{\breve{\pi},\nu}(\mu),
\]
if
 one of the  following threshold cases happens:

\smallskip

 \ \ \   (1) $N=1$;\ \ \ (2) $K>0$, $N\in (1,\infty)$ and $\diam(X)=\pi\sqrt{\frac{N-1}{K}}$.



 \end{remark}

\begin{definition}[{Villani \cite{Vi}}]\label{DC-N-definition}  Given $N\in [1,\infty]$, the class $DC_N$ is defined as the set of continuous convex functions $U:\mathbb{R}_+\rightarrow \mathbb{R}$ twice continuously  differentiable on $(0,\infty)$ with $U(0)=0$, and satisfies any one of the following equivalent
differential conditions:

\smallskip

(i) $p_2+p/N\geq 0$, where $p(r):=rU'(r)-U(r)$ and $p_2(r):=rp'(r)-p(r)$;

\smallskip

(ii) $p(r)/r^{1-1/N}$ is a nondecreasing function in $r$;

\smallskip

(iii) \[
u(\delta):=\left\{
	\begin{array}{lll}
	&\delta^NU(\delta^{-N}) \text{ for }\delta>0, & \ \ \ \text{ if } N<\infty, \\
	\\
	&e^\delta U(e^{-\delta}) \text{ for }\delta\in \mathbb{R}, & \ \ \ \text{ if } N=\infty,
	\end{array}
	\right.\text{ is a convex function of }\delta;
\]

\end{definition}

In this section we focus on the Wasserstein distance
of order $2$. Due to Remark \ref{existencedynamicaloptimal},
an optimal transference plan $\pi$ (with respect to $d^2$) is always associated with a displacement interpolation $(\mu_t)_{0\leq t\leq 1}$ because there is a dynamical optimal transference plan $\Pi$ such that $\mu_t=(e_t)_\sharp\Pi$ and $\pi=(e_0,e_1)_\sharp\pi$.


\begin{definition}\label{weakcdknspace}
Given $K\in \mathbb{R}$ and $N\in[1,\infty]$, a forward boundedly compact, $\sigma$-finite, forward geodesic-measure space $(X,d,\nu)$ is said to {\it satisfy the weak curvature-dimension condition/bound ${\CD}(K,N)$}, or to {\it be a weak ${\CD}(K,N)$ space}, if
for any two $\mu_0,\mu_1\in P_c(X,\nu)$, there exists a displacement interpolation $(\mu_t)_{0\leq t\leq 1}$ and an associated optimal transference plan $\pi$ from $\mu_0$  to $\mu_1$ (w.r.t. $d^2$) such that for all $U\in DC_N$ and for all $t\in [0,1]$,
\[
U_\nu(\mu_t)\leq (1-t)U^{\beta_{1-t}^{(K,N)}}_{\pi,\nu}(\mu_0)+t\,U^{\breve{\beta}_{t}^{(K,N)}}_{\breve{\pi},\nu}(\mu_1),\ \forall\,t\in[0,1].\tag{5.1}\label{weakcdkncon}
\]
Here, $P_c(X,\nu)$ is the collection of compactly supported probability measures $\mu$ with $\supp\mu\subset \supp\nu$.
\end{definition}

In the  Finslerian case, the following  result is valid.
\begin{theorem}[Ohta \cite{O,O1}]\label{OTARic}For an $n(\geq 2)$-dimensional forward complete Finsler metric-measure manifold $(M,F,\mathfrak{m})$, the corresponding forward geodesic-measure space $(M,d_F,\mathfrak{m})$ is a weak ${\CD}(K,N)$ space for some $N\in [n,\infty]$ if and only if
$\mathbf{Ric}_N(y)\geq K$ for any  $y\in SM$.
\end{theorem}

According to Villani \cite{Vi}, a function $U:\mathbb{R}_+\rightarrow \mathbb{R}_+$ satisfies {\it at most polynomial growth}, i.e., there exists a constant $C>0$ with
\[
r U'(r)\leq C\left( U(r)+r  \right),\ \forall\,r>0.\tag{5.2}\label{mostpoly}
\]
Particularly, $U$ grows at most polynomially if one of the following condition holds
\begin{itemize}
\item[(1)] $U$ is Lipschitz;

\item[(2)] $U$ is twice differentiable and behaves at infinity like $a r\log r+b r$;

\item[(3)] $U$ is twice differentiable and behaves at infinity like a polynomial.

\end{itemize}

\begin{proposition}\label{CDKNstringent}
Given $K\in \mathbb{R}$ and $N\in[1,\infty]$, let  $(X,d,\nu)$  be forward boundedly compact, $\sigma$-finite, forward geodesic-measure space. The following statements hold:

\smallskip

\noindent{\rm(i)} for $K'\geq K$ and $N'\leq N$, if $(X,d,\nu)$ is a weak ${\CD}(K',N')$ space, it is also a weak ${\CD}(K,N)$ space;

\smallskip

\noindent{\rm(ii)} $(X,d,\nu)$ is a weak ${\CD}(K,N)$ space if and only if  \eqref{weakcdkncon} is valid for those nonnegative $U\in DC_N$ with:
\begin{itemize}

\item  $U$ is Lipschitz if $N<\infty$;

\item  $U$ is locally Lipschitz and $U(r)=a r\log r+b r$ ($a\geq 0$, $b\in \mathbb{R}$) for $r$ large enough if $N=\infty$.
\end{itemize}

\smallskip

\noindent{\rm(iii)} $(X,d,\nu)$ is a weak ${\CD}(K,N)$ space if and only if \eqref{weakcdkncon} holds for $\mu_0,\mu_1\in P_c(X,\nu)$ which are absolutely continuous (w.r.t. $\nu$) with continuous densities.

 \end{proposition}

\begin{proof} (i) and (ii) follow from the same arguments as in Villani \cite[Propositions 29.10 \& 29.12]{Vi}, respectively. For (iii), it suffices to show the ``$\Leftarrow$" part.
Given $\mu_0,\mu_1\in P_c(X,\nu)$,  owing to (ii), it is enough to show
that   there exists a displacement interpolation $(\mu_t)_{0\leq t\leq 1}$ and an associated optimal  transference plan  $\pi$ of $(\mu_0,\mu_1)$ such that for all nonnegative $U\in DC_N$ with (\ref{mostpoly}),
\[
U_\nu(\mu_t)\leq (1-t)U^{\beta_{1-t}^{(K,N)}}_{\pi,\nu}(\mu_0)+t\,U^{\breve{\beta}_{t}^{(K,N)}}_{\breve{\pi},\nu}(\mu_1),\ \forall\,t\in[0,1].\tag{5.3}\label{growpolycdkn}
\]
We notice that a similar proof has been presented by Villani \cite[Corollary 29.23]{Vi} in the compact case. Here, by taking advantage of Lemma \ref{continuonUvUpi} and Theorem \ref{stabilityoptimal}, we are going to prove (iii) in the forward boundedly compact setting.
%
Since  such a space $X$ might be  noncompact, our strategy {is} to construct a compact metric-measure space $\mathfrak{K}$ with suitable properties so that the argument can be carried out  on $\mathfrak{K}$.

Now let $\mathcal {K}$ be a compact set and  $\epsilon$ be a small positive number as in Lemma \ref{compactmetrc}.
On the compact forward metric-measure space $\mathfrak{K}:=(\mathcal {K},d|_\mathcal {K},\nu|_\mathcal {K})$, since $\supp\mu_i\subset \mathcal {K}\cap\supp\nu=\supp \nu|_\mathcal {K}$ for $i=0,1$,
Lemma \ref{continuonUvUpi}/(iii) yields two sequences of probability measures of continuous density $\mu_{k,i}=\rho_{k,i}\nu|_\mathcal {K}$ converging weakly in $P(\mathcal {K})$ to $\mu_i|_\mathcal {K}=\mu_i$ for $i=0,1$. Moreover,  Remark \ref{suppregularing} together with Lemma \ref{compactmetrc} yields a small $\epsilon>0$ such that (by passing to a subsequence)
\[
\suppor \rho_{k,i}\subset  ({\supp \mu_i}|_{\mathcal {K}})^{\epsilon} =({\supp \mu_i})^{\epsilon}\subset  \mathcal {K}, \text{ for all }k,i.\tag{5.4}\label{supporersetrealtion}
\]
Now we naturally extend $\mu_{k,i}$ on $X$ by setting $\mu_{k,i}[X\backslash\mathcal {K}]=0$. Thus, $\mu_{k,i}\in P_c(X,\nu)$   are absolutely continuous (w.r.t. $\nu$) with continuous densities.

By Remark \ref{limitdefinitionofbeta}, we may assume that $\beta^{(K,N)}_{t}$ is continuous. On $(X,d,\nu)$,
for arbitrary $k\in \mathbb{N}$, the assumption furnishes  a displacement interpolation $(\mu_{k,t})_{0\leq t\leq 1}$ and an associated optimal transference plan $\pi_k\in P(X\times X)$ of $(\mu_{k,0},\mu_{k,1})$ such that for any  $U\in DC_N$,
\[
U_\nu(\mu_{k,t})\leq (1-t)U^{\beta_{1-t}^{(K,N)}}_{\pi_k,\nu}(\mu_{k,0})+t\,U^{\breve{\beta}_{t}^{(K,N)}}_{\breve{\pi}_k,\nu}(\mu_{k,1}),\ \forall\,t\in [0,1].\tag{5.5}\label{continscur1}
\]

Due to (\ref{supporersetrealtion}) and the construction of $\mathcal {K}$, it is easy to check that for each $k$, the quantities $(\mu_{k,i})_{i=0,1}$, $\pi_k$ and $(\mu_{k,t})_{0\leq t\leq 1}$ still satisfy Lemma \ref{compactmetrc}/(iii)-(v). In particular, $\pi_k$ and $\mu_{k,t}$ can be viewed as the optimal transference plan of $(\mu_{k,0},\mu_{k,1})$ and associated displacement interpolation defined on the compact forward metric space  $\mathfrak{K}$ (i.e., Lemma \ref{compactmetrc}/(v) and Definition \ref{metrcdefol}). Although  $\mathfrak{K}$ may be not a forward geodesic space, the same argument as in Theorem \ref{stabilityoptimal} ($\mathcal {X}_k=\mathcal {X}=\mathfrak{K}$, $f_k=\id$) together with the construction of $\mathcal {K}$ still furnishes
 (by passing to a subsequence) that
\begin{itemize}
\item[(a)]  $(\pi_k)_k$ converges weakly to an optimal transference plan $\pi$ of coupling $({\mu_0}|_\mathcal {K},{\mu_1}|_\mathcal {K})=({\mu_0},{\mu_1})$;

\item[(b)] $(\mu_{k,t})_k$ converges weakly to $\mu_{t}$, which is a displacement interpolation with respect to $\pi$.
\end{itemize}
The weak convergence in (a) (resp., (b)) is defined on $P(\mathcal {K}\times \mathcal {K})$ (resp., $P(\mathcal {K})$). However, in view of
 \[
 \supp\mu_{k,i}\subset \mathcal {K}, \ \supp \mu_t\subset  \mathcal {K},\ \supp \pi_k\subset  \mathcal {K}\times \mathcal {K},
 \]
both (a) and (b) remain valid on $P(X\times X)$ and $P(X)$, respectively.
Then Lemma \ref{compactmetrc}/(vi) indicates that $\pi$ is also an optimal transference plan on $(X,d)$ and $(\mu_{t})_{0\leq t\leq 1}$ is the associated displacement interpolation.

On the other hand, note that $\supp\mu_{k,i}\subset \supp \nu|_\mathcal {K}$ and hence, $\pi_k\subset \supp \nu|_\mathcal {K}\times \supp \nu|_\mathcal {K}$. For any nonnegative $U\in DC_N$ with (\ref{mostpoly}),
by considering the compact forward metric-measure space $\mathfrak{K}$ and recalling the construction of $\mu_{k,i}$,
Lemma \ref{continuonUvUpi}/(i)(iii) gives
\begin{align*}
U_{\nu}(\mu_{t})=U_{\nu|_\mathcal {K}}(\mu_{t})\leq \underset{k\rightarrow \infty}{\lim\inf}\, U_{\nu|_\mathcal {K}}(\mu_{k,t})=\underset{k\rightarrow \infty}{\lim\inf}\, U_{\nu}(\mu_{k,t}),\tag{5.6}\label{firstweakcontrol}\\
 \underset{k\rightarrow \infty}{\lim\sup}\,U^{\beta^{(K,N)}_{1-t}}_{\pi_k,\nu}(\mu_{k,0})\leq U^{\beta^{(K,N)}_{1-t}}_{\pi,\nu}(\mu_{0}),\ \ \underset{k\rightarrow \infty}{\lim\sup}\,U^{\breve{\beta}^{(K,N)}_{t}}_{\breve{\pi}_k,\nu}(\mu_{k,1})\leq U^{\breve{\beta}^{(K,N)}_{t}}_{\breve{\pi},\nu}(\mu_{1}),\tag{5.7}\label{secondweakcontrol}
\end{align*}
which together with (\ref{continscur1}) yield (\ref{growpolycdkn}).
\end{proof}

A subset $A$ of a   forward geodesic space  $(X,d)$ is said to be {\it totally convex} if every minimal geodesic whose endpoints belong to $A$ is entirely contained in $A$.
The same argument as in Villani \cite[Proposition 30.1]{Vi} together with the proof of Lemma \ref{compactmetrc} yields the following result.
\begin{proposition}\label{cosntmupcd}
Let $(X,d,\nu)$ be a weak ${\CD}(K,N)$ space. Then
\begin{itemize}
\item[(i)] if $A\subset X$ is a totally convex closed subset, then $(A,d|_A,\nu|_A)$ is also a weak ${\CD}(K,N)$ space;

\item[(ii)] for any $\alpha>0$, $(X,d,\alpha \nu)$ is a weak ${\CD}(K,N)$ space;

\item[(iii)] for any $\alpha>0$, $(X,\alpha d,\nu)$ is a weak ${\CD}(\alpha^{-2}K,N)$ space.

\end{itemize}
\end{proposition}

\begin{theorem}\label{supporsetcdkn}
A forward boundedly compact, $\sigma$-finite, forward geodesic-measure space $(X,d,\nu)$ is a weak ${\CD}(K,N)$ space if and only if $(\supp\nu,d|_{\supp\nu},\nu|_{\supp\nu})$ is a weak ${\CD}(K,N)$ space.
\end{theorem}
\begin{proof} Since the ``$\Leftarrow$" part is direct, we only show the ``$\Rightarrow$" part.
Assume that $(X,d,\nu)$ is a weak ${\CD}(K,N)$ space and  $\mu_0,\mu_1\in P_c(X,\nu)$. Since $(X,d)$ could be noncompact,
 we will use a similar strategy to that in the proof of Proposition \ref{CDKNstringent}/(iii) and consider a compact forward metric-measure space $\mathfrak{K}:=(\mathcal {K},d|_\mathcal {K},\nu|_\mathcal {K})$, where $\mathcal {K}$ is  defined as in Lemma \ref{compactmetrc}.


For each $i=0,1$, the same argument as in the proof of Proposition \ref{CDKNstringent}/(iii) furnishes a sequence of probability measures $(\mu_{k,i})_{k}$ converging weakly to $\mu_i$ in $P(X)$
    such that
\begin{itemize}
\item $\suppor\rho_{k,i}\subset\mathcal {K}$ and $\mu_{k,i}=\rho_{k,i}\nu$ is a probability measure with continuous density $\rho_{k,i}$ on $X$;

\item for any sequence $(\pi_{k})_k$ converging weakly to $\pi$ in $P(\mathcal {K}\times \mathcal {K})$    such that $\pi_k$ admits $\mu_{k,i}$ as first marginal   and $\supp \pi_k\subset \supp \nu|_\mathcal {K}\times \supp \nu|_\mathcal {K}$, for any continuous positive function $\beta$ on $X\times X$
and for any nonnegative $U\in DC_N$ with (\ref{mostpoly}),
there holds
\[
\underset{k\rightarrow \infty}{\lim\sup}\,U^\beta_{\pi_k,\nu}(\mu_{k,i})\leq U^\beta_{\pi,\nu}(\mu_i).
\]
\end{itemize}

Since $(X,d,\nu)$ is a weak ${\CD}(K,N)$ space,  for each $k$, there exists an optimal transference plan $\pi_k$ of $(\mu_{k,0},\mu_{k,1})$  and an associated displacement interpolation
  $(\mu_{k,t})_{0\leq t\leq 1}$ such that for any nonnegative $U\in DC_N$ with (\ref{mostpoly}),
\[
U_\nu(\mu_{k,t})\leq (1-t) U^{\beta^{(K,N)}_{1-t}}_{\pi_k,\nu}(\mu_{k,0})+t U^{\breve{\beta}^{(K,N)}_t}_{\breve{\pi}_k,\nu}(\mu_{k,1}),\ \forall\,t\in[0,1].\tag{5.8}\label{needlimitcdkn}
\]
Furthermore, since $\supp\pi_k\subset\mathcal {K}\times \mathcal {K}$, it follows from the construction of $\mathcal {K}$ (see Lemma \ref{compactmetrc}) that $\mu_{k,i}$, $\pi_k$ and $\mu_{k,t}$ can be viewed as corresponding (optimal) quantities defined on $\mathfrak{K}$. Particularly, they are compactly supported in $\mathcal {K}$.
The same argument as in the proof of Proposition \ref{CDKNstringent}/(iii) yields  (by passing to a subsequence) that
\begin{itemize}
\item[(a)] $(\pi_k)_k$ converges weakly in $P(X\times X)$ to an optimal transference plan $\pi$ of coupling $({\mu_0},{\mu_1})$;

\item[(b)] $(\mu_{k,t})_k$ converges weakly  to $\mu_{t}$ in $P(X)$, which is a displacement interpolation with respect to $\pi$.
\end{itemize}
Now we claim that the following  statements are true:
\begin{itemize}
\item[(C-1)] for all $k$ and all $t\in [0,1]$, there holds
\[
\supp\pi\cup\supp\pi_k\subset  \supp \nu|_\mathcal {K} \times \supp \nu|_\mathcal {K},\ \supp\mu_{k,t}\cup\supp\mu_{t}\subset \supp \nu|_\mathcal {K};\tag{5.9}\label{relationshipofsuppot}
\]
\item[(C-2)]  $(\supp\nu, d|_{\supp\nu})$ is a forward geodesic space;

\smallskip

\item[(C-3)] on $(\supp\nu, d|_{\supp\nu},\nu)$, $\pi$ is still an optimal transference plan of $(\mu_0,\mu_1)$ and $(\mu_t)_{0\leq t\leq 1}$ is the associated displacement interpolation.

\end{itemize}

\smallskip If all the claims are true, then by (C-2) and (C-3),   there exists a displacement interpolation $(\mu_t)_{0\leq t\leq 1}$ and an associated optimal  transference plan  $\pi$ of $(\mu_0,\mu_1)$ on the forward boundedly compact, $\sigma$-finite, forward geodesic-measure space $(\supp\nu, d|_{\supp\nu},\nu)$. On the other hand,  since  $\mathfrak{K}$ is compact, (C-1) together with Lemma \ref{continuonUvUpi} again furnishes (\ref{firstweakcontrol}) and (\ref{secondweakcontrol}), which together with (\ref{needlimitcdkn}) yield
\[
U_\nu(\mu_t)\leq (1-t)U^{\beta_{1-t}^{(K,N)}}_{\pi,\nu}(\mu_0)+tU^{\breve{\beta}_{t}^{(K,N)}}_{\breve{\pi},\nu}(\mu_1).
\]
This concludes the proof; hence, it remains to show (C-1)--(C-3).

\noindent \textbf{(C-1).} The first part of (\ref{relationshipofsuppot}) follows directly by (a) and $\supp\pi_k\subset \supp\nu|_\mathcal {K} \times \supp \nu|_\mathcal {K}$ directly. For the second part,
since $H(r):=r\log r\in DC_N$ satisfies (\ref{mostpoly}), Proposition \ref{CDKNstringent}/(i)  and  (\ref{needlimitcdkn})  imply
\[
H_\nu(\mu_{k,t})\leq (1-t) H^{\beta^{(K,\infty)}_{1-t}}_{\pi_k,\nu}(\mu_{k,0})+t H^{\breve{\beta}^{(K,\infty)}_t}_{\breve{\pi}_k,\nu}(\mu_{k,1}),\tag{5.10}\label{Hvcdkn}
\]
which together with a direct calculation furnishes
\begin{align*}
H_\nu(\mu_{k,t})\leq (1-t)H_\nu(\mu_{k,0})+tH_\nu(\mu_{k,1})-K\frac{t(1-t)}{2}W_2(\mu_{k,0},\mu_{k,1})^2.
\end{align*}
Recall that $\mu_{k,i}=\rho_{k,i}\nu$, $i=0,1$ are the probability measures with continuous densities supported in the compact set $\mathcal {K}$ and hence, $H_\nu(\mu_{k,t})$ is finite for any $t\in [0,1]$ and $k\in \mathbb{N}$. Since $H'(\infty)=\infty$, the definition of
$H_\nu(\mu_{k,t})$ implies that
$\mu_{k,t}$ is absolutely continuous with respect to $\nu$. Thus, $\supp\mu_{k,t}\subset\supp \nu\cap \mathcal {K}=\supp \nu|_\mathcal {K}$, which together with (b) implies  the second part of (\ref{relationshipofsuppot}).

\smallskip

\noindent \textbf{(C-2).}
Given any two points $x_0,x_1\in \supp \nu$, consider two probability measures $\mu_i:=\delta_{x_i}$, $i=0,1$, which are compactly supported in $\supp\nu$. In view of (a), (b) and (C-1), there is a geodesic $(\mu_t)_{0\leq t\leq 1}$ in $(P_2(X),W_2)$ which is supported in $\supp\nu$ for any $t\in [0,1]$. Let $\Pi$ be the associated dynamical optimal transference plan. For a fixed $t_0\in [0,1]$, set
$A_{t_0}:=\{\gamma\in \Gamma_{x_0\rightarrow x_1}|\, \gamma(t_0)\notin \supp\nu\}$, where $\Gamma_{x_0\rightarrow x_1}$ is the set of constant-speed minimal geodesics from $x_0$ to $x_1$ in $X$.  Thus, $\Pi\left[A_{t_0}  \right]=\mu_{t_0}[X\backslash \supp \nu]=0$.

Choose a countable dense sequence $(t_\alpha)_\alpha\subset [0,1]$ and set $A:=\{\gamma\in \Gamma_{x_0\rightarrow x_1}|\,\exists t_\alpha \text{ such that }\gamma(t_\alpha)\notin \supp\nu\}$.
Thus, $ {\Pi}[A]\leq \sum_\alpha {\Pi}[A_{t_\alpha}]=0$.
On the other hand,  if a curve $\gamma\in \Gamma_{x_0\rightarrow x_1}$ satisfies $\gamma(t_0)\notin\supp\nu$ for some $t_0\in [0,1]$,  there exists some $t_\alpha$ such that $\gamma(t_\alpha)\notin \supp\nu$ because $\supp\nu$ is closed.
Hence,
$A=\{\gamma\in \Gamma_{x_0\rightarrow x_1}|\,\exists t\in [0,1] \text{ such that }\gamma(t)\notin \supp\nu\}$.

From above, $\Pi[\Gamma_{x_0\rightarrow x_1}\backslash A]=1$ and hence,
there is at least one $\gamma\in \Gamma_{x_0\rightarrow x_1}$ such that $\gamma([0,1])\subset \supp\nu$, which implies  (C-2).

\smallskip

\noindent \textbf{(C-3).}
From (C-1) and (C-2), $\pi$ is an optimal transference plan of $({\mu_0},{\mu_1})$ on the forward geodesic space $(\supp\nu,d|_{\supp\nu})$. Thus, it suffices to find a dynamical optimal transference plan $\widetilde{\Pi}$ on $\Gamma(\supp \nu)$ such that $\pi=(e_0,e_1)_\sharp \widetilde{\Pi}$ and $\mu_t=(e_t)_{\sharp}\widetilde{\Pi}$.

Recall that $(\mu_t)_{0\leq t\leq 1}$ is a displacement interpolation of $\pi$ on $X$. Hence, there exists a dynamical optimal transference plan ${\Pi}$ on $X$ such that $\mu_t=(e_t)_\sharp{\Pi}$. It follows from the same argument as in (C-2)   that
$A:=\{\gamma\in \Gamma(X)|\,\exists t\in [0,1] \text{ such that }\gamma(t)\notin \supp\nu\} \text{ is ${\Pi}$-negligible}$.
Then the proof is done by setting $\widetilde{\Pi}:={\Pi}|_{\Gamma(X)\backslash A}$. \end{proof}

\subsection{Diameter and volume control}\label{Riccappli}

\begin{theorem}\label{diametercontrol}
If $(X,d,\nu)$ is a weak ${\CD}(K,N)$ space with $K>0$ and $N\in (1,\infty)$, then
\[
\diam(\supp \nu)\leq \pi\sqrt{\frac{N-1}{K}}.
\]
\end{theorem}
\begin{proof}
Suppose by contradiction that there would exist   $x_0,x_1\in \supp \nu$ such that $d(x_0,x_1)>D_{K,N}:=\pi\sqrt{{(N-1)}/{K}}$. Choose a small $r>0$ such that  $d(x,y)>D_{K,N}$ for all $x\in B^+_{x_0}(r)$ and $y\in B^+_{x_1}(r)$.
Take $\mu_i:=\rho_i\nu$ for $i=0,1$, where $\rho_i(x):= {\textbf{1}_{B^+_{x_i}(r)}}/{\nu\left[  B^+_{x_i}(r)\right]}$.

Let $\pi$ be an optimal transference plan of  $(\mu_0,\mu_1)$ and let $(\mu_t)_{0\leq t\leq 1}$ be the associated  displacement interpolation.
 Thus,  for any $U\in DC_N$,
Remark \ref{firstUv}  yields
\begin{align*}
U^{\beta^{(K,N)}_{1-t}}_{\pi,\nu}(\mu_0)
=\int_{B^+_{x_0}(r) \times B^+_{x_1}(r)}U\left( \frac{\rho_0(x)}{\beta^{(K,N)}_{1-t}(x,y)} \right)\frac{\beta^{(K,N)}_{1-t}(x,y)}{\rho_0(x)} \pi({\ddd}x{\ddd}y)=U'(0).
\end{align*}
Similarly,
$U^{\breve{\beta}^{(K,N)}_{t}}_{\breve{\pi},\nu}(\mu_1)=U'(0)$.
Hence, (\ref{weakcdkncon}) furnishes $U_\nu(\mu_t)\leq U'(0)$.

Choose $U(r)=-r^{1-1/N}$. Since $U'(0)=-\infty$, the above inequality yields $U_\nu(\mu_t)=-\infty$. On the other hand,
set $\mu_t=\rho_t \nu+(\mu_t)_s$. Then Jensen's inequality yields
\begin{align*}
U_\nu(\mu_t)&=\int_X U(\rho_t(x)){\ddd}\nu(x)+U'(\infty)(\mu_t)_s[X]=\int_X U(\rho_t(x)){\ddd}\nu(x)\geq U\left( \int_X \rho_t(x){\ddd}\nu(x) \right)\\
&=-\left( \int_X \rho_t(x){\ddd}\nu(x)  \right)^{1-1/N}\geq -1,
\end{align*}
which is a contradiction. Therefore, the theorem follows.
\end{proof}


\begin{theorem}[Brunn-Minkowski inequality]\label{BrunnMinkowineq}
Given
$K\in \mathbb{R}$ and $N\in [1,\infty]$, let $(X,d,\nu)$ be a
weak ${\CD}(K,N)$ space.
For any  compact subsets $A_0,A_1\subset\supp \nu$ and $t\in (0,1)$, denote by $[A_0,A_1]_t$ the set of all $t$-barycenters of $A_0$ and $A_1$, which is the set of $y\in X$ that can be written as $\gamma(t)$, where $\gamma$ is a minimal, constant-speed geodesic with $\gamma(0)\in A_0$ and $\gamma(1)\in A_1$.

\smallskip

\begin{itemize}
\item[(i)]   If $N<\infty$, then
\begin{align*}
\nu\left[ [A_0,A_1]_t \right]^{\frac1N}\geq (1-t)\left[ \inf_{(x,y)\in A_0\times A_1}\beta^{(K,N)}_{1-t}(x,y)^{\frac1N}  \right] \nu[A_0]^{\frac1N}+t\left[ \inf_{(x,y)\in A_0\times A_1}\beta^{(K,N)}_{t}(x,y)^{\frac1N}  \right] \nu[A_1]^{\frac1N}.
\end{align*}
Moreover, if  $K\geq 0$, then
\[
\nu\left[ [A_0,A_1]_t \right]^{\frac1N}\geq (1-t)\,\nu[A_0]^{\frac1N}+t\,\nu[A_1]^{\frac1N}.
\]

\smallskip

\item[(ii)] If $N=\infty$, then
\begin{align*}
\log\left( \frac{1}{\nu\left[ [A_0,A_1]_t \right]} \right)\leq& (1-t)\log\left( \frac{1}{\nu\left[ A_0 \right]} \right)+t\log\left( \frac{1}{\nu\left[ A_1 \right]} \right)\\
&+\frac{t(1-t)}{2}\left[K_-\max_{(x,y)\in A_0\times A_1}d(x,y)^2-K_+\min_{(x,y)\in A_0\times A_1}d(x,y)^2 \right],
\end{align*}
where $K_+:=\max\{K,0\}$ and $K_-:=-\min\{K, 0\}$.
\end{itemize}
\end{theorem}


\begin{proof}\textbf{(i)}
Set
$\mu_i:=\rho_i\nu:= {\textbf{1}_{A_i}}\nu/{\nu[A_i]}$ for $i=0,1$.
Thus,  there exists  a displacement interpolation $(\mu_t)_{0\leq t\leq 1}$ and an associated optimal transference plan $\pi$ of coupling $(\mu_0,\mu_1)$ such that for any $U\in DC_N$,
\[
U_\nu(\mu_t)\leq (1-t) U^{\beta^{(K,N)}_{1-t}}_{\pi,\nu}(\mu_0)+tU^{\breve{\beta}^{(K,N)}_{t}}_{\breve{\pi},\nu}(\mu_1),\ \forall\,t\in[0,1].\tag{5.11}\label{oldCDK}
\]
As before, let $\mu_t=\rho_t \nu+(\mu_t)_s$. For $U(r)=-r^{1-1/N}$,  (\ref{oldCDK}) together with Remark \ref{firstUv} yields
\begin{align*}
\int_X \rho_t(x)^{1-\frac1N}\nu({\ddd}x)\geq (1-t)\int_{X\times X}\rho^{-\frac1N}_0(x)\beta^{(K,N)}_{1-t}(x,y)^{\frac1N}\pi({\ddd}x{\ddd}y)+t\int_{X\times X}\rho^{-\frac1N}_1(y)\beta^{(K,N)}_{t}(x,y)^{\frac1N}\pi({\ddd}x{\ddd}y).
\end{align*}


Since $\pi$ is supported in $A_0\times A_1$ and has marginals $\mu_0$ and $\mu_1$, the above inequality furnishes
\begin{align*}
&\int_X \rho_t(x)^{1-\frac1N}\nu({\ddd}x)\\
\geq &(1-t) \left[ \inf_{(x,y)\in A_0\times A_1}\beta^{(K,N)}_{1-t}(x,y)^{\frac1N}  \right]\int_X\rho^{-\frac1N}_0(x)\mu_0({\ddd}x)+t \left[ \inf_{(x,y)\in A_1\times A_0}\beta^{(K,N)}_{t}(x,y)^{\frac1N}  \right]\int_X\rho^{-\frac1N}_1(y)\mu_1({\ddd}y)\\
=&(1-t)\left[ \inf_{(x,y)\in A_0\times A_1}\beta^{(K,N)}_{1-t}(x,y)^{\frac1N}  \right] \nu[A_0]^{\frac1N}+t\left[ \inf_{(x,y)\in A_0\times A_1}\beta^{(K,N)}_{t}(x,y)^{\frac1N}  \right]\nu[A_1]^{\frac1N},\tag{5.12}\label{firstbmineq}
\end{align*}
where we used
$$\int_X\rho_i^{-\frac{1}{N}}(x)\mu_i({\ddd}x)=\int_X\rho_i^{1-\frac{1}{N}}(x)\nu({\ddd}x)=\nu(A_i)^\frac{1}{N},\ i=0,1.$$
On the other hand, owing to $\supp\mu_t\subset [A_0,A_1]_t$, we have
\[
\int_X\rho_t {\ddd}\nu =\int_{[A_0,A_1]_t}\rho_t {\ddd}\nu \leq 1,
\]
which together with  Jensen's inequality yields
\begin{align*}
\int_X \rho^{1-\frac1N}_t {\ddd}\nu=&\nu[[A_0,A_1]_t]\int_{[A_0,A_1]_t}\rho^{1-\frac1N}_t \frac{{\ddd}\nu}{\nu[[A_0,A_1]_t]}\leq \nu[[A_0,A_1]_t]\left(\int_{[A_0,A_1]_t}\rho_t \frac{{\ddd}\nu}{\nu[[A_0,A_1]_t]}  \right)^{1-\frac1N}\\
=&\nu[[A_0,A_1]_t]^{\frac1N}\left( \int_{[A_0,A_1]_t}\rho_t{\ddd}\nu \right)^{1-\frac1N}\leq \nu[[A_0,A_1]_t]^{\frac1N}.\tag{5.13}\label{estimateofmu-bm}
\end{align*}
Moreover, if  $N<\infty$ and $K\geq 0$, then
\[
\inf_{(x,y)\in A_0\times A_1}\beta^{(K,N)}_t(x,y)\geq \inf_{(x,y)\in A_0\times A_1}\beta^{(0,N)}_t(x,y)=1,\tag{5.14}\label{betaestimate-bm}
\]
as $\beta^{(K,N)}_t$ is nondecreasing in $K$.
Now (i) follows from (\ref{firstbmineq})--(\ref{betaestimate-bm}) immediately.

 \smallskip

 \noindent \textbf{(ii)} Let $\mu_i$, $i=0,1$ and $(\mu_t)_{0\leq t\leq 1}$ be as in (i). By considering $H(r):=r\log r$, a similar argument to (C-1) in the proof
 of Theorem \ref{supporsetcdkn} yields
  \begin{align*}
H_\nu(\mu_{t})\leq(1-t)\log\left( \frac{1}{\nu[A_0]} \right)+ t \log\left( \frac{1}{\nu[A_1]} \right)-K\frac{t(1-t)}{2}W_2(\mu_{0},\mu_{1})^2,\tag{5.15}\label{hinequlityCD}
\end{align*}
and particularly,
 $\mu_t$ is absolutely continuous with respect to $\nu$ for any $t\in [0,1]$, i.e., $\mu_t=\rho_t {\ddd}\nu$.

On the one hand, since $K=K_+-K_-$, we have
\[
-K\, W_2(\mu_0,\mu_1)^2\leq  K_-\max_{(x,y)\in A_0\times A_1}d(x,y)^2-K_+\min_{(x,y)\in A_0\times A_1}d(x,y)^2.\tag{5.16}\label{W2disestimate}
\]
On the other hand, one has
 \[
 1=\mu_t[X]=\int_X \rho_t {\ddd}\nu=\int_{[A_0,A_1]_t} \rho_t {\ddd}\nu,
 \]
 which together with Jensen's inequality yields
 \begin{align*}
 H_\nu[\mu_t]=&\nu\left[[A_0,A_1]_t\right] \int_{[A_0,A_1]_t} \rho_t \log \rho_t \frac{{\ddd}\nu}{\nu[[A_0,A_1]_t]}\\
 \geq& \nu\left[[A_0,A_1]_t\right] \left( \int_{[A_0,A_1]_t} \frac{\rho_t {\ddd}\nu}{\nu[[A_0,A_1]_t]}\right)\log\left( \int_{[A_0,A_1]_t} \frac{\rho_t {\ddd}\nu}{\nu[[A_0,A_1]_t]}\right)=\log\left( \frac{1}{\nu[[A_0,A_1]_t]} \right).
 \end{align*}
 Now this inequality together with (\ref{hinequlityCD}) and (\ref{W2disestimate}) furnishes (ii).
 \end{proof}

A similar argument to Villani \cite[Corollary 30.9]{Vi} yields the following result.

\begin{corollary}\label{nodiractnoatom}
Given  $K\in \mathbb{R}$ and $N\in[1,\infty]$, if  $(X,d,\nu)$ is a weak ${\CD}(K,N)$ space, then either $\nu$ is a Dirac mass, or $\nu$ has no atom (i.e., no single point contains a positive mass).
\end{corollary}

\begin{theorem}[Bishop-Gromov comparison theorem]\label{BishopGromovcomp}
Given $K\in \mathbb{R}$ and $N\in [1,\infty]$, let $(X,d,\nu)$ be a weak ${\CD}(K,N)$ space and let $x_0\in \supp \nu$. Thus
\begin{itemize}

\smallskip

\item[(i)] $\nu[B^+_{x_0}(r)]=\nu{\left[\overline{B^+_{x_0}(r)}\right]}$ for any $r>0;$

\item[(ii)] If $N<\infty$, then $f_{x_0}(r):=\frac{\nu[B_{x_0}^+(r)]}{\int^r_0\mathfrak{s}^{N-1}_{K,N}(t){\ddd}t}$ is  a nonincreasing function$;$

\smallskip

\item[(iii)]
If $N=\infty$, then for any fixed $\delta>0$, the following statements hold:
\begin{itemize}

\smallskip

\item[(a)] there exists a constant $C_1=C_1\left(K_-,\delta, \lambda_d(\overline{B^+_{x_0}(\delta)}), \nu{\left[{B^+_{x_0}(\delta)}\right]},\nu{\left[{B^+_{x_0}(2\delta)}\right]}\right)$  such that
\[
\nu{\left[ {B^+_{x_0}(r)} \right]}\leq e^{C_1 r}e^{K_-\frac{r^2}2},\ \forall\,r\geq \delta,\tag{5.17}\label{estkinftyvolume}
\]
where $\lambda_d(\cdot)$ is the reversibility  in Definition \ref{reversibilitydef}$;$

\smallskip

\item[(b)] provided $K\geq 0$, there exists a constant $C_2=C_2\left(K,\delta,\nu{\left[{B^+_{x_0}(\delta)}\right]},\nu{\left[{B^+_{x_0}(2\delta)}\right]}\right)$  such that
\[
  \nu{\left[ {B^+_{x_0}(r+\delta)}\backslash {B^+_{x_0}(r)} \right]}\leq e^{C_2r}e^{-K\frac{r^2}{2}},\ \forall\,r\geq \delta;\tag{5.18}\label{gromov-bishp1}
\]

\item[(c)] {for} any $K'<K$,
\[
\int_X e^{\frac{K'}{2}d(x_0,x)^2} d\nu(x)<\infty.\tag{5.19}\label{gromov-bishp2}
\]

\end{itemize}
\end{itemize}
\end{theorem}


\begin{proof}In view of  Theorem \ref{supporsetcdkn},
we may suppose $\supp\nu=X$. Since the proof is trivial if $\nu$ is a Dirac mass, due to Corollary \ref{nodiractnoatom} we may assume that $\nu$  has no atom.

\smallskip

\noindent \textbf{(i)}  According to Proposition \ref{CDKNstringent}/(i), it suffices to consider the case of  $N=\infty$.
Given $r>0$, for any small $\epsilon>0$, set $A_0:=\overline{B^+_{x_0}(\epsilon)}$ and $A_1:=\overline{B^+_{x_0}(r)}$. Thus,
$[A_0,A_1]_t\subset \overline{B^+_{x_0}\left(tr+(1-t)\epsilon\right)}=:U_t$, for
$t\in (0,1)$.
Note that $U_{t_1}\subset U_{t_2}$ if $t_1\leq t_2$ and $\cup_{t\in (0,1)} U_t=B^+_{x_0}(r)$, which yields
\[
\nu{\left[ B^+_{x_0}(r)  \right]}=\lim_{t\rightarrow 1^-}\nu{\left[ U_t \right]}\geq \lim_{t\rightarrow 1^-}\nu{\left[[A_0,A_1]_t\right]}.\tag{5.20}\label{ballesitmate}
\]
On the other hand, $\sup_{(x,y)\in A_0\times A_1}d(x,y)\leq (1+\lambda)r$, where $\lambda:=\lambda_d(\overline{B^+_{x_0}(r)})<\infty$. Since $\nu[A_0]>0$, Theorem \ref{BrunnMinkowineq}/(ii) implies
\begin{align*}
\log\nu{\left[[A_0,A_1]_t\right]}\geq (1-t)\log\nu[A_0]+t\log\nu[A_1]-\frac{K_-}{2}t(1-t)(1+\lambda)^2r^2,\ \forall\,t\in (0,1),
\end{align*}
which together with (\ref{ballesitmate}) furnishes $\nu{\left[ B^+_{x_0}(r) \right]}\geq\nu{\left[ \overline{B^+_{x_0}(r)} \right]}$.

\noindent \textbf{(ii)}
Set
\[
A_0:=\{x_0\}, \ \ A_1:=\overline{B^+_{x_0}(r+\epsilon)}\setminus B^+_{x_0}(r),\ \  \mu_0:=\delta_{x_0},\ \ \mu_1=\frac{\textbf{1}_{A_1}}{\nu[A_1]}\nu.
\]
For any $t\in (0,1)$, it is easy to check that $[A_0,A_1]_t\subset \overline{B^+_{x_0}(t(r+\epsilon))}\backslash B^+_{x_0}(tr)$ and
\[
\inf_{(x,y)\in A_0\times A_1}\beta^{K,N}_t(x,y)\geq \left\{
	\begin{array}{lll}
	& \left( \frac{\mathfrak{s}_{K,N}(t(r+\epsilon))}{t\,\mathfrak{s}_{K,N}(r+\epsilon)} \right)^{N-1}, & \ \ \ \text{if } K\geq 0, \\
	\\
	&\left( \frac{\mathfrak{s}_{K,N}(t(r-\epsilon))}{t\,\mathfrak{s}_{K,N}(r-\epsilon)} \right)^{N-1}, & \ \ \ \text{if } K<0.
	\end{array}
	\right.
\]
Suppose $K\geq 0$.
Theorem \ref{BrunnMinkowineq}/(i) furnishes
\begin{align*}
\nu\left[ \overline{B^+_{x_0}(t(r+\epsilon))}\backslash B^+_{x_0}(tr)  \right]^\frac{1}{N}\geq \nu\left[ [A_0,A_1]_t \right]^\frac{1}{N}\geq t\left( \frac{\mathfrak{s}_{K,N}(t(r+\epsilon))}{t\,\mathfrak{s}_{K,N}(r+\epsilon)} \right)^{\frac{N-1}N}\nu\left[ A_1 \right]^{\frac1N},
\end{align*}
which together with (i) implies
\begin{align*}
\frac{\phi(tr+t\epsilon)-\phi(tr)}{t\epsilon\,\mathfrak{s}^{N-1}_{K,N}(t(r+\epsilon))}\geq \frac{\phi(r+\epsilon)-\phi(r)}{\epsilon\,\mathfrak{s}^{N-1}_{K,N}(r+\epsilon)},
\end{align*}
where $\phi(r):=\nu[B^+_{x_0}(r)]$.
By letting   $\epsilon\rightarrow 0^+$, the function ${\left(\frac{d^+}{dr}\phi\right)(r)}/{\mathfrak{s}^{N-1}_{K,N}(r)}$ is nonincreasing in $r$, which
together with the monotonic L'Hospital rule yields (i). The case of $K<0$ follows by a similar argument.

\smallskip

\noindent\textbf{(iii)} (a)
Set
\[
A_0:=\overline{B^+_{x_0}\left( \delta \right)},\ A_1:=\overline{B_{x_0}^+(r)},\ t:=\frac{\delta}{2r}\leq \frac{1}{2}.
\]
It is not hard to check that $[A_0,A_1]_t\subset \overline{B^+_{x_0}(2\delta)}$ and $\sup_{(x,y)\in A_0\times A_1}d(x,y)\leq r+\lambda \delta$, where $\lambda:=\lambda_d(\overline{B^+_{x_0}(\delta)})$. Thus,
Theorem \ref{BrunnMinkowineq}/(ii) yields
\begin{eqnarray*}
\log\left( \frac{1}{\nu\left[\overline{B^+_{x_0}(2\delta)}\right]} \right)&\leq& \log\left( \frac{1}{\nu[[A_0,A_1]_t]} \right)\\
&\leq& \left(1- \frac{\delta}{2r} \right)\log\left( \frac{1}{\nu[A_0]} \right)+\frac{\delta}{2r}\log\left( \frac{1}{\nu[A_1]} \right)+ \frac{K_-}{2}\frac{\delta}{2r}\left(1- \frac{\delta}{2r} \right)\left(r+{\lambda}\delta\right)^2,
\end{eqnarray*}
which together with (i) implies
\begin{align*}
\log\nu{\left[B^+_{x_0}(r)\right]}=\log \nu{\left[A_1\right]}\leq a +b r+\frac{K_-}2 r^2,
\end{align*}
where $a$ and $b$ are the constants only dependent of $K_-$, $\delta$, $\lambda$, $\nu{\left[{B^+_{x_0}(\delta)}\right]}$ and $\nu{\left[{B^+_{x_0}(2\delta)}\right]}$.
Thus,
 (\ref{estkinftyvolume}) follows immediately.

  \smallskip

  (b) Let
\[
A_0:=\overline{B^+_{x_0}\left(  {\delta}  \right)},\ A_1:=\overline{B_{x_0}^+(r+\delta)}\backslash {B_{x_0}^+(r)},\ t:=\frac{\delta}{2r}\leq \frac{1}{2}.
\]
  Since $K=K_+$ and $\inf_{(x,y)\in A_0\times A_1}d(x,y)\geq r-\delta$, a similar argument to that of (a)   yields
 \[
 \log \nu{\left[A_1\right]}\leq \frac{a}{r}+b+br-\frac{K}2r^2\leq \frac{a}{\delta}+b+br-\frac{K}2r^2,
 \]
 where   $a$ and $b$ are the constants only dependent of $K$, $\delta$, $\nu{\left[ {B^+_{x_0}(\delta)}\right]}$ and $\nu{\left[ {B^+_{x_0}(2\delta)}\right]}$. Thus,
 (\ref{gromov-bishp1}) follows.

\smallskip

{(c) If $K\leq 0$, we have
$K'+K_-=K'-K<0$ and hence,
 (\ref{estkinftyvolume}) ($\delta=1$, $r=i\in \mathbb{N}$)  furnishes
\begin{align*}
&\int_X e^{\frac{K'}2d(x_0,x)^2}{\ddd}\nu(x)\leq e^{\frac{K'}2 0^2}\nu{\left[ {B^+_{x_0}(1)} \right]}+\sum_{i=2}^\infty e^{\frac{K'}{2}(i-1)^2}\nu{\left[ {B^+_{x_0}(i)}\backslash {B^+_{x_0}(i-1)} \right]}\\
\leq&  \nu{\left[  {B^+_{x_0}(1)} \right]}+\sum_{i=2}^\infty e^{\frac{K'}{2}(i-1)^2}e^{Ci}e^{\frac{K_-}2 i^2}=\nu{\left[ {B^+_{x_0}(1)} \right]}+\sum_{i=2}^\infty e^{\frac{(K'-K)}2i^2+(C-K')i+\frac{K'}2}<\infty.
\end{align*}
If $K>0$, the inequality immediately follows by a similar argument and \eqref{gromov-bishp1}, respectively.}
\end{proof}

In view of Definition \ref{doulbidefin}, Theorem \ref{BishopGromovcomp}/(ii) furnishes the following result in a direct way.
\begin{corollary}\label{doublingCDKN}
If $(X,d,\nu)$ is a weak ${\CD}(K,N)$ space with $K\in \mathbb{R},N<\infty$ and $X=\supp\nu$, then $(X,d,\nu)$ is  almost doubling  with a constant $L=L(r)=L(K,N,r)$ depending only on $K,N,r$. In particular, if $\diam(X)\leq D$ then $(X,d,\nu)$ is globally doubling with a constant $L=L(K,N,D)$.
\end{corollary}

It should be noticed that Theorems \ref{diametercontrol}--\ref{BishopGromovcomp} together with Theorem \ref{OTARic} furnish Myers' theorem, the Brunn-Minkowski inequality and the Bishop-Gromov theorem in the Finsler setting obtained by Ohta \cite{O,O1}. Moreover, using similar arguments as in Bacher \cite{Ba}, Lott and Villani \cite{LV}, Sturm \cite{Sturm-1,Sturm-2} and Villani \cite{Vi}, one can extend the Borell-Brascamp-Lieb inequality to the irreversible setting; see Ohta \cite{O1,O2,Ot} for the Finsler versions of these inequalities.

\subsection{Stability of weak ${\rm{CD}}(K,N)$ spaces}\label{staofRicc}
For convenience of the presentation,   ``$\nu_k\rightharpoonup\nu$" is used  to denote that a sequence of measures $(\nu_k)_k$ converges weakly to $\nu$ while ``$\mu\ll\nu$" is used to denote  that $\mu$ is absolutely continuous with respect to $\nu$.

\begin{theorem}\label{CDknstabilitycompact}
Given $K\in \mathbb{R}$, $N\in [1,\infty]$ and $\theta\in [1,\infty)$, the  weak curvature-dimension bound ${\CD}(K,N)$ is stable in the measured $\theta$-Gromov-Hausdorff topology. That is,
suppose that $(X_k,d_k,\nu_k)_k$ is a sequence of compact $\theta$-geodesic-measure spaces satisfying  the weak curvature-dimension bound ${\CD}(K,N)$. If  $(X_k,d_k,\nu_k)_k$ converges  in the measured $\theta$-Gromov-Hausdorff topology, then
the limit space, say $(X,d,\nu)$, also  satisfies the weak curvature-dimension bound ${\CD}(K,N)$.
\end{theorem}

\begin{proof} It follows by  Definition \ref{measureconvergence}, Lemma \ref{inprotantisometry} and Theorem \ref{lenghtclsoedghtopo} that $(X,d,\nu)$ is a compact $\theta$-geodesic-measure space.
Due to Theorem \ref{supporsetcdkn}, we may assume $\supp \nu=X$.
Then Definition \ref{measureconvergence} yields a sequence of measurable functions $f_k:X_k\rightarrow X$ such that
\begin{itemize}
\item  $f_k:(X_k,d_k)\rightarrow (X,d)$ is an $\epsilon_k$-isometry, with $\epsilon_k\rightarrow 0$;
\item  $((f_k)_\sharp \nu_k)_k$ converges weakly to $\nu$.
\end{itemize}

Now we check that $(X,d,\nu)$ is a weak ${\CD}(K,N)$ space. According to Proposition \ref{CDKNstringent}/(iii),    it suffices to show that  (\ref{weakcdkncon}) holds for any two probability measures $\mu_i$, $i=0,1$ on $X$ satisfying $\mu_i=\rho_i\nu$ and $\rho_i\in C(X)$.

Owing to Lemma \ref{continuonUvUpi}/(iii), for each $i=0,1$, there is a sequence of probability measures $({\mu}_{\alpha,i})_\alpha$ such that ${\mu}_{\alpha,i}=\rho_{\alpha,i}\nu$ with continuous densities $\rho_{\alpha,i}$ and  ${\mu}_{\alpha,i}\rightharpoonup  \mu_i$ (as $\alpha\rightarrow \infty$). Moreover,  Remark \ref{suppregularing} combined with Lemma \ref{Regukern} ($\mathcal {K}=X$) implies
\[
0\leq \rho_{\alpha,i}(x)= \int_X \mathscr{K}_{1/\alpha}(x,y)\rho_{i}(y){\ddd}\nu(y)\leq \max_X\rho_i<\infty, \text{ for }i=0,1.\tag{5.21}\label{densityestimiate}
\]
For each $k\in \mathbb{N}$, define two probability measures on $X_k$ by
\[
\mu_{k,\alpha,i}:=\frac{(\rho_{\alpha,i}\circ f_k)\nu_k}{Z_{k,\alpha,i}},\ Z_{k,\alpha,i}=\int_{X_k} (\rho_{\alpha,i}\circ f_k){\ddd}\nu_k, \text{ for }i=0,1.
\]
Then  (\ref{densityestimiate}) together with
$(f_k)_\sharp \nu_k\rightharpoonup \nu$  implies
\[
Z_{k,\alpha,i}\rightarrow 1,\ (f_k)_\sharp \mu_{k,\alpha,i}\rightharpoonup \mu_{\alpha,i},\ \text{ as $k\rightarrow \infty$}.\tag{5.22}\label{connectmu01}
\]

For each $k$ and each $\alpha$, since $(X_k,d_k,\nu_k)$ is a  weak ${\CD}(K,N)$ space and  $\mu_{k,\alpha,i}\ll \nu_k$ for $i=0,1$, there is a Wasserstein geodesic $(\mu_{k,\alpha,t})_{0\leq t\leq 1}$ jointing $\mu_{k,\alpha,0}$ to $\mu_{k,\alpha,1}$ associated with an optimal transference plan $\pi_{k,\alpha}$ of coupling $(\mu_{k,\alpha,0},\mu_{k,\alpha,1})$  such that for all $U\in DC_N$ and $t\in [0,1]$,
\[
U_{\nu_k}(\mu_{k,\alpha,t})\leq (1-t) U^{\beta^{(K,N)}_{1-t}}_{\pi_{k,\alpha},\nu_k}(\mu_{k,\alpha,0})+tU^{\breve{\beta}^{(K,N)}_{t}}_{\breve{\pi}_{k,\alpha},\nu_k}(\mu_{k,\alpha,1}).\tag{5.23}\label{goingtolimiXkcase}
\]
Let $\Pi_{k,\alpha}$ denote the associated dynamical optimal transference plan of $\pi_{k,\alpha}$. Thus, $\mu_{k,\alpha,t}=(e_t)_\sharp\Pi_{k,\alpha}$ and $\pi_{k,\alpha}=(e_0,e_1)_\sharp\Pi_{k,\alpha}$. For each $\alpha$,
by Theorem \ref{stabilityoptimal} and a Cantor's diagonal argument, up to extraction of a subsequence in $k$, there is a dynamical optimal transference plan $\Pi_\alpha$ on $\Gamma(X)$ such that
\begin{itemize}

\item[(i)]   $\lim_{k\rightarrow \infty}(f_k,f_k)_{\sharp}\pi_{k,\alpha}=(e_0,e_1)_\sharp\Pi_\alpha=:\pi_\alpha$ in the weak topology on $P(X\times X)$;

\smallskip

\item[(ii)] $\lim_{k\rightarrow \infty}(f_k)_\sharp \mu_{k,\alpha,t}=(e_t)_\sharp\Pi_\alpha=:\mu_{\alpha,t}$ in $P_2(X)$ uniformly in $t$. More precisely,
\[
\lim_{k\rightarrow \infty}\sup_{t\in [0,1]}W_2(\mu_{\alpha,t},(f_k)_\sharp\mu_{k,\alpha,t})=0.
\]

\item[(iii)]   $\pi_\alpha$ is an optimal transference plan and $(\mu_{\alpha,t})_{0\leq t\leq 1}$ is a displacement interpolation which are associated with the dynamical optimal transference plan   $\Pi_\alpha$.
\end{itemize}

It follows from Corollary \ref{lengthpropWASSER} that $(P(X),W_2)$ is a compact $\theta$-geodesic space.
In view of Theorem \ref{lengspacewass}, by passing to a subsequence, one can find a dynamical
optimal transference plan $\Pi$ such that $\Pi_\alpha\rightharpoonup \Pi$ and hence,
\[
\pi_\alpha \rightharpoonup \pi:=(e_0,e_1)_\sharp \Pi,\quad \mu_{\alpha,t}\rightharpoonup \mu_t:=(e_t)_\sharp \Pi, \text{ for each $t\in [0,1]$}.
\]
In particular, $\pi$ is an optimal transference plan from $\mu_0$ to $\mu_1$ and $\mu_t$ is the associated displacement interpolation.
It remains to pass to the limit ($k,\alpha\rightarrow \infty$) in (\ref{goingtolimiXkcase}) and show
\[
U_{\nu}(\mu_{t})\leq (1-t) U^{\beta^{(K,N)}_{1-t}}_{\pi,\nu}(\mu_{0})+tU^{\breve{\beta}^{(K,N)}_{t}}_{\breve{\pi},\nu}(\mu_{1}).\tag{5.24}\label{Xcdkncondition}
\]

According to Proposition \ref{CDKNstringent}/(ii), we only need to check that (\ref{Xcdkncondition}) holds for nonnegative $U\in DC_N$ with (\ref{mostpoly}).
We first assume that $\beta^{(K,N)}_t$ is continuous and bounded.
Thus, Lemma \ref{continuonUvUpi}/(i)(ii)  yield
\begin{align*}
U_\nu(\mu_{t})\leq \underset{\alpha\rightarrow\infty}\liminf \,U_\nu(\mu_{\alpha,t})\leq \underset{\alpha\rightarrow\infty}\liminf\left[\underset{k\rightarrow \infty}{\lim\inf}\, U_{(f_k)_\sharp\nu_k}((f_k)_\sharp\mu_{k,\alpha,t})\right]\leq \underset{\alpha\rightarrow\infty}\liminf\left[\underset{k\rightarrow \infty}{\lim\inf}\,U_{ \nu_k}(\mu_{k,\alpha,t})\right].\tag{5.25}\label{xcdfisrtpart}
\end{align*}
On the other hand, set $\beta(x,y):=\beta^{(K,N)}_{1-t}(x,y)$. Note that
\[
\lim_{k\rightarrow \infty}\sup_{x,y\in X_k} \left| d_k(x,y)-d(f_k(x),f_k(y)) \right|=0.
\]
Thus, the uniform boundedness of $\diam(X_k)$ implies the uniform continuity of $\beta$, that is, for any $\varepsilon>0$, there exists $N=N(\varepsilon)>0$ such that if $k>N$, then
\[
\sup_{x,y\in X_k}\left|\beta(f_k(x),f_k(y))-\beta(x,y)\right|<\varepsilon,
\]
which together with the continuity and boundness of $U$ implies
\[
\beta(f_k(x),f_k(y))\,U\left(\frac{\varrho_{k,\alpha,0}(x)}{\beta(f_k(x),f_k(y))}\right)\rightrightarrows\beta(x,y)\,U\left(\frac{\varrho_{k,\alpha,0}(x)}{\beta(x,y)}\right),\text{ for }i=0,1,
\]
where $\varrho_{k,\alpha,0}(x)\nu_k(x):=\frac{\rho_{\alpha,0}\circ f_k(x) }{Z_{k,\alpha,0}} \nu_k(x)=\mu_{k,\alpha,0}(x)$.
Thus, by setting $\pi_{k,\alpha}({\ddd}x{\ddd}y)=\mu_{k,\alpha,0}({\ddd}x)\pi_{k,\alpha}({\ddd}y|x)$, one has
\begin{align*}
\lim_{k\rightarrow\infty} \int_{X_k\times X_k}\left|\beta(f_k(x),f_k(y))\,U\left(\frac{\varrho_{k,\alpha,0}(x)}{\beta(f_k(x),f_k(y))}\right) -\beta(x,y)\,U\left(\frac{\varrho_{k,\alpha,0}(x)}{\beta(x,y)}\right)\right|\pi_{k,\alpha}({\ddd}y|x)\nu_{k}({\ddd}x)=0.\tag{5.26}\label{Uuniformconverge}
\end{align*}
Let $v(r):=U(r)/r$.
By Remark \ref{firstUv}, one has
\begin{align*}
&\int_{X_k\times X_k}\beta(f_k(x),f_k(y))U\left(\frac{\varrho_{k,\alpha,0}(x)}{\beta(f_k(x),f_k(y))}\right)\pi_{k,\alpha}({\ddd}y|x)\nu_k({\ddd}x)\\
=&\int_{X_k\times X_k} v\left(\frac{\rho_{\alpha,0}\circ f_k(x)}{Z_{k,\alpha,0}\,\beta(f_k(x),f_k(y))}\right)\pi_{k,\alpha}({\ddd}x{\ddd}y)=\int_{X\times X}v\left(\frac{\rho_{\alpha,0}(y_0)}{Z_{k,\alpha,0}\,\beta(y_0,y_1)}\right){\ddd}\left[(f_k,f_k)_\sharp \pi_{k,\alpha}\right](y_0,y_1).\tag{5.27}\label{vfkfkpi}
\end{align*}
Recall that $Z_{k,a,0}\rightarrow 1$ and $(f_k,f_k)_\sharp \pi_{k,\alpha}\rightharpoonup \pi_\alpha$ (as $k\rightarrow \infty$), which together with (\ref{Uuniformconverge}) and (\ref{vfkfkpi})  furnishes
\begin{align*}
&\underset{k\rightarrow\infty}{\lim}U^{\beta^{(K,N)}_{1-t}}_{\pi_{k,\alpha},\nu_k}(\mu_{k,\alpha,0}) =\underset{k\rightarrow\infty}{\lim} \int_{X_k\times X_k}\beta(x,y)\,U{\left(\frac{\varrho_{k,\alpha,0}(x)}{\beta(x,y)}\right)}\pi_{k,\alpha}({\ddd}y|x)\nu_k(dx)\\
=&\underset{k\rightarrow\infty}{\lim}\int_{X\times X}v\left(\frac{\rho_{\alpha,0}(y_0)}{Z_{k,\alpha,0}\,\beta(y_0,y_1)}\right){\ddd}\left[(f_k,f_k)_\sharp \pi_{k,\alpha}\right](y_0,y_1)=\int_{X\times X}v\left( \frac{\rho_{\alpha,0}(y_0)}{\beta(y_0,y_1)} \right)\pi_\alpha({\ddd}y_0{\ddd}y_1)\\
=&U^{\beta^{(K,N)}_{1-t}}_{\pi_\alpha,\nu}(\mu_{\alpha,0}).
\end{align*}
Moreover, this inequality combined with Lemma \ref{continuonUvUpi}/(iii) yields
\[
\underset{\alpha\rightarrow\infty}{\lim\sup}\underset{k\rightarrow\infty}{\lim}U^{\beta^{(K,N)}_{1-t}}_{\pi_{k,\alpha},\nu_k}(\mu_{k,\alpha,0})=\underset{\alpha\rightarrow\infty}{\lim\sup}\,U^{\beta^{(K,N)}_{1-t}}_{\pi_\alpha,\nu}(\mu_{\alpha,0})\leq U^{\beta^{(K,N)}_{1-t}}_{\pi,\nu}(\mu_{0}).  \tag{5.28}\label{suplim1}
\]
Similarly, we have
\[
\underset{\alpha\rightarrow\infty}{\lim\sup}\underset{k\rightarrow\infty}{\lim} U^{\breve{\beta}^{(K,N)}_{t}}_{\breve{\pi}_{k,\alpha},\nu_k}(\mu_{k,\alpha,1})\leq U^{\breve{\beta}^{(K,N)}_{t}}_{\breve{\pi},\nu}(\mu_1),
\]
which together with  (\ref{goingtolimiXkcase}), (\ref{xcdfisrtpart}) and (\ref{suplim1}) furnishes (\ref{Xcdkncondition}).

Note that $\beta^{(K,N)}_t$ is continuous and bounded if one of the  following conditions holds: (a) $K\leq 0$ and $N>1$;
(b) $K>0$ and $N=\infty$;
(c) $K>0$, $1<N<\infty$ and $\sup_k \diam(X_k)<D_{N,K}=\pi\sqrt{(N-1)/K}$.

Now we consider the case when $K\leq 0$ and $N=1$.  Proposition \ref{CDKNstringent}/(i) together with the above argument yields that for any fixed $N'>1$, there is a Wasserstein geodesic $(\mu_t)_{0\leq t\leq 1}$ and an associated optimal transference plan $\pi$ of $(\mu_0,\mu_1)$ such that for any $U\in DC_N$ and $t\in [0,1]$,
\[
U_\nu(\mu_t)\leq (1-t) U^{\beta^{(K,N')}_{1-t}}_{\pi,\nu}(\mu_0)+t\,U^{\breve{\beta}^{(K,N')}_{t}}_{\breve{\pi},\nu}(\mu_1).
\]
In view of  Remark \ref{limitdefinitionofbeta}, by letting $N'\downarrow 1$, we get (\ref{Xcdkncondition}) again. If $K>0$, $1<N<\infty$ and $\sup \diam(X_k)=D_{K,N}$, we can apply a similar argument, introducing again the bounded coefficient $\beta^{(K,N')}_t$ for $N'>N$ and then passing to the limit as $N'\downarrow N$.
\end{proof}


\begin{theorem}\label{stablecdknnoncompact}Given $K\in \mathbb{R}$ and $N\in [1,\infty]$, for any nondecreasing function $\Theta(r)\geq 1$, the  weak curvature-dimension condition ${\CD}(K,N)$ is stable in the pointed measured forward $\Theta$-Gromov-Hausdorff topology. That is,
let $(X_k,\star_k,d_k,\nu_k)_k$ be a sequence of forward boundedly compact, $\sigma$-finite, pointed  forward $\Theta$-geodesic-measure spaces satisfying  the weak curvature-dimension condition ${\CD}(K,N)$. If $(X_k,\star_k,d_k,\nu_k)_k$
 converges in the pointed measured forward $\Theta$-Gromov-Hausdorff topology, then the limit space, say $(X,\star,d,\nu)$, satisfies  the weak curvature-dimension condition ${\CD}(K,N)$ as well.
 \end{theorem}
\begin{proof}
It follows by Definition \ref{noncompactconvergedefi}, Theorem \ref{newnoncp}/(ii) and Proposition \ref{lenthcompactnoncompact}  that  $(X,\star,d,\nu)$ is a forward boundedly compact, $\sigma$-finite,   pointed forward $\Theta$-geodesic-measure space.
We can assume  $\supp\nu=X$.
Let $\mu_i=\rho_i\nu$, $i=0,1$ be two compactly supported measure on $X$ with continuous densities. Choose $R>0$ such that  $\suppor\rho_i\subset \overline{B^+_\star(R)}$ for $i=0,1$. Thus, for each $i$, Lemma \ref{continuonUvUpi}/(iii) together with Remark \ref{suppregularing} furnishes  a sequence of probability measures ${\mu}_{\alpha,i}=\rho_{\alpha,i}\nu$, $\alpha\in \mathbb{N}$ with continuous densities  such that
\[
\suppor\rho_{\alpha,i}\subset \overline{B^+_{\star}\left(R+\frac12\right)} \text{ for each }\alpha\ \  {\rm and }\ \ \ {\mu}_{\alpha,i}\rightharpoonup \mu_{i} \text{ in }P\left(\overline{B^+_{\star}\left(R+\frac12\right)}\right).
\]
Obviously, these properties imply that  ${\mu}_{\alpha,i}\rightharpoonup \mu_{i}$ in $P(\overline{B^+_{\star}\left(2(R+1)\,\Theta(R+1)\right)})$.

On the other hand, for each $k\in \mathbb{N}$, one can choose a pointed $\epsilon_k$-isometry $f_k: \overline{B^+_{\star_k}\left(R_k\right)}\rightarrow \overline{B^+_{\star}\left(R_k\right)}$ such that $\epsilon_k\rightarrow 0$, $R_k\rightarrow \infty$ and $(f_k)_{\sharp}\nu_k$ converges to $\nu$ in the weak-$*$ topology.
As in the proof of Theorem  \ref{CDknstabilitycompact}, set
 \[
\mu_{k,\alpha,i}:=\frac{(\rho_{\alpha,i}\circ f_k)\nu_k}{Z_{k,\alpha,i}},\ Z_{k,\alpha,i}=\int_{X_k} (\rho_{\alpha,i}\circ f_k){\ddd}\nu_k, \text{ for }i=0,1.
\]
Since $f_k$ is a pointed $\epsilon_k$-isometry, there exists $K_1\geq 0$ such that
\[
\supp\mu_{k,\alpha,i}\subset \overline{B^+_{\star_k}\left(R+1\right)}, \text{ for any }k>K_1 \text{ and for }i=0,1.
\]
Let $(\mu_{k,\alpha,t})_{0\leq t\leq 1}$ be a Wasserstein geodesic
associated with an optimal transference plan $\pi_{k,\alpha}$ of coupling $(\mu_{k,\alpha,0},\mu_{k,\alpha,1})$ such that for
all $U\in DC_N$ and $t\in [0,1]$,
\[
U_{\nu_k}(\mu_{k,\alpha,t})\leq (1-t) U^{\beta^{(K,N)}_{1-t}}_{\pi_{k,\alpha},\nu_k}(\mu_{k,\alpha,0})+t\,U^{\breve{\beta}^{(K,N)}_{t}}_{\breve{\pi}_{k,\alpha},\nu_k}(\mu_{k,\alpha,1}).
\]
Note that every minimal geodesic  from $\supp \mu_{k,\alpha,0}$ to $\supp \mu_{k,\alpha,1}$ in $X_k$ is contained in the ball $\overline{B^+_{\star_k}\left(2(R+1)\,\Theta(R+1)\right)}$. Hence, $\supp\mu_{k,\alpha,t}\subset\overline{B^+_{\star_k}\left(2(R+1)\,\Theta(R+1)\right)}$ for each $t\in [0,1]$. Based on this observation,
  the  argument in the proof of Theorem \ref{CDknstabilitycompact} can be applied to the sequence of compact metric-measure spaces
\[
\left(\overline{B^+_{\star_k}\left(2(R+1)\,\Theta(R+1)\right)}, d_k|_{\overline{B^+_{\star_k}\left(2(R+1)\,\Theta(R+1)\right)}},\nu_k|_{\overline{B^+_{\star_k}\left(2(R+1)\,\Theta(R+1)\right)}}\right)_k
\]
 which converges  to   $\left(\overline{B^+_{\star}\left(2(R+1)\,\Theta(R+1)\right)},d|_{\overline{B^+_{\star}\left(2(R+1)\,\Theta(R+1)\right)}},\nu|_{\overline{B^+_{\star}\left(2(R+1)\,\Theta(R+1)\right)}}\right)$ in the measured $\Theta(2(R+1)\,\Theta(R+1))$-Gromov-Hausdorff topology.
\end{proof}

\begin{theorem}\label{compactnessdoubletheCD}
{\rm (i)} Given $\theta,N\in[1,\infty)$, $K\in \mathbb{R}$, $D\in(0,\infty)$ and $0<m\leq M<\infty$, let $\CDD(\theta,K,N,D,m,M)$ be the collection of all compact $\theta$-geodesic-measure spaces $(X,d,\nu)$ satisfying the weak curvature-dimension bound ${\CD}(K,N)$, together with $\diam(X)\leq D$, $m\leq \nu[X]\leq M$ and $\supp \nu=X$. Then $\CDD(\theta,K,N,D,m,M)$ is compact in the measured $\theta$-Gromov-Hausdorff topology.

{\rm (ii)}   Given $K\in \mathbb{R}$, $N\in[1,\infty)$ and  $0<m\leq M<\infty$, for any nondecreasing function $\Theta(r)\geq 1$, let $\pCDD(\Theta,K,N,m,M)$ be the collection of all  forward boundedly compact,  $\sigma$-finite, pointed forward $\Theta$-geodesic-measure spaces satisfying the weak  curvature-dimension bound ${\CD}(K,N)$, together with $m\leq \nu[B^+_\star(1)]\leq M$ and $\supp \nu=X$. Then  $\pCDD(\Theta,K,N,m,M)$ is compact in the measured forward $\Theta$-Gromov-Hausdorff topology.
\end{theorem}

\begin{proof} (i) follows by Corollary \ref{doublingCDKN}, Theorem  \ref{compactmeasuredgromvhaus}, Theorem \ref{lenghtclsoedghtopo} and Theorem \ref{CDknstabilitycompact}
while (ii) follows by Corollary \ref{doublingCDKN}, Theorem \ref{noncomapctdoubling}, Proposition \ref{lenthcompactnoncompact} and Theorem \ref{stablecdknnoncompact}.
\end{proof}

\begin{example}
It may happen that a sequence of weak $\CD(K,\infty)$ spaces $(X_k,d_k,\nu_k)_k$ with $\supp \nu_k=X_k$ converges in the measured Gromov-Hausdorff sense to a forward geodesic-measure space $(X,d,\nu)$ such that $\supp\nu$ is strictly smaller than $X$. For example, consider a sequence of Finsler metric-measure manifolds
\[
 (X_k,F_k,\nu_k)=\left(  \mathbb{R}^n, \|\cdot\|+\frac1{2^{k}} \langle \cdot,{\ddd}x^1\rangle, \frac{\exp\left( -k\|x\|^2 \right){\ddd}x}{Z_k}         \right),\ k\in \mathbb{N},
\]
where $Z_k$ is a normalizing constant.

According to Bao, Chern and Shen \cite{BCS}, each $F_k$ is a Randers-Berwald metric, thus geodesics of $F_k$ are also geodesics of $\|\cdot\|$ and vice-versa. Hence, if $t\mapsto \gamma_y(t)$, $t\in [0,\infty)$, is a constant-speed geodesic of $F_k$ with $\dot{\gamma}_y(0)=y=(y^1,\ldots,y^n)$, we have
$\gamma_y(t)=\gamma_y(0)+ty$,
which implies
\[
\det g_{ij}(\dot{\gamma}_y(t))=\left( \frac{\|y\|+\frac{1}{2^k} y^1}{\|y\|} \right)^{n+1}=\text{const}.
\]
On the other hand,  since
\[
{\ddd}\nu_k=\frac{\exp\left( -k\|x\|^2 \right)}{Z_k} {\ddd}x^1\cdots {\ddd}x^n=:\sigma(x) {\ddd}x^1\cdots {\ddd}x^n,
\]
the distortion is
\[
\tau(\dot{\gamma}_y(t))=\log \frac{\sqrt{\det g_{ij}(\dot{\gamma}_y(t))}}{\sigma({\gamma}_y(t))}=\text{const.}+k\langle \gamma_y(0)+ty,\gamma_y(0)+ty  \rangle.
\]
Accordingly, we obtain
\begin{align*}
\mathbf{S}(\dot{\gamma}(t))=\frac{d}{dt}\tau(\dot{\gamma}(t))=2k\langle y,\gamma_y(0)+ty\rangle,\quad
\frac{d}{dt}\mathbf{S}(\dot{\gamma}(t))=2k\|y\|^2>0.
\end{align*}
Note that $(X_k,F_k)$ is a Minkowski space and hence, $\mathbf{Ric}\equiv0$. Owing to (\ref{defRicN}), we have $\mathbf{Ric}_\infty>0$. Let  $\mathcal {X}_k:=(X_k,\mathbf{0},d_{k},\nu_k)$ be the associated $3$-pointed forward geodesic-measure space (see (\ref{reversibdindentiy}) and (\ref{Randersnormreveruniform})).
It follows from Theorem \ref{OTARic} that  $\mathcal {X}_k$ is a weak $\CD(0,\infty)$ space with $\supp\nu_k=X_k$.

Let $\mathcal {X}_\infty:=(\mathbb{R}^n,\mathbf{0},d_{\mathbb{R}^n},\delta_{\mathbf{0}})$. Now consider the identity $f_k:=\id:\mathcal {X}_k\rightarrow \mathcal {X}_\infty$.
Let $t\mapsto \gamma(t)$, $t\in [0,1],$ be an arbitrary straight line contained in the Euclidean ball $\overline{\mathbb{B}_{\mathbf{0}}(R)}$, which is also a minimal geodesic of $\mathcal {X}_k$. Obviously,
\[
|d_k(\gamma(0),\gamma(1))-d_\infty(f_k(\gamma(0)),f_k(\gamma(1)))|\leq \frac{R}{2^{k-1}}.
\]
By choosing $R=2^{\frac{k-1}2}$ and slightly modifying the map $f_k$, it is not hard to see that $\mathcal {X}_k$ converges to $\mathcal {X}_\infty$ in the pointed
measured forward $3$-Gromov-Hausdorff topology (see Definition \ref{noncompactconvergedefi}). In particular, the limit measure is supported in the point $\mathbf{0}$.
\end{example}

Similarly to the Riemannian case (e.g. \cite{Lo1,LV,Vi}), the weighted Ricci curvature introduced by Ohta and Sturm\cite{Ot} is also stable in the Finsler setting.

\begin{theorem}
Given $K\in \mathbb{R}$, $N\in [2,\infty]$ and $\theta\in[1,\infty)$ $($resp., a nondecreasing function $\Theta(r)\geq 1$$),$ if a sequence of forward complete  Finsler metric-measure manifolds satisfying $\mathbf{Ric}_N\geq K$ converges to a forward complete Finsler metric-measure manifold in the measured  $\theta$-Gromov-Hausdorff topology $($resp.,
the pointed measured forward $\Theta$-Gromov-Hausdorff topology$),$ then the limit also satisfies $\mathbf{Ric}_N\geq K$.
\end{theorem}
\begin{proof}
The result directly follows by Theorems \ref{CDknstabilitycompact}, \ref{stablecdknnoncompact} and \ref{OTARic}.
\end{proof}

\subsection{Functional aspects of weak ${\rm{CD}}(K,N)$ spaces}
\label{section5-4} In this section we shall present some analytic properties of weak $\CD(K,N)$ irreversible spaces by  extending their reversible counterparts (cf. Lott-Villani \cite{LV}, Strum \cite{Sturm-1,Sturm-2} and Villani \cite{Vi}) and from the Finsler setting (cf. Ohta\cite{O,O1,O2}).
We also refer to  Ambrosio, Gigli and Savar\'e \cite{AGS,AGS2,AGS3}, Cavalletti and Mondino \cite{CM},  Gigli, Mondino and Savar\'e \cite{GMS} for recent developments of the geometric analysis on reversible metric-measure spaces.

In the sequel,  we use $P^{ac}_c(X,\nu)$ to denote the set of compactly supported probability measures on $X$ which are absolutely continuous with respect to $\nu$.

\begin{definition}\label{Lipschitzcontinuous} Let $(X,d)$ be a forward metric space. A function $f:X\rightarrow \mathbb{R}$ is said to be {\it Lipschitz continuous} if there exists a constant $C>0$ such that
\[
f(y)-f(x)\leq C\, d(x,y),\ \forall\,x,y\in X,
\]
in which case  $f$ is  called a {\it $C$-Lipschitz function}. A function  {$f:X\rightarrow \mathbb{R}$} is said to be {\it locally Lipschitz continuous} if for every $x\in X$ there exists a neighborhood $U$ of $x$ such that $f|_U$ is Lipschitz continuous.
\end{definition}

\begin{remark}\label{Lisppropt} The property of a Lipschitz function in the irreversible case is similar to the usual one. Indeed,
let $(X,\star,d)$ be a pointed forward $\Theta$-metric space and let $f$ be a $C$-Lipschitz function; thus,
\[
|f(x)-f(y)|\leq C \, \Theta\left(d(\star,x)+d(x,y)\right)\,d(x,y),\ \forall\,x,y\in X.
\]
\end{remark}

\begin{definition}
Let $(X,d)$ be a forward metric space and let $f:X\rightarrow \mathbb{R}$ be a continuous function.  If $x$ is not an isolated point of $X$, the {\it norm of gradient} $|\nabla f|(x)$ at $x\in X$ is defined as
\[
|\nabla f|(x):=\limsup_{y\rightarrow x}\frac{|f(y)-f(x)|}{d(x,y)}.
\]

\end{definition}
\begin{remark}\label{gradremarks}Due to Remark \ref{Lisppropt}, the norm  of gradient of a locally Lipschitz function is always finite. Furthermore,
if a forward metric space $(M,d_F)$ is induced by a forward complete Finsler manifold $(M,F)$, then for every $x\in M$ one has $|\nabla f|(x)=\max\{F^*(\pm df|_{x})\}$; this notion is more convenient in the present setting than $F^*(df)$, see e.g.\ the study of Heisenberg-Pauli-Weyl uncertainty principles and Hardy inequalities in generic Finsler structures, cf. Huang, Krist\'aly and Zhao \cite{HKZ} and Zhao \cite{Z4}.
A slightly finer notion is
\[
|\nabla^- f|(x):=\limsup_{y\rightarrow x}\frac{[f(y)-f(x)]_-}{d(x,y)}.
\]
 Clearly, $|\nabla^- f|\leq |\nabla f|$.
\end{remark}


\begin{definition}
	
	Let $(X,d,\nu)$ be a forward metric-measure space.
	Let $U:\mathbb{R}_+\rightarrow \mathbb{R}$ be a continuous convex function which is twice continuously differentiable on $(0,\infty)$ and let $\mu\in P_c^{ac}(X,\nu)$ with locally  Lipschitz density $\rho$. Define
	\[
	I_{U,\nu}(\mu):=\int_X \frac{|\nabla p(\rho)|^2}{\rho}{\ddd}\nu=\int_X \rho U''(\rho)^2 |\nabla \rho|^2{\ddd}\nu,\quad I^-_{U,\nu}(\mu) :=\int_X \rho U''(\rho)^2 |\nabla^- \rho|^2{\ddd}\nu,
	\]
	where $p(r)=rU'(r)-U(r)$.
	In particular, the  {\it Fisher information} is defined as
	\[
	I_\nu(\mu):=\int_X \frac{|\nabla\rho|^2}{\rho}{\ddd}\nu,\quad I^-_\nu(\mu):=\int_X \frac{|\nabla^-\rho|^2}{\rho}{\ddd}\nu.
	\]
\end{definition}

\begin{theorem}[Distorted HWI inequality]\label{HWIDISTOR} Let $(X,d,\nu)$ be a weak ${\CD}(K,N)$ space for some $K\in \mathbb{R}$ and $N\in (1,\infty]$. Given $\mu_0,\mu_1\in P^{ac}_c(X,\nu)$ such that $\mu_0$ has a Lipschitz density $\rho_0$,
there exists an optimal transference plan $\pi$ of coupling $(\mu_0,\mu_1)$ such that for any $U\in DC_N$,
\begin{align*}
		U_\nu(\mu_0)\leq& U^{\breve{\beta}^{(K,N)}_0}_{\breve{\pi},\nu}(\mu_1)+\int_{X\times X}p(\rho_0(x))\,\left[\left.\frac{d}{dt}\right|_{t=1}\beta^{(K,N)}_t(x,y)\right]\,\pi({\ddd}y|x)\nu({\ddd}x)\\
		&+\int_{X\times X}U''(\rho_0(x))\,|\nabla^- \rho_0(x)|\,d(x,y)\,\pi({\ddd}x{\ddd}y), \tag{5.29}\label{HWIDIST7.5}
	\end{align*}
	where $\pi({\ddd}x{\ddd}y)=:\pi({\ddd}y|x)\mu_0({\ddd}x)$.
	Moreover,
	\begin{itemize}
		
		\item[(i)] if $K=0$ and $U_\nu(\mu_1)<\infty$, then
		\[
		U_\nu(\mu_0)-U_\nu(\mu_1)\leq \int_{X\times X}U''(\rho_0(x))|\nabla^- \rho_0(x)|\,d(x,y)\,\pi({\ddd}x{\ddd}y)\leq W_2(\mu_0,\mu_1)\sqrt{I_{U,\nu}^-(\mu_0)};
		\]
		
		\item[(ii)] if $N=\infty$ and $U_\nu(\mu_1)<\infty$, then
		\begin{align*}
			U_\nu(\mu_0)-U_\nu(\mu_1)&\leq \int_{X\times X}U''(\rho_0(x))|\nabla^- \rho_0(x)|\,d(x,y)\,\pi({\ddd}x{\ddd}y)-\frac{K_{\infty,U}}2W_2(\mu_0,\mu_1)^2\\
			&\leq W_2(\mu_0,\mu_1)\sqrt{I_{U,\nu}^-(\mu_0)}-\frac{K_{\infty,U}}2W_2(\mu_0,\mu_1)^2,
		\end{align*}
		where
		\begin{align*}K_{N,U}:=\inf_{r>0}\frac{K p(r)}{r^{1-1/N}}=\left\{
			\begin{array}{lll}
				K\underset{r\rightarrow 0}{\lim}\frac{p(r)}{r^{1-1/N}}, & \text{ if }&K>0,\\
				0, & \text{ if }&K=0,\\
				K\underset{r\rightarrow \infty}{\lim}\frac{p(r)}{r^{1-1/N}},  & \text{ if }&K<0.
			\end{array}
			\right.
		\end{align*}
	\end{itemize}
\end{theorem}

\begin{proof}[Sketch of proof] Since the proof is similar to the reversible case (cf. Villani \cite[Theorem 20.10]{Vi}), we just outline it. Obviously, $\rho_0(\log\rho_0)_+$, $\rho_1(\log\rho_1)_+$, $\rho_0U'(\rho_0)$, $p^2(\rho_0)/\rho_0\in L^1(\nu)$, and  the weak curvature-dimension condition ${\CD}(K,N)$ yields an optimal transference plan $\pi$ of coupling $(\mu_0,\mu_1)$ associated with a displacement interpolation $(\mu_t)_{0\leq t\leq 1}$ such that (\ref{weakcdkncon}) holds for all $U\in DC_N$. Thus, by repeating the proof of (C-1) in Theorem \ref{supporsetcdkn}, one can easily show $\mu_t=\rho_t\nu\in P^{ac}_c(X,\nu)$ for each $t\in [0,1]$.
	For convenience, we also set
\[
\beta_t(x,y):=\beta^{(K,N)}_t(x,y),\quad \beta(x,y):=\lim_{t\rightarrow 0}\beta_t(x,y), \quad \beta'(x,y):=\left.\frac{d}{dt}\right|_{t=1}\beta_t(x,y).
\]
	The proof  is divided into three steps.
	
	\smallskip
	
	\noindent\textbf{Step 1.  Suppose that $\beta,\beta'$ are bounded.}
	Since $\mu_t=\rho_t\nu$, the weak  ${\CD}(K,N)$ condition (\ref{weakcdkncon}) furnishes
	\begin{align*}
		&\int_{X\times X} U\left( \frac{\rho_0(x)}{\beta_{1-t}(x,y)} \right)\beta_{1-t}(x,y)\pi({\ddd}y|x)\nu({\ddd}x)\leq \int_{X\times X}U\left( \frac{\rho_1(y)}{\beta_t(x,y)} \right)\beta_t(x,y)\pi({\ddd}x|y)\nu({\ddd}y)\\
		+&\int_{X\times X}\left[ \frac{U\left( \frac{\rho_0(x)}{\beta_{1-t}(x,y)} \right)\beta_{1-t}(x,y)-U(\rho_0(x))}{t} \right]\pi({\ddd}y|x)\nu({\ddd}x)-\frac1t\int_X \left[ U(\rho_t(x))-U(\rho_0(x)) \right]\nu({\ddd}x),\tag{5.30}\label{imHWI7.6}
	\end{align*}
	where $\pi({\ddd}x{\ddd}y)=\mu_0({\ddd}x)\pi({\ddd}y|x)=\pi({\ddd}x|y)\mu_1({\ddd}y)$.
	We need to pass to the limit as $t\rightarrow 0$.

	\smallskip
	
	\textbf{$1^{\bf st}$ term of (\ref{imHWI7.6}):} If $K=0$, then $\beta_t\equiv 1$ and hence,
	\begin{align*}
		&\int_{X\times X} U\left( \frac{\rho_0(x)}{\beta_{1-t}(x,y)} \right)\beta_{1-t}(x,y)\pi({\ddd}y|x)\nu({\ddd}x)=\int_{X\times X}U(\rho_0){\ddd}\nu=U_\nu(\mu_0).
	\end{align*}
	If $K>0$, then $t\mapsto \beta_t(x,y)$ is a decreasing function and $\beta_1=1$; since $r\mapsto U(r)/r$ is  nondecreasing, we have
	\[
	-U_-\left(\frac{\rho_0(x)}{\beta(x,y)}\right)\beta(x,y)\leq U\left(\frac{\rho_0(x)}{\beta(x,y)}\right)\beta(x,y)\leq U\left(\frac{\rho_0(x)}{\beta_{1-t}(x,y)}\right)\beta_{1-t}(x,y)\nearrow U(\rho_0(x)) \text{ as }t\textcolor[rgb]{1.00,0.00,0.00}{\searrow} 0.
	\]
	Since $\mu_0\in P^{ac}_c(X,\nu)$, it follows by
	Villani \cite[Theorem 17.28]{Vi} that $U_-(\rho_0/\beta)\beta$ is integrable. Thus,  the Lebesgue dominated convergence theorem yields
	\[
	\lim_{t\rightarrow0}\int_{X\times X}U\left(\frac{\rho_0}{\beta_{1-t}(x,y)}\right)\beta_{1-t}(x,y)\pi({\ddd}y|x)\nu({\ddd}x)= \int_{X}U(\rho_0){\ddd}\nu.\tag{5.31}\label{frsittermest7.7}
	\]
	The case of $K<0$ follows in a similar manner.

	\smallskip
	
	\textbf{$2^{\bf nd}$ term of (\ref{imHWI7.6}):} The process is almost the same as  in the 1$^{\rm st}$ term estimate. In fact,
	Villani \cite[Theorem 17.8]{Vi} together with the Lebesgue dominated convergence theorem furnishes
	\[
	\lim_{t\rightarrow 0} \int_{X\times X}U\left( \frac{\rho_1(y)}{\beta_t(x,y)} \right)\beta_t(x,y)\pi({\ddd}x|y)\nu({\ddd}y)= U^{\breve{\beta}^{(K,N)}_0}_{\breve{\pi},\nu}(\mu_1).\tag{5.32}\label{sedondtermestimhwi}
	\]

	\smallskip
	
	\textbf{$3^{\bf rd}$ term of (\ref{imHWI7.6}):} It suffices to consider the case when $K\neq 0$. Since $b\mapsto bU(r/b)$ is convex  and $\frac{d}{db} (bU(r/b))=-p(r/b)$, we have
	\begin{align*}
		\frac{U\left( \frac{\rho_0(x)}{\beta_{1-t}(x,y)} \right)\beta_{1-t}(x,y)-U(\rho_0(x))}{t}\leq p\left( \frac{\rho_0(x)}{\beta_{1-t}(x,y)} \right)\,\left(\frac{1-\beta_{1-t}(x,y)}t\right).\tag{5.33}\label{Hwithirdestimate}
	\end{align*}
	Note that $p$ is nondecreasing, since $p'(r)= r U''(r)\geq 0$. If $K>0$, then $\beta_{1-t}\geq 1$ is decreasing as $t\searrow0$, which implies
	\[
	0=p(0)\leq p\left( \frac{\rho_0(x)}{\beta_{1-t}(x,y)} \right)\nearrow p(\rho_0)\text{ as }t\searrow 0.\tag{5.34}\label{contalldom1}
	\]
	Moreover, consider the nonnegative function $f(t,x,y):=\left|\frac{1-\beta_{1-t}(x,y)}{t}\right|$ for $t\in (0,1]$ and its limit $f(0,x,y):=|\beta'(x,y)|$. A continuity argument on $f$ and the boundedness of $\beta,\beta'$ imply the existence of a small number $\delta>0$ and a constant $S>0$ such that
	\[
	0\leq f(t,x,y)\leq \max\left\{|\beta'(x,y)|+1,\, \left|\frac{1-\beta(x,y)}{\delta}\right|   \right\}\leq S<+\infty.
	\]
The latter fact together with (\ref{contalldom1}) and the Lebesgue dominated convergence theorem yields
	\[
	\lim_{t\rightarrow 0}\int_{X\times X} p\left( \frac{\rho_0(x)}{\beta_{1-t}(x,y)} \right) \left(\frac{1-\beta_{1-t}(x,y)}{t}\right)\pi({\ddd}y|x)\nu({\ddd}x)=\int_{X\times X}p(\rho_0(x))\beta'(x,y)\pi({\ddd}y|x)\nu({\ddd}x).
	\]
	Thus, relation (\ref{Hwithirdestimate}) immediately furnishes
	\begin{align*}
		\limsup_{t\rightarrow 0}\int_{X\times X}\frac{U\left( \frac{\rho_0(x)}{\beta_{1-t}(x,y)} \right)\beta_{1-t}(x,y)-U(\rho_0(x))}{t}\pi({\ddd}y|x)\nu({\ddd}x)\leq \int_{X\times X}p(\rho_0(x))\beta'(x,y)\pi({\ddd}y|x)\nu({\ddd}x).\tag{5.35}\label{finallythirdHWI}
	\end{align*}

	If $K<0$, then $ \beta_{1-t}(x,y)\leq 1$ is increasing as $t\searrow 0$ and hence,
	\begin{align*}
		0\leq  p\left( \frac{\rho_0(x)}{\beta_{1-t}(x,y)} \right)\leq p\left( \frac{\rho_0(x)}{\beta(x,y)}\right).
	\end{align*}
By a similar argument as  above  one can easily show that (\ref{finallythirdHWI}) holds again.

	\smallskip
	
	\textbf{$4^{\bf th}$ term of (\ref{imHWI7.6}):} {The same argument as in Villani \cite[Theorem 20.1]{Vi} furnishes}
	\begin{align*}
		\limsup_{t\rightarrow0}-\frac1t\int_X \left[ U(\rho_t(x))-U(\rho_0(x)) \right]\nu({\ddd}x)
		\leq \int_{X\times X} U''(\rho_0(x))\,|\nabla^-\rho_0 |(x)\, d(x,y)\,\pi({\ddd}x{\ddd}y).\tag{5.36}\label{forthtermHWI}
	\end{align*}
Therefore, (\ref{HWIDIST7.5}) follows by  (\ref{imHWI7.6})-(\ref{sedondtermestimhwi}), (\ref{finallythirdHWI})        and (\ref{forthtermHWI}), respectively.

	\bigskip
	
	\noindent\textbf{Step 2. Relaxation of the assumptions on $\beta$.} If $N<\infty$, or  $K\leq 0$, or  $K>0$ with $\diam(M)<D_{K,N}:=\pi\sqrt{{(N-1)}/{K}}$, both $\beta$ and $\beta'$ are bounded. Therefore, the aforementioned argument remains valid. Hence, the only problem is the case when $K>0$ and $\diam(M)=D_{K,N}$. Thus, for any $N'>N$, Step 1 yields
	\begin{align*}
		U_\nu(\mu_0)\leq& U^{\breve{\beta}^{(K,N')}_0}_{\breve{\pi},\nu}(\mu_1)+\int_{X\times X}p(\rho_0(x))  \left[\left.\frac{d}{dt}\right|_{t=1}\beta_t^{(K,N')}(x,y)\right] \pi({\ddd}y|x)\nu({\ddd}x)\\
		&+\int_{X\times X}U''(\rho_0(x))\,|\nabla^- \rho_0(x)|\,d(x,y)\,\pi({\ddd}x{\ddd}y).\tag{5.37}\label{HWIFORSTEP1}
	\end{align*}
	In view of Remark \ref{limitdefinitionofbeta}, we obtain
	\[
	U^{\breve{\beta}^{(K,N')}_0}_{\breve{\pi},\nu}(\mu_1)\rightarrow U^{\breve{\beta}^{(K,N)}_0}_{\breve{\pi},\nu}(\mu_1), \text{ as $N'\searrow N$}.\tag{5.38}\label{fiststepesc}
	\]
	On the other hand, $(\beta^{(K,N')})'_1:=\left.\frac{d}{dt}\right|_{t=1}\beta_t^{(K,N')}(x,y)$ is decreasing as $N'\searrow N$ and hence,
	\[
	\int_{X\times X}p(\rho_0(x))\, (\beta^{(K,N')})'_1(x,y) \,\pi({\ddd}y|x)\nu({\ddd}x)\rightarrow \int_{X\times X}p(\rho_0(x))\, (\beta^{(K,N)})'_1(x,y) \,\pi({\ddd}y|x)\nu({\ddd}x),
	\]
	which together with (\ref{fiststepesc}) and (\ref{HWIFORSTEP1}) implies (\ref{HWIDIST7.5}).
	
	\bigskip

	\noindent\textbf{Step 3.} In this step, we prove Statements (i)\&(ii). For (i), since $K=0$, one has $\beta=1$ and $\beta'=0$. Hence, $U^{\breve{\beta}^{(K,N)}_0}_{\breve{\pi},\nu}(\mu_1)=U_\nu(\mu_1)$.
	Moreover,
	the Cauchy-Schwarz   inequality furnishes
	\begin{align*}
		&\int_{X\times X}U''(\rho_0(x))\,|\nabla^-\rho_0(x)|\,d(x,y)\,\pi({\ddd}x{\ddd}y)\\
		\leq &\sqrt{\int_{X\times X}d(x,y)^2\pi({\ddd}x{\ddd}y)}\sqrt{\int_{X\times X}U''(\rho_0(x))^2|\nabla^- \rho_0(x)|^2\pi({\ddd}x{\ddd}y)}=W_2(\mu_0,\mu_1)\sqrt{I_{U,\nu}^-(\mu_0)}.\tag{5.39}\label{caswinHWI}
	\end{align*}
	Thus, (i) follows directly by  (\ref{HWIDIST7.5}).

	\smallskip
	
	(ii) For $N=\infty$, let $u(\delta):=U(e^{-\delta})e^{\delta}$, which is a convex function with $u'(\delta)=-e^{\delta}p(e^{-\delta})$. Hence,
	\[
	U(e^{-\delta_2})e^{\delta_2}\leq U(e^{-\delta_1})e^{\delta_1}+e^{\delta_2}p(e^{-\delta_2})(\delta_1-\delta_2),
	\]
	which together with $\delta_1:=\log(\frac1{\rho_1(y)})$, $\delta_2:=\log(\frac{\beta(x,y)}{\rho_1(y)})$ and $\beta(x,y):=e^{\frac{K}{6}d(x,y)^2}$ yields
	\begin{align*}
		U\left( \frac{\rho_1(y)}{\beta(x,y)}  \right)\frac{\beta(x,y)}{\rho_1(y)}\leq  \frac{U(\rho_1(y))}{\rho_1(y)}-\frac{\beta(x,y)}{\rho_1(y)}p\left( \frac{\rho_1(y)}{\beta(x,y)}\right)\frac{K}6d(x,y)^2\leq \frac{U(\rho_1(y))}{\rho_1(y)}-\frac{K_{\infty,U}}{6}d(x,y)^2.
	\end{align*}
	Hence, by Remark \ref{firstUv}, we get
	\begin{align*}
		U^{ \breve{\beta}}_{\breve{\pi},\nu}(\mu_1)\leq
		\int_X U(\rho_1(y))\nu({\ddd}y)-\frac{K_{\infty,U}}{6}\int_{X\times X}d(x,y)^2\pi({\ddd}x{\ddd}y)=U_\nu(\mu_1)-\frac{K_{\infty,U}}{6} W_2(\mu_0,\mu_1)^2.\tag{5.40}\label{HwiIIfirstex}
	\end{align*}
	On the other hand, since $\beta'(x,y)= -\frac{K}3d(x,y)^2$, we have
	\begin{align*}
		&\int_{X\times X}p(\rho_0(x))\beta'(x,y)\pi({\ddd}y|x)\nu({\ddd}x)=-\int_{X\times X}\frac{K}3\frac{p(\rho_0(x))}{\rho_0(x)}d(x,y)^2\pi({\ddd}x{\ddd}y)\leq-\frac{K_{\infty,U}}3W_2(\mu_0,\mu_1)^2,
	\end{align*}
	which combined with (\ref{HwiIIfirstex}) and (\ref{HWIDIST7.5}) implies
	\begin{align*}
		U_\nu(\mu_0)-U_\nu(\mu_1)&\leq \int_{X\times X}U''(\rho_0(x))|\nabla^- \rho_0(x)|d(x,y)\pi({\ddd}x{\ddd}y)-\frac{K_{\infty,U}}2W_2(\mu_0,\mu_1)^2.
	\end{align*}
	The latter inequality together with (\ref{caswinHWI}) then yields (ii).
\end{proof}

\begin{theorem}[HWI and log-Sobolev inequalities]\label{HWIinequality} Let $(X,d,\nu)$ be a weak ${\CD}(K,\infty)$ space for some $K\in \mathbb{R}$ and let $\mu_i\in P_c(X;\nu)$, $i=0,1$. If $\mu_0=\rho_0\nu$ has a Lipschitz density $\rho_0$, then
	\begin{align*}
		H_\nu(\mu_0)\leq H_\nu(\mu_1)+ W_2(\mu_0,\mu_1)\sqrt{I^-_{\nu}(\mu_0)}-\frac{K}2W_2(\mu_0,\mu_1)^2,
	\end{align*}
	where $H(r):=r\log r$. Moreover, the following statements are true:
	\begin{itemize}
		
		\item  there holds
		\[
		H_\nu(\mu_0)\leq W_2(\mu_0,\nu)\sqrt{I^-_{\nu}(\mu_0)}-\frac{K}2W_2(\mu_0,\nu)^2;
		\]

		\item  if  $K>0$, then
		\[
		H_\nu(\mu_0)\leq \frac{I_\nu^-(\mu_0)}{2K},\quad \int_X f^2{\ddd}\nu\leq \frac{1}{K}\int_X |\nabla^-f|^2{\ddd}\nu,
		\]
		where  $f:\supp\nu\rightarrow \mathbb{R}$ is a Lipschitz function with $\displaystyle\int_X f{\ddd}\nu=0$.
	\end{itemize}
\end{theorem}
\begin{proof}
	Without loss of generality, we may assume  $H_\nu(\mu_1)<\infty$.  Since $\mu_0,\mu_1\in P_c(X;\nu)\subset P_2(X)$, Theorem \ref{wassdistacbasisthe} yields $W_2(\mu_0,\mu_1)<\infty$; thus, the assumption guarantees a displacement interpolation $(\mu_t)_{0\leq t\leq 1}$ associated with an optimal transference plan $\pi$ of coupling $(\mu_0,\mu_1)$ such that
	\[
	H_\nu(\mu_t)\leq (1-t)H_\nu(\mu_0)+t H_\nu(\mu_1)-\frac{K}2 t(1-t)W_2(\mu_0,\mu_1)^2<+\infty.
	\]
	It follows by the proof of (C-1) {in} Theorem \ref{supporsetcdkn} that $\mu_t$ is absolutely continuous w.r.t. $\nu$. In particular, $\mu_i\in P^{ac}_c(M,\nu)$, $i=0,1$. All statements  - except the last, Poincar\'e inequality - directly follow by Theorem \ref{HWIDISTOR} and Cauchy-Schwarz inequality, respectively.
	
	In order to show the last inequality, consider the probability measure $\mu_\varepsilon=\rho_\varepsilon \nu=(1+\varepsilon f)\nu$ for $\varepsilon>0$, where $f:\supp\nu\rightarrow \mathbb{R}$ is a Lipschitz function with $\displaystyle\int_X f{\ddd}\nu=0$. The Taylor expansion yields
	\begin{align*}
		H_\nu(\mu_\varepsilon)&=\int_X \rho_\varepsilon \log \rho_\varepsilon {\ddd}\nu=\frac{\varepsilon^2}2\int_X {f^2}{\ddd}\nu+o(\varepsilon^2),\ I_\nu^-(\mu_\varepsilon)=\int_X \frac{|\nabla^-\rho_\varepsilon|^2}{\rho_\varepsilon}{\ddd}\nu=\int_X \frac{\varepsilon^2 |\nabla^-f|^2}{1+\varepsilon f}{\ddd}\nu.
	\end{align*}
It remains to let $\varepsilon\rightarrow 0$ in  $H_\nu(\mu_\varepsilon)\leq \frac{I_\nu^-(\mu_\varepsilon)}{2K}$, which implies the required Poincar\'e inequality.
\end{proof}

\begin{remark}
	In order to avoid further technicalities, Theorems \ref{HWIDISTOR}-\ref{HWIinequality} are presented only for measures  $\mu_i\in P^{ac}_c(X;\nu)$, $i=0,1$. We notice however that  a suitable modification of the proof
	allows to handle more general measures.
\end{remark}


\begin{theorem}[Sobolev inequality in weak ${\CD}(K,N)$ spaces]\label{theorem-Sob} Let $(X,\star,d,\nu)$ be a pointed forward $\Theta$-geodesic-measure space which satisfies the weak dimension-curvature bound ${\CD}(K,N)$ for some $K<0$ and $N\in (1,\infty)$. Thus, for any $R>0$, there exist constants $A=A(K,N,R,\Theta(R))$ and  $B=B(K,N,R,\Theta(R))$ such that for any Lipschitz function $f$ supported in $B^+_\star(R)$,
	\[
	\|f\|_{L^{\frac{N}{N-1}}(\nu)}\leq A \|\nabla^- |f|\|_{L^1(\nu)}+B\|f\|_{L^1(\nu)}\leq A \|\nabla f\|_{L^1(\nu)}+B\|f\|_{L^1(\nu)}.\tag{5.41}\label{L1sobolevineq}
	\]
	Moreover, if $\Theta$ is a constant function, then \eqref{L1sobolevineq} remains valid for any Lipschitz function supported in any forward $R$-ball.
\end{theorem}
\begin{proof}We just prove the first inequality in (\ref{L1sobolevineq}), since  the second one immediately follows by the inequalities $|\nabla^-|f||\leq |\nabla|f||\leq |\nabla f|.$
Due to Proposition \ref{cosntmupcd}/(ii), we can assume  $\nu[B^+_\star(R)]=1$ (otherwise, replace $\nu$ by  $\tilde{\nu}:=\nu/\nu[B^+_\star(R)]$).
Given  a non-zero $C$-Lipschitz function $f$ supported in $B^+_\star(R)$, we have
\[
|f(x)|-|f(y)|\leq |f(x)-f(y)|\leq C \cdot \Theta(R) \cdot d(x,y),\ \forall\,x,y\in X,
\]
which together with $\max_{x\in X}|f(x)|<\infty$ implies that $|f|^{\frac{N}{N-1}}(x)$ is a Lipschitz function. In particular, a direct calculation yields
\[
\left|\nabla^-|f|^{\frac{N}{N-1}}\right|(x)=\frac{N}{N-1}\,|f(x)|^{\frac1{N-1}}\,|\nabla^-|f||(x).\tag{5.42}\label{graduN}
\]
Now set $\mu_i:=\rho_i\nu$ for $i=0,1$, where
	\[
	\rho_0(x):=\frac{|f|^{\frac{N}{N-1}}(x)}{Z},\ \ \ Z:=\int_X |f|^{\frac{N}{N-1}}(y){\ddd}\nu(y)=\left(\|f\|_{L^{\frac{N}{N-1}}(\nu)}\right)^{\frac{N}{N-1}},\ \ \ \rho_1(x):=\frac{\textbf{1}_{B^+_\star(R)}}{\nu[B^+_\star(R)]}.
	\]
	Thus, $\mu_i\in P^{ac}_c(X;\nu)$ for $i=0,1$ and $\rho_0$ is Lipschitz. Consider $U_N(r):=Nr(1-r^{-\frac1N})\in DC_N$.  According to (\ref{HWIDIST7.5}), we have
	\begin{align*}
		&N-N\int_X \rho_0^{1-\frac1N}(x)\nu({\ddd}x)\leq N-N\int_{X\times X} \rho_1(y)^{1-\frac{1}{N}}\beta(x,y)^{\frac1N}\pi({\ddd}x|y)\nu({\ddd}y)\\
		&+\int_{X\times X}\rho_0(x)^{1-\frac1N}\beta'(x,y)\pi({\ddd}y|x)\nu({\ddd}x)+\frac{N-1}{N}\int_{X\times X} \rho_0(x)^{-1-\frac1N}|\nabla^- \rho_0(x)|\,d(x,y)\,\pi({\ddd}x{\ddd}y).\tag{5.43}\label{sol1-key}
	\end{align*}
	Obviously, $d(x,y)< (1+\Theta(R))R$ for any $x,y\in B^+_\star(R)$.
	Since $\suppor\rho_i\subset B^+_\star(R)$ for $i=0,1$, the quantities $\beta,\beta'$ in (\ref{sol1-key}) are bounded by constants which depend on  $N,K,R,\Theta(R)$. Hence, there exist positive constants $L_i=L_i(N,K,R,\Theta(R))$, $i=1,2$ such that
	\[
	-\int_X \rho_0^{1-\frac1N}(x){\ddd}\nu(x)\leq -L_1\,\nu[B^+_\star(R)]^{\frac1N}+L_2\, \left[ \int_{X} \rho_0^{1-\frac1N}(x){\ddd}\nu(x)+\int_X \rho_0^{-\frac1N}(x)|\nabla^- \rho_0(x)|{\ddd}\nu(x)  \right].\tag{5.44}\label{soli2}
	\]
	An easy computation combined with (\ref{graduN}) yields
	\begin{align*}
		\int_X \rho_0^{1-\frac1N}(x){\ddd}\nu(x)=\frac{\|f\|_{L^1(\nu)}}{\|f\|_{L^{\frac{N}{N-1}}(\nu)}},\quad \int_X \rho_0^{-\frac1N}(x)|\nabla^-\rho_0|(x){\ddd}\nu(x)=\left(\frac{N}{N-1}\right)\frac{\|\nabla^- |f|\|_{L^1(\nu)}}{\|f\|_{L^{\frac{N}{N-1}}(\nu)}},
	\end{align*}
	which together with $\nu[B^+_\star(R)]=1$ and (\ref{soli2})  imply (\ref{L1sobolevineq}).
	
	Furthermore, if $\Theta$ is a constant, then  $d(x,y)< (1+\Theta )R$ for any $x,y$ belonging to any forward $R$-ball. The same argument as above yields the rest of the proof.
\end{proof}

\begin{remark}Instead of $|\nabla f|$, one may expect to control $|\nabla^-|f||$ by $|\nabla^- f|$ in (\ref{L1sobolevineq}); however, this  is usually impossible.
For example, for any fixed $k_1,k_2>0$, consider the function $f:(-\epsilon+x,x+\epsilon)\rightarrow \mathbb{R} $ given by
$
 f(y):=-1+k_1(y-x)  \text{ if }y\geq x,$ and $f(y):=-1+k_2(y-x)  \text{ if }y<x.$
Thus $|\nabla^-|f||(x)=k_1$ and $|\nabla^-f|(x)=k_2$.
\end{remark}

\begin{remark}
	While the last part of Theorem \ref{HWIinequality} establishes a Sobolev inequality on spaces $(X,\star,d,\nu)$ satisfying the weak dimension-curvature bound ${\CD}(K,N)$ for some $K>0$, Theorem \ref{theorem-Sob} deals with its counterpart  $K<0.$ Due to Theorem \ref{diametercontrol}, in the former case $K>0$ we have  $\nu[X]<\infty$. We also notice that Sobolev inequalities have been studied on ${\CD}(0,N)$ spaces by Krist\'aly \cite{Kristaly-GN} (in particular, $\nu[X]=\infty$); in fact, in such a setting, the validity of a Sobolev inequality implies certain topological rigidities characterized as  non-collapsing properties of metric balls.
\end{remark}

\begin{theorem}\label{basissobolelemma}
	Let $(X,d,\nu)$ be a weak ${\CD}(K,N)$ space with some $K>0$ and $N\in (1,\infty)$. If $\nu$ is a probability measure, then for any probability measure $\mu=\rho\nu$ with
	Lipschitz continuous and strictly positive density $\rho$,  one has
	\begin{align*}
		H_{N,\nu}(\mu):=\int_{X}U_{N}(\rho){\ddd}\nu\leq \frac{3}{2K}\left( \frac{N-1}{N}\right)^2\int_{X}\frac{|\nabla^-\rho|^2}{\rho}  \frac{\rho^{-\frac2N}}{1+2\rho^{-\frac1N}} {\ddd}\nu,
	\end{align*}
	where $U_N(r):=Nr(1-r^{-\frac1N})\in DC_N$.
\end{theorem}
\begin{proof}
	Set $\alpha:=\sqrt{{K}/({N-1})}\,d(x,y)$; Theorem \ref{diametercontrol} implies $\alpha\in [0,\pi]$.
	Let $\rho_0:=\rho$ and $\rho_1:=1$.
	Due to  \eqref{sol1-key} and the expressions of $\beta,\beta'$ (see Villani \cite[p.530]{Vi}), we have
	\[
	H_{N,\nu}(\mu)\leq \int_X \mathcal {H}^{(N,K)}(\rho,|\nabla^- \rho|){\ddd}\nu,\tag{5.44}\label{lemmsov-1}
	\]
	where
	\begin{align*}
		&\mathcal {H}^{(N,K)}(r,g)\\
		=&r\sup_{\alpha\in [0,\pi]}\left[ \frac{N-1}{N}\left({\frac{g}{r^{1+1/N}}}\sqrt{\frac{N-1}{K}}\alpha \right)+ (N-1)\left( \frac{\alpha}{\tan \alpha}-1 \right)r^{-\frac1N}+N\left(1-\left( \frac{\alpha}{\sin \alpha}  \right)^{1-1/N} \right) \right].
	\end{align*}
	Note that $ \frac{\alpha}{\tan \alpha}-1\leq -\frac{\alpha^2}{3}$ and $N\left(1-\left( \frac{\alpha}{\sin \alpha}  \right)^{1-1/N} \right)\leq - {(N-1)}  \frac{\alpha^2}{6}$.
	Hence, we have
	\[
	\mathcal {H}^{(N,K)}(r,g)\leq r\sup_{\alpha\in[0,\pi]}\left( A\alpha-B\alpha^2  \right)\leq r \frac{A^2}{4B},\tag{5.45}\label{lemmsov-2}
	\]
	where $A=\frac{(N-1)^{\frac32}}{N\sqrt{K}}\frac{g}{r^{1+1/N}}$ and $B=\frac{N-1}{6}(1+2r^{-\frac1N})$.
	The result follows by (\ref{lemmsov-1}) and (\ref{lemmsov-2}).
\end{proof}

\begin{remark}
In the limit case $N\rightarrow \infty$, Theorem \ref{basissobolelemma} provides  the logarithmic Sobolev inequality from Theorem \ref{HWIinequality}, i.e., $H_\nu(\mu)\leq \frac{I_\nu^-(\mu)}{2K}$.
\end{remark}

We conclude this section by establishing a Lichnerowicz inequality on  weak $\CD(K,N)$ spaces with $K>0$, improving Theorem \ref{HWIinequality}.

\begin{theorem}[Lichnerowicz inequality in weak ${\CD}(K,N)$ spaces] Let $(X,d,\nu)$ be a weak ${\CD}(K,N)$ space with $K>0$ and $N\in (1,\infty)$. Then, for any Lipschitz function $f:\supp\nu\rightarrow \mathbb{R}$ with $\displaystyle\int_X f{\ddd}\nu=0$, one has
	\[
	\int_X f^2{\ddd}\nu\leq  \frac{N-1}{NK} \int_X  |\nabla^- f|^2{\ddd}\nu.
	\]
\end{theorem}

\begin{proof}
	In view of Theorem \ref{diametercontrol} and Proposition \ref{cosntmupcd}/(ii), we may assume that $\nu$ is a probability measure (otherwise, replace $\nu$ by $\tilde{\nu}:=\nu/\nu[X]$).
	Let us consider the probability measure $\mu_\varepsilon:=\rho_\varepsilon\nu=(1+\varepsilon f)\nu$, where $\varepsilon>0$ and $f$ is a Lipschitz function with $\displaystyle\int_X f{\ddd}\nu=0$. On the one hand, the Taylor expansion yields
	$$
		H_{N,\nu}(\mu_\varepsilon)=\varepsilon^2\frac{N-1}{2N}\displaystyle\int_X f^2{\ddd}\nu+o(\varepsilon^2).
$$
	On the other hand, we have
	\begin{align*}
	 	\int_{X}\frac{|\nabla^-\rho_\varepsilon|^2}{\rho_\varepsilon} \frac{\rho_\varepsilon^{-\frac2N}}{1+2\rho_\varepsilon^{-\frac1N}} {\ddd}\nu= \frac{\varepsilon^2}{3}\int_X |\nabla^-f|^2{\ddd}\nu+o(\varepsilon^2).
	\end{align*}
	Now the required inequality directly follows by Theorem \ref{basissobolelemma}, letting $\varepsilon \to 0.$
\end{proof}

\appendix



\section{Complementary results for forward metric spaces
}\label{propergenerlengappex}


\begin{proof}[Proof of Theorem \ref{topologychara}] \textbf{(i):} Firstly, we claim that for  any $R>0$ and any $x\in X$, there exists $r>0$ such that $B^+_x (r)\subset B^-_x(R)$.
Thanks to Remark \ref{forwardpointspaceandbackwardones}/(i), it suffices to show the claim for $x=\star$.
Since $\lim_{r\rightarrow 0^+}\Theta(r)r=0$, there exists $r>0$ with $\Theta(r)\cdot r<R$, which implies
\[
B^-_\star(r)\subset B^-_\star\left(\Theta(r)\cdot r\right)\subset B^-_\star(R).\tag{A.1}\label{ballcontain1}
\]
On the other hand, we always have $B^+_\star(r)\subset B^-_\star\left(\Theta(r)\cdot r\right)$,
which together with (\ref{ballcontain1}) furnishes the claim.

Now we prove Statement (i). Given an arbitrary backward ball $  B^-_y(R)$ and any point $x\in B^-_y(R)$,
 set $l:=d(x,y)$ and $\varepsilon:=(R-l)/2$. The triangle inequality yields   $B^-_x(\varepsilon)\subset B^-_y(R)$, while the claim furnishes $r>0$ with $B^+_x(r)\subset B^-_x(\varepsilon)$. Hence, we are done.

\smallskip

\noindent\textbf{(ii):}  For any $x\in X$ and $r>0$,
 it is trivial  $\widehat{B}_x(r)\subset B^+_x(2r)$, where $\widehat{B}_x(r):=\{y\in X| \, \hat{d}(x,y)<r\}$. This observation together with the triangle inequality implies $\mathcal {T}_+\subset \hat{\mathcal{T}}$ directly.

Now we turn to prove $\mathcal {T}_+\supset \hat{\mathcal{T}}$. It suffices to show that
for an arbitrary ball $\widehat{B}_x(r)$ and any $z\in \widehat{B}_x(r)$,  there holds $ B^+_z\left( \varepsilon/(1+\theta)\right)\subset \widehat{B}_x(r)$, where
$l:=\hat{d}(x,z)$, $\varepsilon:=(r-l)/2$ and $\theta:=\Theta(d(\star,x)+2r)$.

In fact, for any $q\in B^+_z\left( \varepsilon/(1+\theta)\right)$, we have
\begin{align*}
d(x,q)\leq d(x,z)+d(z,q)\leq 2l+\frac{\varepsilon}{(1+\theta)}<2r,
\end{align*}
which implies $B^+_z\left( \varepsilon/(1+\theta)\right)\subset B^+_x(2r)$.
On the other hand,  for the same $q$ as above, Remark \ref{forwardpointspaceandbackwardones}/(a) indicates $d(q,z)\leq \theta\, d(z,q)$, which
yields
\[
\hat{d}(z,q)=\frac12\left[ d(z,q)+d(q,z)  \right]\leq \frac{1+\theta}{2} d(z,q)<\frac{\varepsilon}{2}.
\]
Therefore, $\hat{d}(x,q)\leq \hat{d}(x,z)+\hat{d}(z,q)<l+\frac{\varepsilon}{2}<r$, i.e., $q\in \widehat{B}_x(r)$. This completes the proof.
 \end{proof}




In what follows, all the spaces are accessible.

\begin{proposition}\label{boundedlyinstrinct}Let $(X,d)$ be a    forward  metric space. If $(X,d)$ is forward boundedly compact, then for every two points $x,y\in X$ connected by a rectifiable curve there exists a shortest path between $x$ and $y$.
\end{proposition}
\begin{proof}[Sketch of the proof] Firstly, suppose that $X$ is compact. Thus, thanks to Theorem \ref{Arzela-Ascoli Theorem}, the proof in this case is the same as the one of Burago, Burago and Ivanov  \cite[Proposition 2.5.19]{DYS}. Secondly, assume that $(X,\star,d)$ is a (possibly noncompact) pointed forward $\Theta$-metric space.
Let $\gamma:[0,1]\rightarrow X$ be a rectifiable curve from $x$ to $y$. For any $t\in [0,1]$, the triangle inequality yields
\[
d(\star,\gamma(t))\leq d(\star,x)+d(x,\gamma(t))\leq  d(\star,x)+L_d(\gamma)=:R<\infty.
\]
Thus, $\gamma\subset \overline{B^+_\star(R)}$. Since $\overline{B^+_\star(R)}$ is compact, the proposition follows from the first part of the proof.
\end{proof}


\begin{definition}Let $(X,d)$ be a forward metric space. Given two points $x,y\in X$,
a point $z\in X$ is called a {\it midpoint} between the points $x,y$  if $d(x,z)=d(z,y)=\frac{1}{2}d(x,y)$.
\end{definition}

The relation between midpoints and forward geodesic spaces is as follows.
\begin{theorem}\label{mdistrctilylength}
Let $(X, d)$ be a forward complete   forward metric space. Thus, $(X,d)$ is a forward geodesic space if and only if
for every points $x,y\in X$, there exists a midpoint.
\end{theorem}
\begin{proof}
The ``$\Rightarrow$" part is obvious. For the ``$\Leftarrow$" part, given any two points $x,y\in X$, it suffices to construct a path $\gamma:[0,1]\rightarrow X$ from $x$ to $y$ with $\gamma(0)=x$, $\gamma(1)=y$ and $L_d(\gamma)=d(x,y)$.

 By the sandaled augment (cf. Burago, Burago and Ivanov  \cite[p.43]{DYS}), we can define $\gamma$ on every dyadic rational between $0$ and $1$ such that
\[
d(\gamma(a),\gamma(b))=(b-a)\cdot d(x,y), \text{ for every two dyadic rationals $0\leq a\leq b\leq 1$}.\tag{A.2}\label{dyadicnumber}
\]

For a non-dyadic rational $t_0\in I$, choose  a  sequence of dyadic rationales $(t_i)_i$ with $t_i\nearrow t_0$. Thus,  $(\gamma(t_i))_i$ is a forward Cauchy sequence and hence, it converges to some point $p$. We assign it as $\gamma(t_0)$. By (\ref{dyadicnumber}), it is not hard to check that $\gamma(t_0)$ is well-defined. In particular, there holds
\[
d(\gamma(a),\gamma(b))=(b-a)\cdot d(x,y), \ 0\leq a\leq b\leq 1,
\]
which implies $L_d(\gamma)=d(x,y)$.
\end{proof}



According to Definition \ref{shortpathdef}, if $(X,d)$ is a forward length space, $\gamma$ is a shortest path if and only if it is a minimal geodesic. Moreover, Proposition \ref{basisessentially2}/(iv) yields the following result.
\begin{proposition}\label{shorttoshort}
Let $(X, d)$  is a  forward  length space. If $\gamma_i:[0,1]\rightarrow X$ is a sequence of minimal geodesics converging pointwise to  a path $\gamma$ as $i\rightarrow \infty$, then $\gamma$ is also a minimal geodesic.
\end{proposition}

 The following result can be proved by the same method as employed in Burago, Burago and Ivanov \cite[Proposition 2.5.22]{DYS}.
\begin{proposition}\label{compctlacalcomapctboundecompact}Let $(X,d)$ be a forward complete locally compact forward  length space. Then  $(X,d)$ is forward boundedly compact.
\end{proposition}






\section{Additional results for Gromov-Hausdorff topology
}\label{appendixGro-Hauscon}


\begin{proposition}[Ascoli theorem in the $\theta$-Gromov-Hausdorff topology]\label{Ascloghtipo}
Let $(X_i,d_{X_i})_{i}$ be a sequence of compact $\theta$-metric spaces converging to a compact space $(X,d_X)$ in the $\theta$-Gromov-Hausdorff topology, by means of $\epsilon_i$-isometries $f_i:X_i\rightarrow X$; and let $(Y_i,d_{Y_i})_{i}$ be another sequence of  compact $\theta$-metric spaces converging to a compact space $(Y,d_Y)$ in the $\theta$-Gromov-Hausdorff topology, by means of $\epsilon_i$-isometries $g_i:Y_i\rightarrow Y$. Let $(\alpha_i)_{i\in \mathbb{N}}$ be a sequence of maps $X_i\rightarrow Y_i$ that are asymptotically equicontinuous, in the sense that for every $\varepsilon>0$, there are $\delta=\delta(\varepsilon)>0$ and $N=N(\varepsilon)\in \mathbb{N}$ such that for all $i\geq N$,
\[
d_{X_i}(x,y)\leq \delta\Longrightarrow d_{Y_i}\left(\alpha_i(x),\alpha_i(y)  \right)\leq \varepsilon.
\]
Then after passing to a subsequence, the maps $g_i\circ\alpha_i\circ (f_r)_i:X\rightarrow Y$ converge uniformly to a continuous map $\alpha:X\rightarrow Y$, where $(f_r)_i:X_i\rightarrow X$ is defined in Proposition \ref{deltanprooxima}.
\end{proposition}
\begin{proof}[Sketch of proof] Since $(X_i,d_{X_i})_{i}$ converges to $(X,d_X)$ in the $\theta$-Gromov-Hausdorff topology, then $(X_i,\hat{d}_{X_i})_{i}$ converges to $(X,\hat{d}_X)$ in the $1$-Gromov-Hausdorff topology, where   $\hat{d}_{X_i},\hat{d}_X$ are the symmetrized metrics of $d_{X_i}$ and $d$, respectively.
 Furthermore, $f_i:(X_i,\hat{d}_{X_i})\rightarrow (X,\hat{d}_X)$ is a $ {(1+\theta)\epsilon_i}/{2}$-isometry and $(f_r)_i:(X,\hat{d}_X)\rightarrow (X_i,\hat{d}_{X_i})$ is still an approximate inverse of $f_i$. Now the statement  follows from Villani \cite[Proposition 27.20]{Vi} and Theorem \ref{topologychara}/(ii).
\end{proof}

A similar argument together with Villani \cite[Proposition 27.22]{Vi} furnishes the following result.
\begin{proposition}[Prokhorov theorem in the $\theta$-Gromov-Hausdorff topology]\label{ProtheoremGromHauss} Let $(X_i,d_i)_{i}$ be a sequence of compact $\theta$-metric spaces converging to a compact space $(X,d)$ in the $\theta$-Gromov-Haudsdorff topology, by means of $\epsilon_i$-isometries $f_i:X_i\rightarrow X$. For each $i$, let $\mu_i$ be a probability measure on $X_i$. Then after extraction of a subsequence, $(f_i)_\sharp \mu_i$ converges to a probability measure $\mu$ in the weak topology.
\end{proposition}

\begin{proof}[Proof of Proposition \ref{welldefinednoncompactGHCONVER}] Since the proof is similar to the reversible case (cf. Burago, Burago and Ivanov  \cite{DYS}), we just give the sketch.
\smallskip

(i)  According to (\ref{noncompactballconverge}) and Proposition \ref{deltanprooxima}, given $r>0$ and $\epsilon>0$, one can find two subsets $Y_{r,\epsilon}\subset X,Y'_{r,\epsilon}\subset X'$ and a map $f_{r,\epsilon}:Y_{r,\epsilon}\rightarrow Y'_{r,\epsilon}$ such that
\begin{align*}
&\overline{B^+_\star(r-\epsilon)}\subset Y_{r,\epsilon}\subset \overline{B^+_\star(r+2\epsilon)},\ \overline{B^+_{\star'}(r-\epsilon)}\subset Y'_{r,\epsilon}\subset \overline{B^+_{\star'}(r+2\epsilon)},\\
&f_{r,\epsilon}(\star)=\star', \ f_{r,\epsilon}(Y_{r,\epsilon})=Y'_{r,\epsilon}, \  \dis f_{r,\epsilon}\leq (3+\Theta(r))\epsilon.
\end{align*}
Using the Cantor diagonal procedure first for $\epsilon\rightarrow 0^+$ and then for $r\rightarrow \infty$, one can end up with a distance-preserving map from a dense  subset of $X$ to $X'$, which extends to a distance-preserving map $f:X\rightarrow X'$ with $f(\star)=\star'$. The properties of $f_{r,\epsilon}$ imply that $f|_{\overline{B_\star(r)}}$ is a subjective map  and therefore, it is an isometry from $\overline{B_\star(r)}$ to $\overline{B_{\star'}(r)}$ for any $r>0$.

\smallskip

(ii) The proof is trivial.

\smallskip

(iii) First we show $\diam(X)\leq D$. In fact, given any $a,b\in X$, there exists $r>1$ such that $a,b\in \overline{B^+_\star(r-1)}$. For any $\epsilon\in (0,1)$,  one can find  $I=I(r,\epsilon)>0$ such that for $i>{I}$,  there is a map $f_i: \overline{B^+_{\star_i}(r)}\rightarrow X$ satisfying
\[
f_i(\star_i)=\star,\quad \text{dis}f_i\leq \epsilon,\quad \overline{B^+_\star(r-\epsilon)}\subset \left[ f_i\left( \overline{B^+_{\star_i}(r)} \right) \right]^\epsilon\subset\left[f_i(X_i)\right]^\epsilon,\tag{B.1}\label{B.1frconver}
\]
which implies $d(a,b)\leq  D+\epsilon$ and hence, $\diam(X)\leq D$. Now choose $r=D+1$ and any $\theta\geq \Theta(D)$. Thus, (\ref{B.1frconver}) together with Lemma \ref{inprotantisometry}/(ii)
furnishes $\lim_{i\rightarrow \infty}d^\theta_{GH}(\mathcal {X}_i,\mathcal {X})=0$.

\smallskip

(iv) Fix an arbitrary point $\star\in X$.
In view of the proof of Lemma \ref{inprotantisometry}/(i), by passing a subsequence, one can choose a sequence of points $\star_i\in X_i$ and a sequence of $(1+\theta)\epsilon_i$-isometries $f_i:X_i\rightarrow X$ such that $f_i(\star_i)=\star$ and  $\epsilon_i\rightarrow 0$.

 Given any $\delta>0$, there is a large $I=I(\delta)$ such that $ (1+3\theta)\epsilon_i<\delta$ for all $i>I$. Given any $i>I$ and $r>0$, for any $x\in \overline{B^+_{\star}(r-\delta)}$, there is $\tilde{x}\in X_i$ such that $d(f_i(\tilde{x}),x)\leq2\epsilon_i$ (see the proof of Lemma \ref{inprotantisometry}/(i)). Moreover, one has
\begin{align*}
d_i(\star_i,\tilde{x})\leq d(f_i(\star_i),f_i(\tilde{x}))+(1+\theta)\epsilon_i=d(\star,f_i(\tilde{x}))+(1+\theta)\epsilon_i\leq d(\star,x)+d(x,f_i(\tilde{x}))+(1+\theta)\epsilon_i<r,
\end{align*}
which implies $\tilde{x}\in \overline{B^+_{\star_i}(r)}$. Thus, the above construction yields $\overline{B^+_{\star}(r-\delta)}\subset \left[ f_i\left(\overline{B^+_{\star_i}(r)}\right)\right]^\delta$. We are done because
$\text{dis}f_i<\delta$ and $f_i(\star_i)=\star$.
\end{proof}

\begin{proof}[Sketch of the proof of Theorem \ref{noncompactprecompact}] Let $(X_i,\star_i,d_i)_i$ be an arbitrary sequence in $\mathscr{C}$. Given $r>0$, by passing a subsequence, a standard construction (cf. Burago, Burago and Ivanov \cite[Theorem 7.4.15]{DYS})
furnishes  a compact pointed $\Theta(r)$-metric space $(Y_r,\star)$  satisfying
\[
\lim_{i\rightarrow \infty}d^{\Theta(r)}_{pGH}\left( (\overline{B^+_{\star_i}(r)},\star_i),(Y_r,\star) \right)=0.\tag{B.2}\label{willdeletebutnowgeive}
\]
Letting $r\rightarrow \infty$ and using the Cantor diagonal argument, one can get {an}  irreversible metric space $Y$.

Note that each $Y_r$ is the completion of a collection of sequences $(x_i)_i$, where $x_i\in \overline{B^+_{\star_i}(r)}$. By this fact, a direct but tedious argument yields    $B^+_\star(r-\delta)\subset Y_r$ for any small $\delta>0$, which implies  $\overline{B^+_\star(r)}\subset \overline{Y_r}=Y_r$ and hence, $Y\in \mathcal {M}^{\Theta}_*$.

For any $\epsilon>0$, the proof of Lemma \ref{inprotantisometry}/(i) together with
 (\ref{willdeletebutnowgeive}) yields an $N(r,\epsilon)>0$ such that for any $i>N(r,\epsilon)$, there exists
 an $\epsilon$-isometry $f_{i}: \overline{B^+_{\star_i}(r)} \rightarrow Y_r\subset Y$ with $f_i(\star_i)=\star$ and
 \[
 \overline{B^+_\star(r-\epsilon)}\subset Y_r\subset \left[f_{i}\left(\overline{B^+_{\star_i}(r)}\right)\right]^\epsilon,
 \]
 which concludes the proof.
 \end{proof}

\begin{proof}[Proof of Proposition \ref{lenthcompactnoncompact}] Owing to Theorem \ref{mdistrctilylength}, it suffices to show that
for any $x,y\in X$, there exists a   midpoint $z$ between $x$ and $y$. In order to do this,
choose a large $R$ such that $x,y\in \overline{B^+_\star(R)}$ and set $\widetilde{R}:=(4+\Theta(R))R$.
Owing to (\ref{noncfincazse}), by passing a subsequence, the same argument in the proof Lemma \ref{inprotantisometry}/(ii) yields a sequence $\epsilon_i\rightarrow 0 $ and a sequence of $\Theta(\widetilde{R})$-admissible metric $\tilde{d}_i$ on $ \overline{B^+_{\star_i}(\widetilde{R})}\sqcup \overline{B^+_\star(\widetilde{R}-\epsilon_i)}$ satisfying
\[
  2(1+\Theta(\widetilde{R}))\epsilon_i+(2+\Theta(R))R<\widetilde{R},\quad
 \tilde{d}_{iH}\left(\overline{B^+_{\star_i}(\widetilde{R})},\overline{B^+_{\star}(\widetilde{R}-\epsilon_i)}\right)+\frac{\tilde{d}_i(\star_i,\star)+\tilde{d}_i(\star,\star_i)}{2}\leq \epsilon_i.
\]
For each $i$,
choose $x_i,y_i\in \overline{B^+_{\star_i}(\widetilde{R})}$ with $\tilde{d}_i(x_i,x)<\epsilon_i$ and $\tilde{d}_i(y_i,y)<\epsilon_i$. Then the triangle inequality yields
\begin{align*}
\tilde{d}_i(x_i,y_i)&\leq \tilde{d}_i(x_i,x)+ {d}(x,y)+\tilde{d}_i(y,y_i)\leq (1+\Theta(\widetilde{R}))\epsilon_i+(1+\Theta(R))R,\\
 \tilde{d}_i(\star_i,x_i)&\leq \tilde{d}_i(\star_i,\star)+{d}(\star,x)+\tilde{d}_i(x,x_i)\leq (1+\Theta(\widetilde{R}))\epsilon_i+R.
\end{align*}
Hence, the shortest path  from $x_i$ to $y_i$ is still contained in $\overline{B^+_{\star_i}(\widetilde{R})}$ and so is the midpoint $z_i$ between $x_i$ and $y_i$. Choose a point $\mathfrak{z}_i\in \overline{B^+_\star(\widetilde{R}-\epsilon_i)}$ such that $\tilde{d}_i(\mathfrak{z}_i,z_i)<\epsilon_i$. By passing a subsequence, one can assume that $(\mathfrak{z}_i)_i$ converges to some point $z\in \overline{B^+_\star(\widetilde{R})}\subset X$. It is not hard to check that $z$ is a midpoint between $x$ and $y$.
\end{proof}

\begin{proposition}\label{noncompactmesurecon}Let $(X_i,\star_i,d_i)_{i}$ be a sequence of pointed forward metric spaces in $\mathcal {M}^{\Theta}_*$ converging to a space $(X,\star,d)\in \mathcal {M}^{\Theta}_*$ in the pointed forward $\Theta$-Gromov-Hausdorff topology, by means of pointed $\epsilon_i$-isometries {$f_i:\overline{B^+_{\star_i}(R_i)}\rightarrow  \overline{B^+_{\star}(R_i)}$ with $\epsilon_i\rightarrow 0$ and $R_i\rightarrow \infty$.} For each $i\in \mathbb{N}$, let $\mu_i$ be a locally finite Borel measure on $X_i$.  Assume that for each $r>0$, there is a finite constant $M=M(r)$ such that $\mu_i[\overline{B^+_{\star_i}(r)}] \leq M$ for every $i\in \mathbb{N}$.
Then, there is a locally finite measure $\mu$ on $X$ such that, up to extraction of subsequence, $(f_i)_\sharp \mu_i\rightarrow\mu$
in the weak-$*$ topology.
\end{proposition}

\begin{proof}
For any fixed $r>0$, {there exists a large $I>0$ such that $r+\epsilon_i<R_i$ and $\epsilon_i<1$ for any $i>I$. Thus,}
\[
(f_i)_\sharp\mu_i\left[ \overline{B^+_\star(r)} \right]=\mu_i\left[(f_i)^{-1}\left( \overline{B^+_\star(r)}\right) \right]\leq \mu_i\left[ \overline{B^+_{\star_i}(r+\epsilon_i)}  \right]\leq M(r+1).
\]
By Alaoglu's theorem, we can assume that $(f_i)_\sharp\mu_i|_{  \overline{B^+_\star(r)} }$ converges to some finite measure $\mu_r$ in the weak-$*$ topology.
By letting $r\rightarrow \infty$ and applying a diagonal extraction, we can find $\mu$.
\end{proof}

\section{Auxiliary properties of the Wasserstein distance
}\label{optimaltarsn}

The same argument as in Villani \cite[Theorem 6.15]{Vi} yields the following result.

\begin{theorem}\label{totallvacontwp}Let $(X,d)$ be a forward boundedly compact forward metric space. Given $p\in [1,\infty)$, for any $x_0\in X$,
we have
\[
W_p(\mu,\nu)\leq 4\left( \int_X \hat{d}(x_0,x)^p d|\mu-\nu|(x)  \right)^{\frac1p},\ \forall\,\mu,\nu\in P_p(X).
\]
\end{theorem}

\begin{definition}
Let $(X,d)$ be a forward metric space  and  set $c:=d^{\,p}$ for some $p\in [1,\infty)$.
 A function $\psi:X\rightarrow \mathbb{R}\cup \{+\infty\}$ is said to be {\it $c$-convex} if it is not identically $+\infty$, and there exists $\zeta: X\rightarrow \mathbb{R}\cup \{\pm \infty\}$ such that
$\psi(x)=\sup_{y\in X}\left( \zeta(y)-c(x,y) \right)$ for any  $x\in X$.
Then its $c$-transform $\psi^c$ is defined by
$\psi^c(y):=\inf_{x\in X}\left( \psi(x)+c(x,y) \right)$ for any $y\in X$.

\end{definition}

\begin{lemma}[Kantorovich-Rubinstein distance]\label{W1KANrUBDIS}Let $(X,d)$ be
  a forward boundedly compact forward metric space. Thus for any $\mu,\nu\in P_1(X)$, we have
\[
W_1(\mu,\nu)=\sup_{\psi \in {\rm{Lip}}_1(X)}\left( \int_X \psi(y)d\nu(y)-\int_X\psi(x)d\mu(x) \right),\tag{C.1}\label{B.1}
\]
where ${\rm{Lip}}_1(X):=\{ f\in C(X): f(y)-f(x)\leq d(x,y),\ \forall\,x,y\in X  \}$.
\end{lemma}
\begin{proof}[Sketch of proof] Let $c(x,y):=d(x,y)$, which is a continuous function on the Polish space $(X,\hat{d})$. Thus, Kantorovich duality (cf. Villani \cite[Theorem 5.10/(i)]{Vi}) yields
\[
W_1(\mu,\nu)=\sup_{\{\psi \in L^1(\mu):\,\psi \text{ is $c$-convex} \}}\left( \int_X \psi^c(y)d\nu(y)-\int_X\psi(x)d\mu(x) \right).
\]

 On the other hand, the same argument as in Villani \cite[Particular Case 5.4]{Vi}, one can show the set of $c$-convex functions is exactly $\text{Lip}_1(X)$. In particular, if $\psi\in \text{Lip}_1(X)$, then $\psi^c=\psi$. Since $\mu\in P_1(X)$, it is easy to check $ \text{Lip}_1(X)\subset L^1(\mu)$. Thus, the lemma follows.
\end{proof}

\begin{lemma}\label{tightCauchy}
Let $(X,\star,d)$ be an element in $\mathcal {M}^\Theta_*$  such that $\Theta^{\frac{p}{p-1}}$ is a concave function for some $p\in [1,\infty)$. Every forward Cauchy sequence $(\mu_k)_k$ in $(P_p(X),W_p)$
is tight, i.e., for each $\varepsilon>0$, there exists a compact set $K_\varepsilon\subset X$ such that $\mu_k[X\backslash K_\varepsilon]\leq \varepsilon$ for each $k$.
\end{lemma}
\begin{proof}
Since $(\mu_k)_k$ is a forward Cauchy sequence, there exists $M>0$ such that $W_p(\mu_M,\mu_k)<1$ for any $k\geq M$. Hence, the triangle inequality of $W_p$ yields
 a constant $D$ with $W_p(\delta_{\star},\mu_k)\leq D$  for all $k$. Moreover, given $\varepsilon\in (0,1)$ and $\ell\in \mathbb{N}$, there exists $N=N(\varepsilon,\ell)\geq M$ so that $W_p(\mu_N,\mu_k)<2^{-2\ell-1}\varepsilon^2/\Theta(D+1)$ for any $k\geq N$. Therefore, for any $k\in \mathbb{N}$, there is $j\in \{1,\ldots,N\}$ satisfying
\[
W_p(\delta_\star,\mu_j)\leq D,\ W_p(\mu_j,\mu_k)<2^{-2\ell-1}\varepsilon^2/\Theta(D+1)<1.\tag{C.2}\label{controlWptoinetaulty}
\]

On the other hand, since the finite set $\{\mu_1,\ldots,\mu_N\}$ is always tight, there is a compact set $K$ such that $\mu_j[X\backslash K]<2^{-\ell-1}\varepsilon$ for all $j\in \{1,\ldots,N\}$. The compact set $K$ can be covered by a finite number of small balls
\[
K\subset \bigcup_{1\leq i\leq m(\ell)}B^+_{x_i}(2^{-\ell-1}\varepsilon)=:U_\ell.
\]
Now set $U_\ell^\varepsilon:=\{x\in X:\,d(U_\ell,x)<2^{-\ell-1}\varepsilon\}\subset \cup_{1\leq i\leq m(\ell)}B^+_{x_i}(2^{-\ell}\varepsilon)$. Now set
\[
\phi(x):=-\min\left\{0,\  \frac{d(U_\ell,x)}{2^{-\ell-1}\varepsilon}-1 \right\}.
\]
Thus, $\textbf{1}_{U_\ell}\leq \phi\leq \textbf{1}_{U_\ell^\varepsilon}$ and $2^{-\ell-1}\varepsilon\phi$ belongs to $\text{Lip}_1(X)$ (see Lemma \ref{W1KANrUBDIS}). Given $k\in \mathbb{N}$, choose $j\in \{1,\ldots,N\}$ such that $\mu_j$ satisfies (\ref{controlWptoinetaulty}). Then
(\ref{B.1}) together with Lemma \ref{reversbilityofW} and (\ref{controlWptoinetaulty}) furnishes
\begin{align*}
\mu_k\left[  \bigcup_{1\leq i\leq m(\ell)}B^+_{x_i}(2^{-\ell}\varepsilon) \right]&\geq\mu_k[U_\ell^\varepsilon]\geq \int_X \phi d\mu_k=\int_X \phi d\mu_j-\left(\int_X \phi d\mu_j -\int_X \phi d\mu_k  \right)\\
&\geq \int_X \phi d\mu_j-\frac{W_1(\mu_k,\mu_j)}{2^{-\ell-1}\varepsilon}\geq \int_X \phi d\mu_j-\Theta(D+1)\frac{W_p(\mu_j,\mu_k)}{2^{-\ell-1}\varepsilon}\\
&\geq \mu_j[U_\ell]-\Theta(D+1)\frac{W_p(\mu_j,\mu_k)}{2^{-\ell-1}\varepsilon}\geq \mu_j[K]-\Theta(D+1)\frac{W_p(\mu_j,\mu_k)}{2^{-\ell-1}\varepsilon}\geq 1-2^{-\ell}\varepsilon,
\end{align*}
which indicates
\[
\mu_k\left[ X\backslash \bigcup_{1\leq i\leq m(\ell)}B^+_{x_i}(2^{-\ell}\varepsilon) \right]\leq 2^{-\ell}\varepsilon,\ \forall \,k\in \mathbb{N}.
\]
Now set
$K_\varepsilon:=\cap_{1\leq \ell\leq \infty} \cup_{1\leq i\leq m(\ell)} \overline{B^+_{x_i}(2^{-\ell}\varepsilon)}$.
Thus, for any $k\in \mathbb{N}$, we have
\begin{align*}
\mu_k\left[ X\backslash K_\varepsilon \right]=\mu_k\left[ \bigcup_{1\leq \ell\leq \infty} \left(X\backslash\bigcup_{1\leq i\leq m(\ell)} \overline{B^+_{x_i}(2^{-\ell}\varepsilon)} \right) \right]\leq \sum_{\ell=1}^\infty\mu_k\left[ X\backslash \bigcup_{1\leq i\leq m(\ell)}B^+_{x_i}(2^{-\ell}\varepsilon) \right]\leq \sum_{\ell=1}^\infty2^{-\ell}\varepsilon=\varepsilon.
\end{align*}
It remains to show that $K_\varepsilon$ is compact. In fact, for any small $\delta>0$, choose an $\ell>0$ such that $2^{-\ell}\varepsilon<\delta$ and then
\[
K_\varepsilon\subset \bigcup_{1\leq i\leq m(\ell)} \overline{B^+_{x_i}(2^{-\ell}\varepsilon)}\subset\bigcup_{1\leq i\leq m(\ell)} {B^+_{x_i}(\delta)},
\]
 which implies that $K_\varepsilon$ is forward totally bounded. Since $(X,d)$ is forward complete and $K_\varepsilon$ is closed, the compactness of $K_\varepsilon$ follows from Theorem \ref{compactequvitheorem}.
\end{proof}

\begin{lemma}\label{tightsetoptima}
Let $(X,d)$ be a forward boundedly compact forward metric space.

\begin{itemize}
	\item[(1)] The function
	$\mathscr{F}:\pi\mapsto \int_{X\times X}d(x,y)^p {\ddd}\pi(x,y)$
	is lower semicontinuous
	on $P(X\times X)$, equipped {with the weak topology};
		\item[(2)]  Let $\mathcal {P},\mathcal {Q}\subset P(X)$ be two compact subsets (with respect to the weak topology). Then the set of optimal transference plans
		$\pi$ whose marginals respectively belong to $\mathcal {P},\mathcal {Q}$ is itself compact in
		$P(X \times X)$.
\end{itemize}
\end{lemma}
\begin{proof} Let $c:=d^{\,p}$, which is a non-negative continuous function on the Polish space $(X,\hat{d})$.
Thus, (1) and (2) follow immediately  from Villani   \cite[Lemmas 4.3, Corollary 5.21]{Vi} respectively.
\end{proof}

\begin{lemma}\label{lowrconunityc}
Let $(X,d)$ be a  forward boundedly compact forward metric space. If a sequence of probability measures $(\mu_k)_k$ (resp., $(\nu_k)_k$) converges weakly to $\mu$ (resp., $\nu$), then $W_p(\mu,\nu)\leq \underset{k\rightarrow\infty}{\lim\inf}W_p(\mu_k,\nu_k)$ for any $p\in [1,\infty)$.
\end{lemma}
\begin{proof}
Let $M:=\{\mu,\mu_k,k\in \mathbb{N}\}$ and $N:=\{\nu,\nu_k,k\in \mathbb{N}\}$ be two compact sets in the weak topology.  Let $\pi_k$ denote an optimal  transference plans from $\mu_k$ to $\nu_k$. Thus,
 Lemma \ref{tightsetoptima}/(2) implies that a subsequence $(\pi_{k_l})_l$ converges weakly to an optimal transference plan $\pi$ from $\mu$ to $\nu$.
Moreover,
thanks to Lemma \ref{tightsetoptima}/(1), we have
\[
W_p(\mu,\nu)^p= \int_{X\times X} d(x,y)^p{\ddd}\pi(x,y)\leq \underset{l\rightarrow \infty}{\lim\inf}\int_{X\times X} d(x,y)^p {\ddd}\pi_{k_l}(x,y)=\underset{l\rightarrow \infty}{\lim\inf}W_p(\mu_{k_l},\nu_{k_l})^p.
\]
Then a standard argument by contradiction furnishes  $W_p(\mu,\nu)^p\leq \underset{k\rightarrow \infty}{\lim\inf}W_p(\mu_k,\nu_k)^p$.
\end{proof}

\begin{corollary}\label{backwardconvergence}
Let $(X,\star,d)$ be an element in $\mathcal {M}^\Theta_*$  such that $\Theta^{\frac{p}{p-1}}$ is a concave function for some $p\in [1,\infty)$. Thus, for any forward Cauchy sequence $(\mu_k)_k$  in $(P_p(X),W_p)$, there exists $\mu\in P(X)$ with $\lim_{k\rightarrow\infty} W_p(\mu_k,\mu)=0$.
\end{corollary}
\begin{proof}
According to Lemma \ref{tightCauchy}, $(\mu_k)_k$ is tight.  By Theorem \ref{topologychara}/(ii) and Prokhorov's theorem, there exists a subsequence $(u_{k'})_{k'}$ converging weakly to some measure $\mu\in P(X)$. Since $(u_{k'})_{k'}$ is still a forward Cauchy sequence, for  each $\varepsilon>0$, there exists $N_1>0$ such that for $N_1<l'\leq k'$, we have $W_p(\mu_{l'},\mu_{k'})<\varepsilon$, which together with  Lemma \ref{lowrconunityc} yields
$W_p(\mu_{l'},\mu)\leq \underset{k'\rightarrow \infty}{\lim\inf}W_p(\mu_{l'},\mu_{k'})\leq \varepsilon$.

On the other hand,
since $(u_k)_k$ is a forward Cauchy sequence, for any $\varepsilon>0$, there exists $N_2>0$ such that $W_p(\mu_k,\mu_l)<\varepsilon$ for any $N_2<k<l$. Now choose a large $l'>\max\{N_1,N_2\}$, the triangle inequality yields
$W_p(\mu_k,\mu)\leq W_p(\mu_k,\mu_{l'})+W_p(\mu_{l'},\mu)\leq 2\varepsilon$,
which concludes the proof.
\end{proof}

\begin{remark}\label{wassersteindistaceweaktopology}
The limit $\mu$ may not belong to $P_p(X)$. And  $\lim_{k\rightarrow\infty} W_p(\mu_k,\mu)=0$ does not indicate that $\mu_k$ converges to $\mu$ in the forward topology of $P_p(X)$ (i.e., $\lim_{k\rightarrow\infty} W_p(\mu,\mu_k)=0$).
 \end{remark}

\begin{lemma}\label{uniformtopologythesame}
Let $(X,d_\alpha)$, $\alpha=1,2$ be two reversible metric spaces (i.e., $\lambda_{d_\alpha}(X)=1$) such that their metric topologies coincide. Then the uniform topologies of $C([0,1];X)$ induced by $d_1$ and $d_2$ coincide as well.
\end{lemma}
\begin{proof}Set $\rho_\alpha(\gamma_1,\gamma_2):=\max_{0\leq t\leq 1}d_\alpha(\gamma_1,\gamma_2)$, $\alpha=1,2$.
It suffices to show that if a sequence $(\gamma_n)_n\subset C([0,1];X)$ satisfies $\rho_1(\gamma_n,\gamma)\rightarrow0$, then $\rho_2(\gamma_n,\gamma)\rightarrow0$. If not, there exist a (sub)sequence $(\gamma_{n_k})_k$ and a sequence $(t_{n_k})_k\subset [0,1]$ such that $d_2(\gamma_{n_k}(t_{n_k}),\gamma(t_{n_k}))\geq \varepsilon_0$ for some fixed $\varepsilon_0>0$. By passing a subsequence, we may assume that $(t_{n_k})_k$ converges to some $t_0\in [0,1]$. Hence, $\lim_{k\rightarrow \infty}d_\alpha(\gamma(t_0),\gamma(t_{n_k}))=0$ for $\alpha=1,2$.
On the one hand, for large $k$,
\[
d_2(\gamma_{n_k}(t_{n_k}),\gamma(t_0))\geq d_2(\gamma_{n_k}(t_{n_k}),\gamma(t_{n_k}))-d_2(\gamma(t_0),\gamma(t_{n_k}))\geq \varepsilon_0-d_2(\gamma(t_0),\gamma(t_{n_k}))>\varepsilon_0/2.
\]
Hence, $(\gamma_{n_k}(t_{n_k}))_k$ does not converge to $\gamma(t_0)$ under the metric topology. On the other hand,
\[
d_1(\gamma_{n_k}(t_{n_k}),\gamma(t_0))\leq d_1(\gamma_{n_k}(t_{n_k}),\gamma(t_{n_k}))+d_1(\gamma(t_0),\gamma(t_{n_k}))\leq \rho_1(\gamma_{n_k},\gamma)+d_1(\gamma(t_0),\gamma(t_{n_k}))\rightarrow 0.
\]
Thus, $(\gamma_{n_k}(t_{n_k}))_k$ converges to $\gamma(t_0)$ under the same metric topology, which is a contradiction.
\end{proof}

\begin{proof}[Proof of Lemma \ref{compactgeodeisc}]

Let $\widehat{W}_1$ denote the Wasserstein distance on $P(X)$ induced by the symmetrized space $(X,\hat{d})$. Due to Villani \cite[Corollary 6.13]{Vi}, we can assume that the topology on $P(X)$ induced by $\widehat{W}_1$ is the weak topology. Now equip $C([0,1];{P}(X))$ with the uniform topology induced by $\widehat{W}_1$, in which case $\mathfrak{E}:P(C([0,1];X))\rightarrow C([0,1];P(X))$ is continuous (cf.\,Villani \cite[p.136]{Vi}). On the other hand, since
the topologies induced by $d_P$ and $\widehat{W}_1$ on $P(X)$ coincide, Lemma \ref{uniformtopologythesame} implies  the uniform topologies on $C([0,1];P(X))$ induced by these two metrics are the same. Thus, the lemma follows.
 \end{proof}


\section{Qualitative properties of the displacement interpolation
}

\begin{lemma}[Regularizing kernels]\label{Regukern} Let $(X,d,\nu)$ be a forward boundedly compact forward metric-measure space. Let $\mathcal {K}$ be a compact subset of $X$. There exists
a $(\mathcal {K},\nu)$-regularizing kernel $(\mathscr{K}_\epsilon)_{\epsilon>0}$. That is, $(\mathscr{K}_\epsilon)_{\epsilon>0}$ is a family of nonnegative continuous symmetric functions such that for any fixed $\epsilon>0$, there hold

\smallskip

\rm{(i)}  $\int_X \mathscr{K}_\epsilon(x,y)\,\nu({\ddd}y)=1$ for any $x\in \mathcal {K}$;\ \ \ \ \rm{(ii)} $\mathscr{K}_\epsilon(x,y)=0$ if $\min\{d(x,y),d(y,x)\}>\epsilon$.

\end{lemma}
\begin{proof}For any $\epsilon>0$, consider the compact set $\overline{\mathcal {K}^\epsilon}:={\{x\in X:\, d(\mathcal {K},x)\leq \epsilon\}}$. Set $\theta:=\lambda_d(\overline{\mathcal {K}^\epsilon})<\infty$.
 Choose an finite $\epsilon/(2\theta)$-ball-covering of $\mathcal {K}$, say $(B^+_{x_i}(\epsilon/(2\theta)))_{i=1}^{N(\epsilon)}$, such that $x_i\in \mathcal {K}$ for $1\leq i\leq N(\epsilon)$. Thus, $\cup_i B^+_{x_i}(\epsilon/(2\theta))\subset \overline{\mathcal {K}^\epsilon}$.

Now let $(\phi_i)_i$ be a continuous subordinate partition of unity  with $\suppor \phi_i\subset B^+_{x_i}(\epsilon/(2\theta))$, $\suppor\phi_i\cap \mathcal {K}\neq\emptyset$ and $\sum_i\phi_i=1$ on $\mathcal {K}$. We concludes the proof by setting
$\mathscr{K}_\epsilon(x,y):=\sum_i\frac{\phi_i(x)\phi_i(y)}{\int_X\phi_i{\ddd}\nu}$.
\end{proof}

A slight modification of the proof of Villani \cite[Theorem 29.20]{Vi}  provides the following result.
\begin{lemma}\label{continuonUvUpi}
Let $(X,d)$ be a compact forward metric space endowed with a finite measure $\nu$. Then with the notation of Definition \ref{dispconvexfirst}:

\smallskip

\rm{(i)} for every continuous convex function $U:\mathbb{R}_+\rightarrow \mathbb{R}_+$  with $U(0)=0$, $U_\nu(\mu)$ is a weakly lower semi-continuous function of both $\mu$ and $\nu$ in $\mathfrak{M}_+(X)$, where $\mathfrak{M}_+(X)$ denotes the set of finite nonnegative Borel measures on $X$. More precisely, if $(\mu_k)_k$ (resp., $(\nu_k)_k$) converges weakly to $\mu$ (resp., $\nu$), then
\[
U_{\nu}(\mu)\leq \underset{k\rightarrow \infty}{\lim\inf}\,U_{\nu_k}(\mu_k).
\]

\smallskip

\rm{(ii)} for each  continuous convex function $U:\mathbb{R}_+\rightarrow \mathbb{R}_+$  with $U(0)=0$, $U_\nu$ satisfies a contraction principle in both $\mu$ and $\nu$, i.e., if $Y$ is another compact space and   $f:X\rightarrow Y$ is a measurable function, then
\[
U_{f_\sharp \nu}(f_\sharp \mu)\leq U_\nu(\mu).
\]

\smallskip

\rm{(iii)} given a probability measure $\mu\in P(X)$ with $\supp\mu\subset\supp \nu$, there is a sequence $(\mu_k)_k$ of probability measures converging weakly to $\mu$ such that
\begin{itemize}
\item[(1)] each $\mu_k=\rho_k\nu$ has a continuous density $\rho_k$;

\item[(2)] for any sequence $(\pi_k)_k$ converging weakly to $\pi$ in $P(X\times X)$ such that $\pi_k$ admits $\mu_k$ as first marginal   and $\supp \pi_k\subset \supp \nu\times \supp \nu$, there holds
\[
\underset{k\rightarrow \infty}{\lim\sup}\,U^\beta_{\pi_k,\nu}(\mu_k)\leq U^\beta_{\pi,\nu}(\mu),
\]
for every continuous positive function $\beta$ on $X\times X$ and for
every continuous convex function $U:\mathbb{R}_+\rightarrow \mathbb{R}_+$  with $U(0)=0$ and  at most polynomial growth  (see (\ref{mostpoly})).
\end{itemize}
\end{lemma}


\begin{remark}\label{suppregularing}
In view of the proof of Villani \cite[Theroem 29.20/(iii)]{Vi}, the construction of $\mu_k=\rho_k\nu$ in Lemma \ref{continuonUvUpi}/(iii) is independent of both $\beta$ and $U$. In fact, by setting $\mu=\rho\nu+\mu_s$ (i.e., the Lebesgue decomposition), one has (cf. Villani \cite[p.812]{Vi})
\[
\rho_{k}:=\rho^a_{k}+\rho^s_{k},\ \rho^a_{k}(x)=\int_{X}\mathscr{K}_{1/k}(x,y)\rho(y)\nu({\ddd}y),\ \rho^s_{k}(x)=\int_{X}\mathscr{K}_{1/k}(x,y) \mu_s({\ddd}y),
\]
where  $\mathscr{K}_{1/k}$ is the regularizing kernel in Lemma \ref{Regukern}.  In particular, $$\suppor \rho_{k}\subset (\supp \mu)^{1/k}= \{x\in X\,|\, d(\supp \mu_i,x)< 1/k\}.$$
\end{remark}

Although dynamical optimal transference plans and associated displacement interpolations are usually defined on a forward geodesic space, we can extend these notions to compact forward metric spaces.

\smallskip
\begin{definition}\label{metrcdefol}
Let $(X,d)$ be a compact forward metric space and set
 \[
 \Gamma(X):=\left\{\gamma\in C([0,1];X)\,|\,d(\gamma(s),\gamma(t))=(t-s)\,d(\gamma(0),\gamma(1)),\,0\leq s\leq t\leq 1\right\}.
 \]
 If $\Gamma(X)\neq\emptyset$, endow $\Gamma(X)$ by the uniform topology, which becomes a compact space.

 Given $\mu_0,\mu_1\in P(X)$, let $\pi$ denote an optimal transference plan  of coupling $(\mu_0,\mu_1)$ with respect to $d^2$. If a measure $\Pi\in P(\Gamma(X))$ with $(e_0,e_1)_\sharp\Pi=\pi$, then $\Pi$ is called a {\it dynamical optimal transference plan} of $(\mu_0,\mu_1)$ and $\mu_t:=(e_t)_\sharp\Pi$, $0\leq t\leq 1$ is called the  {\it associated displacement interpolation}.
\end{definition}

\begin{lemma}\label{compactmetrc} Let $(X,d)$ be a  forward boundedly compact  forward geodesic  space. If $\mu_0,\mu_1$ are two probability measures with compact support, then
there exist a small positive number $\epsilon>0$ and a compact subset $\mathcal {K}\subset X$ such that
\begin{itemize}

\smallskip

\item[(i)]  for $i=0,1$, we have $ \overline{(\supp \mu_i)^\epsilon}= {\{x\in X\,|\, d(\supp \mu_i,x)\leq \epsilon\}}\subset \mathcal {K}$;

\smallskip

\item[(ii)] all the minimal geodesics from $\overline{ (\supp \mu_0 )^{\epsilon}}$ to $\overline{ (\supp \mu_1 )^{\epsilon}}$ are contained in $\mathcal {K}$;

\smallskip

\item[(iii)] for any optimal transference plan $\pi$ of coupling $(\mu_0,\mu_1)$ associated with the dynamical optimal transference plan $\Pi$, we have $\supp \pi \subset \mathcal {K}\times \mathcal {K}$ and $\supp \Pi\subset \Gamma(\mathcal {K})$,
where $$\Gamma(\mathcal {K}):=\left\{\gamma\in C([0,1];\mathcal {K})\,|\,d|_\mathcal {K}(\gamma(s),\gamma(t))=(t-s)\,d|_\mathcal {K}(\gamma(0),\gamma(1)),\,0\leq s\leq t\leq 1\right\};$$

\item[(iv)] let   $(\mu_t)_{0\leq t\leq 1}$ be an associated displacement interpolation in (iii). Then $\supp \mu_t\subset \mathcal {K}$;

\smallskip

\item[(v)]  the quantities $\pi$, $\Pi$ and $(\mu_t)_{0\leq t\leq 1}$ in (iii)-(iv) can be viewed as an optimal transference plan, associated dynamical optimal transference plan and  associated displacement interpolation defined on the compact forward metric space $(\mathcal {K},d|_\mathcal {K})$;

\smallskip

\item[(vi)] let $\tilde{\pi}$, $\widetilde{\Pi}$ and $(\tilde{\mu}_t)_{0\leq t\leq 1}$ be an optimal transference plan of $(\mu_0|_\mathcal {K},\mu_1|_\mathcal {K})$, the associated dynamical optimal transference plan and the associated displacement interpolation defined on the compact forward metric space $(\mathcal {K},d|_\mathcal {K})$. Thus, under the natural extensions induced by the embedding $i:\mathcal {K}\hookrightarrow X$, they can be viewed as the corresponding (optimal) quantities on $(X,d)$ respectively.
\end{itemize}

\end{lemma}
\begin{proof}
 We consider  a  compact set $\mathcal {B}:=\overline{(\supp \mu_0)^{  D+1 }}$, where $D:=\max_{(x,y)\in\supp \mu_0\times\supp \mu_1}d(x,y)<\infty$. Since
$\theta:=\sup_{x\in \mathcal {B}} \Theta\left( d(\star,x)\right)$ is finite, one can choose a small $\epsilon>0$ (e.g., $\epsilon=1/2$) such that
\begin{itemize}
\item $A_i:=\overline{ (\supp \mu_i)^{\epsilon}}\subset \mathcal {B}$ for $i=0,1$;

\item $\Gamma_{A_0\rightarrow A_1}\subset \mathcal {B}$, where $\Gamma_{A_0\rightarrow A_1}$ is the set of minimal geodesics from $A_0$ to $A_1$.
\end{itemize}
Since $(\mathcal {B}, d|_\mathcal {B})$ is a compact $\theta$-metric space, Theorem \ref{Arzela-Ascoli Theorem} implies that $\Gamma_{A_0\rightarrow A_1}$
 is compact in the uniform topology.
Now set $\mathcal {K}:=\left\{\gamma(t)\,| \,\gamma\in \Gamma_{A_0\rightarrow A_1},\, t\in [0,1]\right\}$, i.e., the image set of $\Gamma_{A_0\rightarrow A_1}$.

 Obviously, $\mathcal {K}$ is compact in $(X,d)$.
Hence, both (i) and (ii) follow. In order to prove (iii) and (iv), choose an optimal transference plan $\pi$ of coupling $(\mu_0,\mu_1)$. It is easy to see
\[
\supp \pi\subset \supp \mu_0\times \supp \mu_1\subset A_1\times A_2.\tag{D.1}\label{imprsupppi}
\]
Since $(e_0,e_1)_\sharp \Pi=\pi$ and $\mu_t=(e_t)_\sharp \Pi$, we obtain
\[
\supp \Pi\subset \Gamma_{A_1\rightarrow A_2}\subset \Gamma(\mathcal {K}),\ \supp \mu_t =\supp \left(\Pi\circ e_t^{-1}\right)\subset \mathcal {K}.
\]
Therefore, (iii) and (iv) are true.

For (v), note that $\Pi$, $\pi$ and $\mu_t$ can be defined on $(\mathcal {K},d|_\mathcal {K})$ due to (iii) and (iv). Since the Wasserstein distance on $P_2(\mathcal {K})$ is always not less than the one on $P_2(X)$, $\pi$ is also an optimal transference plan on $\mathcal {K}$. Thus, (v) follows.

Now we show (vi). Firstly, a contradiction argument together with (v) yields $\tilde{\pi}$ (factually $i_\sharp\tilde{\pi}$) is an optimal  transference plan of $(\mu_0,\mu_1)$. Secondly, for any $\gamma\in \supp\widetilde{\Pi}\subset \Gamma(\mathcal {K})$, we have
\[
d(\gamma(s),\gamma(t))=d|_\mathcal {K}(\gamma(s),\gamma(t))=(t-s)\,d|_\mathcal {K}(\gamma(0),\gamma(1))=(t-s)\,d(\gamma(0),\gamma(1)),
\]
for any $0\leq s\leq t\leq 1$, which implies $\gamma\in \Gamma(X)$. Thus, $\widetilde{\Pi}$ can be defined naturally on $\Gamma(X)$ and hence, it is an dynamical optimal transference plan on $(X,d)$. Therefore, $(\tilde{\mu}_t)_{0\leq t\leq 1}$ is the associated displacement interpolation on $(X,d)$.
\end{proof}





\appendix
\noindent{\textbf{Acknowledgements.}} The first author was supported by the UEFISCDI/CNCS grant PN-III-P4-ID-PCE2020-1001.
The second author was supported by National Natural Science Foundation of China (No. 11761058) and Natural Science Foundation of Shanghai (No. 21ZR1418300, No. 19ZR1411700).

\end{document}